\documentclass[12pt]{article}
\usepackage{graphicx} % Required for inserting images
\usepackage{amsmath,amsthm,amsfonts,amssymb}
\usepackage{xcolor}
\usepackage{tikz-cd}

\usepackage[margin=1in,marginparwidth=0.8in, marginparsep=0.1in]{geometry}

\usepackage[draft]{say}
\definecolor{electricultramarine}{rgb}{0.25, 0.0, 1.0}

\definecolor{airforceblue}{rgb}{0.36, 0.54, 0.66}

\definecolor{nicegreen}{rgb}{0.0, 0.50, 0.0}

\usepackage[hidelinks]{hyperref}

\theoremstyle{plain}
\newtheorem{thm}{Theorem}[section]
\newtheorem{theorem}[thm]{Theorem}
\newtheorem{prop}[thm]{Proposition}
\newtheorem{proposition}[thm]{Proposition}
\newtheorem{proposition-definition}[thm]{Proposition-Definition}

\newtheorem{corollary}[thm]{Corollary}
\newtheorem{lemma}[thm]{Lemma}
\newtheorem{lemma-definition}[thm]{Lemma-Definition}

\theoremstyle{remark}
\newtheorem{rmk}[thm]{Remark}
\newtheorem{remark}[thm]{Remark}

\theoremstyle{definition}
\newtheorem{defin}[thm]{Definition}
\newtheorem{definition}[thm]{Definition}

\newtheorem{example}[thm]{Example}

\newtheorem{construction}[thm]{Construction}

\newcommand\C{\mathbb{C}}

\newcommand\G{\mathbb{G}}

\newcommand\N{\mathbb{N}}

\renewcommand\P{\mathbb{P}}
\newcommand\Q{\mathbb{Q}}
\newcommand\R{\mathbb{R}}

\newcommand\Z{\mathbb{Z}}
 
\newcommand\cA{\mathcal{A}}
\newcommand\cB{\mathcal{B}}
\newcommand\cC{\mathcal{C}}
\newcommand\cD{\mathcal{D}}
\newcommand\cE{\mathcal{E}}
\newcommand\cF{\mathcal{F}}
\newcommand\cG{\mathcal{G}}
\newcommand\cH{\mathcal{H}}

\newcommand\cM{\mathcal{M}}
\newcommand\cN{\mathcal{N}}
\newcommand\cO{\mathcal{O}}
\newcommand\cP{\mathcal{P}}

\newcommand\cR{\mathcal{R}}
\newcommand\cS{\mathcal{S}}

\newcommand\cU{\mathcal{U}}

\newcommand\cY{\mathcal{Y}}

\newcommand\base{B}
\newcommand\La{\Lambda}
\renewcommand\flat{\mathcal{U}}
 
\newcommand\Sh{\operatorname{Sh}}
\newcommand\Mod{\operatorname{Mod}}
\newcommand\cMod{\operatorname{cMod}}
\newcommand\End{\operatorname{End}}
\newcommand\Aut{\operatorname{Aut}}
\newcommand\Hom{\operatorname{Hom}}
\newcommand\Ob{\operatorname{Ob}}
\newcommand\Nov{\operatorname{Nov}}
\newcommand\Per{\operatorname{Per}}
\newcommand\sh{\operatorname{sh}}
\newcommand\Fuk{\operatorname{Fuk}}
\newcommand\cFuk{\operatorname{cFuk}}
\newcommand\LFuk{\operatorname{Fuk}}
\newcommand\cLFuk{\operatorname{cFuk}}
\newcommand\WFuk{\operatorname{WFuk}}
\newcommand\RFuk{\operatorname{RFuk}}

\newcommand\colim{\operatorname{colim}}

\newcommand\Prop{\operatorname{Mod^{fd}}}

\newcommand\id{\operatorname{id}}
\newcommand\Gr{\operatorname{Gr}}
\newcommand\Reeb{R_{\alpha^{\circ}}}
\newcommand\Ch{\operatorname{Ch}}
\newcommand\cCh{\operatorname{cCh}}
\newcommand\cTw{\operatorname{cTw}}
\newcommand\Tw{\operatorname{Tw}}
\newcommand\qAinf{\operatorname{qA_{\infty}}}
\newcommand\cAinf{\operatorname{cA_{\infty}}}
\newcommand\Ainf{\operatorname{A_{\infty}}}
\newcommand\Cone{\operatorname{Cone}}
\newcommand\cont{\operatorname{cont}}
\newcommand\unit{\operatorname{unit}}
\newcommand\inp{\operatorname{in}}
\newcommand\outp{\operatorname{out}}

\newcommand\Crit{\operatorname{Crit}}
\newcommand\module{\mathfrak{N}}

\title{On Fukaya categories and prequantization bundles}
\author{Tatsuki Kuwagaki, Adrian Petr, Vivek Shende}

\begin{document}
 
\maketitle

\begin{abstract}
    We show: the Floer homology over the Novikov ring of (nonexact!) rational Lagrangians in an (nonexact!) integral symplectic manifold can be computed in terms of exact Lagrangians in an exact filling of the prequantization bundle. 

    As a consequence, we give a Fukaya-sheaf correspondence for rational (nonexact!) Lagrangians in Weinstein manifolds, as conjectured by Ike and the first-named author.  We also show that bounding cochains for immersed rational Lagrangians transform naturally under Legendrian isotopy, as conjectured by Akaho and Joyce.

    As an illustration, we show that quantum cohomology of the complex projective line -- which requires the counting of one holomorphic sphere -- can be recovered from purely sheaf-theoretic calculations. 
\end{abstract}

\thispagestyle{empty}

\newpage 
\tableofcontents

\newpage

\section{Introduction}
The symplectic geometer studying 
Lagrangian Floer theory in (noncompact, exact) Weinstein symplectic manifolds has many tools ready to hand:  Lefschetz fibrations \cite{Sei08}; surgery methods \cite{BEE12, EL21}; and  gluing / microsheaf methods \cite{GPS3, GPS2, GPS20}.  Indeed, the celebrated calculations of Fukaya categories of compact symplectic manifolds typically proceed by applying said methods to the (Weinstein) complement of a sufficiently positive divisor, then studying how the result is deformed by disks passing through the divisor \cite{Seidel-genustwo, Seidel-quartic,  Sheridan-CY, Sheridan-Fano, GHHPS}.  

In the present article, we develop a new method to reduce calculations of Lagrangian Floer theory in compact symplectic manifolds to calculations in Weinstein manifolds. Our setup is the following. Let $V$ be a contact manifold, equipped both with a Weinstein filling $W$, and the structure of a prequantization bundle over some symplectic manifold $\base$:  
\begin{equation} \label{setup} \partial_\infty W = V \qquad \qquad V \xrightarrow{\pi} \base. \end{equation}
That is, $\base$ is the symplectic reduction of $W$ along a contact level.  
In the spirit of ``quantization commutes with reduction'' \cite{Guillemin-Sternberg-quantization-reduction}, we will show that Lagrangian Floer theory in $\base$ can be computed in terms of Lagrangian Floer theory in $W$, at least under certain monotonicity hypotheses on $\base$.

As our general results in this direction are nontrivial even to state, let us first mention two consequences.  
The first concerns a sheaf-Fukaya correspondence for {\em non-exact} Lagrangians in cotangent bundles and, more generally, Weinstein manifolds.  Recall that the wrapped Fukaya category of a Weinstein manifold was shown in \cite{GPS3} to be equivalent to the category of microsheaves on the skeleton.  It was explained in \cite{shende-hitchinfiber} that nonexact Lagrangians can be incorporated into this correspondence if one works over the Novikov {\em field}; doing so, however, loses important information, see Remark \ref{rem: hitchin} below for a compelling example. 
Meanwhile, it has been understood that the Novikov {\em ring}  naturally appears in sheaf theory as the endomorphisms of the convolution monoidal unit for the $\R^\delta$-equivariant positively microsupported sheaves on $\R$ \cite{Kuwagaki-WKB,Kuwagaki-Almost}.  This identification leads to a notion of sheaves with possibly non-conic, nonexact microsupport.  It was shown in \cite{Ike-Kuwagaki-Novikov} that sheaves with (in this new sense) Lagrangian microsupport are characterized by an $A_\infty$ structure structurally similar to the Fukaya category, and conjectured that in fact there is a sheaf/Fukaya correspondence in this setting.  
(Let us mention that these works develop along a line of inquiry with roots in \cite{Tamarkin, Viterbo}.) 

From our general results, we will obtain a resolution of this conjecture for compact rational Lagrangians: 

\begin{theorem} 
\label{intro-1 thm: R-equivariant sheaves comparison} (Theorem \ref{intro thm: R-equivariant sheaves comparison})
    Let $W$ be a Weinstein manifold with $c_1(W)=0$ and let $L_1,..., L_n$ be compact, mutually rational, transverse Lagrangians in $W$. Then the Novikov ring  Fukaya category on objects $L_1,...,L_n$ is almost fully faithfully embedded into category of nonconic microsheaves on $W$.  The image is characterized as the objects microsupported on the $L_i$ with finite rank microstalks. 
\end{theorem}

Our results also allow some sheaf-theoretic calculations in non-exact symplectic manifolds.  Here we highlight one remarkable illustration: 

\begin{theorem}
    The quantum cohomology of $\C \P^1$ (in other words, the self Floer cohomology of the diagonal in $\C \P^1\times \C \P^1$) can be recovered from a purely sheaf-theoretic computation.
\end{theorem}

We now return to the general context of \eqref{setup} to discuss our constructions.   
Fix $\La \subset V$ a (possibly disconnected) compact Legendrian; consider its image $\pi_{| \La} \subset B$.  We are  interested in the category $\Fuk(\pi_{| \La})$, whose objects are Lagrangians (with local systems, bounding cochains, etc.) supported in $\pi_{| \La}$, as defined  by \cite{FOOO} in case $\pi$ restricts to an embedding on each component of $\Lambda$, and by \cite{AJ10} in general.  

Another structure of interest is the Chekanov-Eliashberg differential graded algebra $\cA_\La$ \cite{Eli98, Chekanov-dga, EES05}.  $\cA_\La$ enjoys two (related) relations to Lagrangian Floer theory in the filling $W$:  to  infinitesimally wrapped  Floer homology for Lagrangians ending on $\La$  \cite{EGH00, Ekh08}, and  to partially wrapped Floer theory with stop at $\La$ \cite{EL21}.  
It is the latter which is more directly of interest to us, and correspondingly we take $\cA_\La$ to have coefficients in the based loop space of $\La$.   
We denote by $\Prop(\cA_\La)$ the category of finite-dimensional modules over $\cA_\La$.  

When $V$ is a jet bundle (rather than a prequantization bundle), then $\cA_\La$ concerns only positive Reeb chords while $\Fuk(\pi_{| \La})$ involves both positive and negative Reeb chords (see e.g. last paragraph of \cite{AJ10}). It is well known to experts that this can be elaborated to a map $\Prop(\cA_\La) \to \Fuk(\pi_{| \La})$, essentially by taking the quotient by the negative Reeb chords, which however is very far from being full and faithful. In fact, already when $\dim \base = 2$ and $\base$ is exact, said map has a very rich structure, sufficient to encode the Legendrian skein relations \cite{Haiden-skein}. 
Here, however, we are interested in going the reverse direction, and recovering $\Fuk(\pi_{| \La})$ from contact geometry in the prequantization bundle $V$. 

We will {\em not} attempt to recover $\Fuk(\pi_{| \La})$ directly from $\Prop(\cA_\La)$ alone. 
Instead, we proceed as follows. The Reeb flow for the prequantization bundle contact form generates an action of $\R / \Z$ on $V$.
Note that $\frac{1}{n} \Z \cdot \La$ is the union of $n$ copies of $\La$, and there is an inclusion $\frac{1}{m} \Z \cdot \La \subset \frac{1}{n} \Z \cdot \La$ when $m | n$, and correspondingly a morphism $\Prop(\cA_{\frac{1}{m} \Z \cdot \La}) \to \Prop(\cA_{\frac{1}{n} \Z \cdot \La})$.

We will write $\frac{1}{n} \vec \Z$ for the symmetric monoidal category whose objects are elements of $\frac{1}{n} \vec \Z$, monoidal structure is from addition, and morphisms $a \to b$ if $a < b$.  
There are evident maps 
$\frac{1}{m} \vec \Z \to \frac{1}{n} \vec \Z$ when $m | n$.    
We will say a category has an action of $\frac{1}{n} \vec \Z  / \Z$ if it has an action of $\frac{1}{n} \vec \Z$ together with a trivialization of the action of the (not full) subcategory $\Z \subset  \frac{1}{n} \vec \Z$.  We write similarly $\vec \Q$ and $\vec \Q / \Z$. 

\begin{theorem}\label{intro thm: Q mod Z structure}
    Assume that the action spectrum of $\La$ is contained in $\Z \cup (\R \setminus \Q)$.
    Then the Reeb flow contactomorphisms and continuation maps determine an action of $ \frac{1}{n} \vec \Z / \Z$ on $\Prop(\cA_{\frac{1}{n} \Z \cdot \La})$, compatibly with the morphisms $\Prop(\cA_{\frac{1}{m} \Z \cdot \La}) \to \Prop(\cA_{\frac{1}{n} \Z \cdot \La})$ when $m | n$.
\end{theorem}

In general, given a sequence of categories as above, the limit, in this case $\varinjlim  \Prop(\cA_{\frac{1}{n} \Z \cdot \La})$, has a $\vec \Q / \Z$-action. A $\vec \Q / \Z$-category $\mathcal{C}$ can be subjected to an orbit category construction to produce a $\Z[\Q_{\ge 0}]$-linear category $\mathcal{C}[\vec \Q / \Z]$ (see Lemma \ref{lem: pog-quotient orbit category}).  From the density of $\Q \subset \R$, we obtain a completion process relating $\Z[\Q_{\ge 0}]$-modules to $\Z[\R_{\ge 0}]$-modules (Lemma \ref{lemma from Q to R}).     
These algebraic ingredients
allow us to formulate our fundamental comparison result: 

\begin{theorem} \label{intro thm: fukaya from augmentations}
    Assume that $\base$ is monotone with minimal Chern number different from $1$, and that $\pi_{| \La}$ has double and transverse multiple points. 
    Assume moreover that the action spectrum of $\La$ is contained in $\Z \cup (\R \setminus \Q)$.
    Then there is an $\Ainf$-equivalence
    $$\Fuk(\pi_{| \La}) \simeq \Nov \otimes_{\Z[\R_{\ge 0}]}\overline{(\varinjlim  \Prop(\cA_{\frac{1}{n} \Z \cdot \La}))[\vec \Q/\Z]}$$
    where $\Nov$ is the Novikov ring.
\end{theorem}

\begin{remark}
    The assumption on $\base$ allows us to define $\cA_\La$ without having to consider discs with interior punctures asymptotic to contractible Reeb orbits or, in other words, using the trivial algebra map as an augmentation of the orbit algebra. See the discussion in the beginning of Section \ref{section augmentation category}.
\end{remark}

The proof has two remaining steps. 
We show in Section \ref{section precompleted Fukaya category} that $\La \subset V$ (i.e. the choice of a Legendrian lift of the Lagrangian immersion $\pi_{|\La}$) allows us to define a sequence of $\frac{1}{n} \vec \Z /\Z$ categories $\cC_n$ such that $\Fuk(\pi_{| \La}) \cong \Nov \otimes_{\Z[\R_{\ge 0}]} \overline{(\varinjlim  \cC_n)[\vec \Q/\Z]}$.
Finally, 
in Section \ref{comparison section}, we construct a functorial equivalence of each term: 
$\cC_n \simeq \Prop(\cA_{\frac{1}{n} \Z \cdot \La})$. 

\begin{remark}
    We prove versions of the above statements at the level of curved $\Ainf$-categories. 
    This allows us to prove \cite[Conjecture 13.19]{AJ10} of Akaho-Joyce in the monotone case.
    Moreover, the result simplifies greatly when $\pi_{| \La}$ is an embedding.
    See Theorem \ref{thm cFuk=cAug} and Corollaries thereafter.
\end{remark}

Let $\mu sh(W, \La)$ be the category of microsheaves on the relative skeleton, as constructed in \cite{KashiwaraSchapira, Shende-microlocal, Nadler-Shende}, and let $\mu_{\La}$ be the direct sum of microstalk functors at all components of the given Legendrian. 
Using the correspondence theorems (1) between Chekanov-Eliashberg modules and partially wrapped Fukaya categories \cite{EL21} and (2) between partially wrapped Fukaya categories and microsheaves \cite{GPS3}, we deduce:     

\begin{corollary}\label{intro corollary: fukaya from microsheaves}
    Assume that $W$ is topologically simple, i.e. $c_1(TW)=0$ and the map $\pi_1(V) \to \pi_1(W)$ is injective  (the latter always holds if $\dim W \geq 6)$.  
    Then there is an $\Ainf$-equivalence 
    $$\Fuk(\pi_{| \La}) \simeq \Nov \otimes_{\Z[\R_{\ge 0}]} \overline{(\varinjlim \Prop (\End_{\mu sh(W, \frac{1}{n} \Z /\Z  \cdot \La)}(\mu_{\frac{1}{n} \Z /\Z  \cdot \La})))[\vec \Q/\Z]}.$$
\end{corollary}
While the right hand side is rather involved, we emphasize that it is defined entirely from sheaf theory.  Corollary \ref{intro corollary: fukaya from microsheaves} is proven in 
Section \ref{proofs of intro theorems}. 

In case $\base$ is also a Weinstein manifold and hence admits its own microsheaf category (in particular, if $\base = T^*M$ is a cotangent bundle), we derive Theorem \ref{intro-1 thm: R-equivariant sheaves comparison} by establishing a version of Theorem \ref{intro thm: fukaya from augmentations} in sheaf theory and using  the comparison theorem of \cite{GPS3}.

\begin{remark}[Hitchin fibration] \label{rem: hitchin}
    Let $\pi\colon \cM\rightarrow B$ be a Hitchin fibration, which is a holomorphic integrable system. Then the base space has an affine manifold structure. Over each rational point $p$ the fiber $\pi^{-1}(p)$ is a rational Lagrangian.  Because said Lagrangian is holomorphic,  it is unobstructed for a generic complex structure among those coming from the hyperk\"ahler structure of $\cM$ \cite{solomon-verbitsky}.
    
    In \cite{shende-hitchinfiber}, it was explained that, despite the nonexactness of Hitchin fibers, the results of \cite{GPS2, GPS3} could nevertheless be used to construct a corresponding microsheaf on the skeleton of the moduli of Higgs bundles, defined over the Novikov field. Here we obtain a much stronger result: a sheaf quantization of the fiber {\em over the Novikov ring}.  
    This additional strength has the following significance.  One expects the microsheaf described in \cite{shende-hitchinfiber} to be (the restriction of) a Hecke eigensheaf, essentially because Hitchin fibers are eigen-sets for the singular supports of the Hecke transforms.  But it is not clear how to use this eigen-set property either in the Fukaya category (because the microlocalization of the Hecke correspondence is not smooth) or for the Novikov field microsheaf (because passing to the field inverts all Hamiltonian isotopies, and in particular forgets the precise geometry of the Hitchin fiber).  By contrast,  sheaf quantization over the Novikov ring suffers from neither of these disabilities. 
\end{remark}

{\bf Acknowledgments.}
We learned from Tobias Ekholm, who attributes it to Dominic Joyce, the idea that Lagrangian Floer homology should be some kind of limit `as $\hbar \to 0$' of the Legendrian Floer homology on the prequantization bundle.   We thank Noémie Legout for helpful discussions. We also would like to thank Wenyuan Li for pointing out some subtleties around Tamarkin categories for Weinstein manifolds.

V.S. and A. P. are supported by  Villum Fonden Villum Investigator grant 37814, Novo Nordisk Foundation grant NNF20OC0066298, and Danish National Research Foundation grant DNRF157. T.~K.\ is supported by JSPS KAKENHI Grant Numbers JP22K13912, 23H01068, and JP20H01794.

\section{Partially ordered groups}
\label{section pogs}

A partially ordered group (pog) is a group $G$ equipped with a partial order preserved by left multiplication $(b \le c \implies ab \le ac)$.  
We will often denote partially ordered sets or groups with an overhead arrow: $\vec G$.  
These have appeared in the literature at least since  \cite{Cartan-pog, Clifford-pog}, and have previously been of
use in contact geometry \cite{Eliashberg-Polterovich-pog}. There is a textbook:  \cite{Glass-pogbook}. 

A partial order on a group is  characterized by the {\em cone of positive elements}:
$$ G_+ := \{ g\,|\, 1_G \le g\}$$
Note $ G_+$ is a monoid in which only the unit is invertible.  Conversely, given a group $G$ and such monoid $G_+ \subset G$, we may define a partial order on $G$ 
by taking $a \le b$ if $a^{-1} b \in G_+$.  

\begin{example}
    The addition structure on $\vec \R = (\R, \le)$ defines a pog, for which $\vec \R_+$ is the non-negative real numbers. 
    For our purposes, we are interested in $(\vec \R, +)$ and its sub-pogs.  
\end{example}

\begin{example}
    A non-abelian example: if $\cB_n$ is the braid group on $n$ strands, we may use the positive braid monoid $\cP_n$ to define a partial order on $\cB_n$, where $a \le b$ if one can write $b = ap$ with $p$ a positive braid.
\end{example}

Just as we may regard a poset $\vec P$ as a category with a unique morphism $p \to q$ if $p \le q$, we may regard a pog as a monoidal category. 
In all applications we are interested in the pog $\vec{\R}$
consisting of the real numbers with order $a \le b$ and groups structure addition, and its sub-pogs $\vec{\Q}$, $\vec{\Z}$, etc.

We throughout this section take coefficients in $\Z$, but the results hold with $\Z$ replaced by an arbitrary commutative ring $k$.

\subsection{Modules over posets and pogs}

For a poset $\vec P$, we write $\vec{P}-mod$ for the category of functors $\Hom(\vec P, \Z-mod)$  (either for the abelian or dg derived category of $\Z$-modules, as appropriate), 
and similarly by $mod-\vec{P}$ we mean $\Hom(\vec{P}^{op}, \Z-mod)$.  Of course $\Z-mod$ can be replaced by another appropriate category to consider modules valued in said category.

We write $\Z[\vec P]:=\bigoplus_{p\leq q}\Z(p\rightarrow q)$ for the ring generated by the morphisms of $P$, subject to the relation that multiplication of composable morphisms give the corresponding morphism, and that multiplication of not-composable morphisms gives zero.  There is an  equivalence between $\vec P-mod$ and modules for $\Z[\vec P]$ with the property $M = \prod_{p \in P} 1_p M$,
given by sending a functor $F$ to $\prod_{p \in P} F(p)$. 

Suppose now $\vec P$ is in fact a pog.
The monoid-ring $\Z[ P_+]$ admits a homomorphism
\begin{eqnarray*}
    \Z[ P_+]  & \to & \Aut(\id_{\Z[\vec P]})\cong \Aut(\id_{\vec P-mod}) \\
    \rho & \mapsto &  \prod_{p\in P}\left( \rho\colon 1_pM\rightarrow 1_{\rho+p}M \right)
\end{eqnarray*}
where $\id_{\Z[\vec P]}$ is the identity functor of the category of $\Z[P_+]$-modules of the form $M=\prod_{p\in P}1_p M$.
Here, $\rho$ is some map $\rho: 0 \to q$, and so $p+\rho$ is the corresponding map $p \to p+q$.  
We view $\Z[P_+]$ as a $P$-graded ring.

\begin{lemma} \label{lem: poset module}
    Let $\vec P$ be a pog.  Then
    there is an equivalence of categories between $\vec P-mod$
    and $P$-graded $\Z[ P_+]$ modules, given on objects by $F \mapsto \prod_{p \in P} F(p)$. 
\end{lemma}

Note that the definition of $\Hom(\vec P, \Z-\Mod)$ above does not involve the group structure, the role of which in the above statement is to ensure that morphisms out of any given element of $P$ is `independent of the element'.

\begin{corollary} \label{cor: poset equivariance}
  Let $\vec P$ be a pog.  Then there is an action of the group $P$ on $\vec P-mod$, and the isomorphism of
  Lemma \ref{lem: poset module} descends to an equivalence of
  $\vec P-mod^{P}$ with (not $P$-graded) $\Z[P_+]$ modules. 
\end{corollary}

More generally:

\begin{corollary} \label{cor: sub-poset equivariance}
  Let $\vec P$ be a pog, 
  and $P_0 \subset P$ a subgroup.  
  Then there is an action of the group $P_0$ on $\vec P-mod$, and the isomorphism of
  Lemma \ref{lem: poset module} descends to an equivalence of
  $\vec P-mod^{P_0}$ with $P/P_0$-graded $\Z[P_+]$ modules, where $P_+$ carries the induced $P/P_0$ grading.  We denote this category by $\vec P/P_0-mod$. 
\end{corollary}

We give the various module categories appearing in Lemma \ref{lem: poset module} and Corollaries \ref{cor: poset equivariance} and \ref{cor: sub-poset equivariance} the natural monoidal structures which are $\otimes_\Z$ on the underlying $\Z$-modules.  

Consider now $\vec P \subset \vec Q$ an inclusion of pogs.  The restriction of modules 
$\vec Q-mod \to \vec P-mod$ is monoidal.  Across Lemma \ref{lem: poset module}, the corresponding map from $Q$-graded $\Z[Q_+]$-modules to $P$-graded $\Z[P_+]$-modules is given as follows: take the $P\subset Q$-graded components, on which a $\Z[P_+]$ action remains well defined.  Fixing further a normal subgroup $P_0 \subset P$, the analogous statements apply to restriction of 
modules $\vec Q/P_0-mod \to \vec P/P_0 -mod$.

\begin{remark}
Let us describe some other functors 
$\vec Q-mod \to \vec P-mod$ which are not monoidal.
We say as inclusion $\vec P \subset \vec Q$ of posets  is {\em discrete} if it has left and right adjoints, which we denote 
$\lceil \cdot \rceil^{\vec P}$ and $\lfloor \cdot \rfloor_{\vec P}$, respectively.  These agree with the usual ceiling and floor functions in the case $\vec \Z \subset \vec \R$.  In this case, we could also define functors $\vec Q-mod$ to $\vec P-mod$
by extension of scalars ``$\otimes_{\vec Q} \vec P$''
along either of $\lceil \cdot \rceil^{\vec P}$ or $\lfloor \cdot \rfloor_{\vec P}$.  Across Lemma \ref{lem: poset module}, this has the effect of collapsing the $Q$ grading to a $P$ grading via $\lceil \cdot \rceil^{\vec P}$ or $\lfloor \cdot \rfloor_{\vec P}$, and restricting along  $\Z[P_+] \to \Z[Q_+]$. 

These are {\em not} monoidal: collapsing the grading by floor or ceiling function does not commute with tensor product, since e.g. $\lfloor \frac{1}{2} + \frac{1}{2} \rfloor \ne \lfloor \frac{1}{2} \rfloor + \lfloor \frac{1}{2} \rfloor$.
\end{remark}

\subsection{Orbit categories for pogs} \label{section: orbit categories}

We recall some constructions of orbit categories from \cite{Cibils-Marcos}; see also \cite{Bongartz-Gabriel, Gabriel-cover, Keller-orbit, Asashiba}.  

\begin{definition} \label{def: orbit category}
    Let $\mathcal{C}$ be a $\Z$-linear category, with a group action $G \to \Aut(\cC)$.  One writes $\mathcal{C}[G]$
    whose objects are the same as $\mathcal{C}$, and whose morphisms are $G$-graded: 
    $$\Hom_{\mathcal{C}[G]}(x, y) := \bigoplus_{g \in G} \Hom(x, gy)$$
    This has a composition defined via 
    $$\Hom(x, gy) \times \Hom(y, hz) = 
    \Hom(x, gy) \times \Hom(gy, ghz) \to 
    \Hom(x, ghz)$$
\end{definition}

Note that if $x = gy$ in $\cC$, then $x$ and $y$ become isomorphic in $\cC[G]$.
It is shown in \cite{Cibils-Marcos} that
$\cC[G]$ agrees with various historical definitions
of orbit category, when those made sense.

\begin{example} 
    Let $G$ be a group. Consider the trivial action $G\rightarrow \Aut(B\Z)$. Then $(B\Z)[G] = B(\Z[G])$. 
\end{example}

By construction, the Hom spaces of $\cC[G]$ are $G$-graded.  

\begin{definition}  \label{def: unorbit} 
    Let $\cD$ be any $\Z$-linear category with $G$-graded Hom spaces.  We write $\cD \# G$ for the category whose objects are pairs $(d, g)$ for $d \in \cD$ and $g \in G$, and such that 
    $\Hom((d, g), (c, h))$ is the $g^{-1} h$-graded piece of $\Hom(d, c)$.    
\end{definition}

The category $\cD \# G$ has an evident $G$ action, which acts on objects through the second factor.  
It is easy to see (and shown in \cite{Cibils-Marcos}) that $\cC \mapsto \cC[G]$ and $\cD \mapsto \cD\#G$ are inverse operations.  E.g. $B\Z[G] \# G$ is the category with $G$ objects where every Hom space is $\Z$, with the natural compositions.

\vspace{2mm}

The constructions admit the following additional structures for pogs.  

\begin{lemma-definition} \label{lem: pog orbit category}
    Let $\vec P$ be a pog acting on a category, $\vec P \to \Aut(\cC)$.  Then 
    $\cC[P]$ is enriched over $P$-graded $\Z[P_+]$ modules. ($P_+$ carries the tautological $P$ grading.)

    We denote the resulting structure as $\cC[\vec P]$. 
\end{lemma-definition}
\begin{proof}
    Given $p \in P$, $q \in \Z[P_+]$ and $h \in \Hom_{\cC}(x, py)$, we should exhibit
    some map 
    $\Hom_{\cC}(x, py) \to \Hom_{\cC}(x, qpy)$.  The map is
    the composition of $h$ with the map 
    $py \to qpy$ obtained from the map $1 \to q$ in $\vec P$.  
\end{proof}

\begin{lemma-definition} \label{lem: pog unorbit}
    Let $\vec P$ be a pog and let $\cD$
    be a category enriched over $P$-graded $\Z[P_+]$ modules.  Then the action
    $P \to \Aut(\cD\#P)$ naturally extends to 
    $\vec P \to \Aut(\cD\#P)$. 
    We denote the resulting structure as $\cD \# \vec P$.
\end{lemma-definition}
\begin{proof}
Given $(p \to q) \in \vec P$ and $d \in \cD$, we should exhibit a morphism $(d, p) \to (d, q)$ in $\cD \# P$, which is to say, an element of the $p^{-1} q$-graded piece of $\Hom_{\cD}(d, d)$.  Now if there is a map $p \to q$ in $\vec P$, then there's also a map $1 \to p^{-1} q$, which we may identify as an element $\tau$ of $\Z[P_+]$ in degree $p^{-1} q$.  Recalling that $\Hom_{\cD}(d, d)$ is by hypothesis a $P$-graded $\Z[P_+]$-module, the desired element is $\tau \cdot 1_d$. 
\end{proof}

The constructions $\cC \mapsto \cC[\vec P]$ and $\cD \mapsto \cD\#\vec P$ are inverse. 

We will need the following variants.

\begin{lemma-definition} \label{lem: pog-quotient orbit category}
    Let $\vec P$ be a pog acting on a category, $\vec P \to \Aut(\cC)$.  
    Assume $P_0 \subset P$ is the kernel of the underlying group action.   
    Then 
    $\cC[P/P_0]$ is enriched over $P/P_0$-graded $\Z[P_+]$ modules, where $P_+$ carries the grading induced from $P$.        
        We denote the resulting structure as $\cC[\vec P/P_0]$.
\end{lemma-definition}
\begin{proof}
    Observe that $\cC[P/P_0]$ is obtained from
    $\cC[P]$ by taking $P_0$ fixed points with respect to the natural $P_0$ action on Hom spaces, which collapses the $P$ grading to a $P/P_0$ grading.  Via this, we may inherit the $\Z[P_+]$ action from from Lemma \ref{lem: pog orbit category}. 
\end{proof}

\begin{lemma-definition} \label{lem: pog-quotient unorbit}
    Let $\vec P$ be a pog, $P_0 \subset P$
    a normal subgroup.  Let $\cD$ be a category enriched over $P/P_0$-graded $\Z[P_+]$ modules. 
    Then the $P/P_0$ action on $\cD\# P/P_0$ extends
    to a $\vec P$ action (with $P_0$ the kernel of the underlying group action). 
    We denote the resulting structure as $\cD \# \vec{P} / P_0$. 
\end{lemma-definition}
\begin{proof}
    Identical to the proof of Lemma \ref{lem: pog unorbit}, save that $P$-gradings are replaced with $P/P_0$-gradings. 
\end{proof}

The constructions $\cC \mapsto \cC[\vec P/P_0]$ and $\cD \mapsto \cD\#\vec P/P_0$ are inverse. 

\begin{rmk}\label{rmk structure on opposite category}
    If $\cC$ is a $(\vec{P}/P_0)$-category, i.e. a linear category together with a monoidal functor $\vec P \to \Aut (\cC), a \mapsto \tau_a$ sending any object in $P_0$ to $\id_{\cC}$, then $\cC^{op}$ and $\Mod(\cC)$ are $(\vec{P}/P_0)$-categories via
    \[\left( \vec{P}/P_0 \to \Aut(\cC^{op}), \, a \mapsto \tau_{a^{-1}} \right) \text{ and } \left( \vec{P}/P_0 \to \Aut(\Mod(\cC)), \, a \mapsto (- \circ \tau_{a^{-1}}) \right). \] 
    Moreover, the Yoneda embedding $\cC \hookrightarrow \Mod(\cC)$ is $(\vec{P}/P_0)$-equivariant.
\end{rmk}

\begin{rmk}\label{rmk filtration}
    If $\cC$ is a $(\vec{P}/P_0)$-category, then $\cC$ is enriched over $P_+$-filtered $\Z[P_+ \cap P_0]$-modules, with
    \[\cF^{\geq c} \cC (X, Y) := T_{c^{-1}, 1_P} \cdot \cC (X, \tau_{c^{-1}} Y) \]
    where $T_{a,b} : \cC(- , \tau_a(-)) \to \cC(- , \tau_b(-))$ is the natural transformation associated to the morphism in $\hom_{\vec P} (a, b)$ corresponding to the relation $a \leq b$.
\end{rmk}

\subsection{Change of enrichment and reconstruction}
\label{section change of enrichement and reconstruction}

Suppose given an inclusion of pogs $\vec P \to \vec Q$.  We have the restriction of modules map $\vec Q-mod \to \vec P-mod$; recall it is monoidal.  Thus if $\cD$ is a category enriched over $\vec Q-mod$, we may form the change of enrichment $\cD|_{\vec P}$ by applying the restriction of modules to all morphism spaces. 

\begin{lemma}\label{lemma change of enrichment}
    There is a natural functor $\cD|_{\vec P}\# \vec{P} \to \cD \# \vec Q$ commuting with the $\vec P$ actions on both sides. 
\end{lemma}

Suppose now that
$\vec P_1 \subset \vec P_2 \subset \cdots \subset \vec Q$ is a sequence of sub-pogs such that 
$\varinjlim \vec P_i = \vec Q$. We say the sequence $\vec P_i$ {\em exhausts} $\vec Q$. 

\begin{example}
    $\frac{1}{n} \vec \Z$ exhausts $\vec \Q$. 
\end{example}

Recall that by definition $\vec Q-mod = \Hom(\vec Q, \Z-mod)$. It is evident that $\varprojlim \vec P_i-mod = \vec Q-mod$.

\begin{proposition}
    Let $\vec Q$ be a pog, and $\vec P_i \subset \vec Q$ an exhaustion. 
    Suppose given $\cD$ a category enriched over $\vec Q-mod$. 
    Then 
    $$\varinjlim \cD|_{\vec P_i}\# \vec{P_i} \xrightarrow{\sim} \cD \# \vec Q$$
\end{proposition}
\begin{proof}
    Obvious from the explicit constructions.  
\end{proof}

\begin{corollary}
    $(\varinjlim \cD|_{\vec P_i}\# \vec{P_i})[\vec Q] \xrightarrow{\sim} \cD $
\end{corollary}

More generally,

\begin{proposition}\label{prop reconstruction}
Let $\vec Q$ be a pog, and $\vec P_i \subset \vec Q$ an exhaustion. 
Suppose all the $\vec P_i$ contain some normal subgroup $P_0$. 
Given $\cD$, a category enriched over $\vec Q/P_0-mod$: 
$$\varinjlim \cD|_{\vec P_i/P_0}\# \vec{P_i}/P_0 \xrightarrow{\sim} \cD \# \vec Q/P_0 \text{ and } (\varinjlim \cD|_{\vec P_i/P_0}\# \vec{P_i}/P_0)[\vec Q/P_0] \xrightarrow{\sim} \cD$$.
\end{proposition}

\subsection{Dense inclusion of pogs}\label{section from Q to R}

We say that an inclusion $\vec Q \to \vec R$ of pogs is {\em limit-dense} if the inclusion preserves filtered limits and every object in $R$ is a filtered limit of objects in $Q$.  For example, the inclusion $\vec \Q \to \vec \R$ is limit-dense  (it is also colimit-dense).

We write $\Per(\vec P) \subset \Mod(\vec P)$ for the filtered-colimit preserving functors $\vec P ^{op} \to \Mod(\Z)$.  We term these `persistence modules'. 

\begin{lemma}\label{lemma from Q to R}
    Let $\vec Q \to \vec R$ be a limit-dense inclusion of pogs.
    The restriction functor $\Mod(\vec R^{op}) \to \Mod(\vec Q^{op})$ has a right adjoint $G \mapsto \overline{G}$ where, for $a \in R$, $\overline{G} (a) = \underset{b \in Q, \, b \geq a}{\varprojlim} \, G(b)$.
    Moreover, the latter define an adjoint equivalence between $\Per(\vec R^{op})$ and $\Per(\vec Q^{op})$.
\end{lemma}
\begin{proof}
    For $G \in \Mod(\vec Q^{op})$ and $b \in Q$, consider the natural isomorphism $\eta_b : \overline{G}(b) \xrightarrow{\sim} G(b)$. 
    It is straightforward to check that the map 
    \[\hom_{\Mod(\vec R^{op})} (F, \overline{G}) \to \hom_{\Mod(\vec Q^{op})} (F_{| \vec Q}, G), \quad (\psi_a)_{a \in R} \mapsto (\eta_b \circ \psi_b)_{b \in Q} \]
    has an inverse
    \[\hom_{\Mod(\vec Q^{op})} (F_{| \vec Q}, G) \to \hom_{\Mod(\vec R^{op})} (F, \overline{G}), \quad (\varphi_b)_{b \in Q} \mapsto (\overline{\varphi}_a)_{a \in R}, \]
    where $\overline{\varphi}_a : F(a) \to \overline{G}(a)$ is induced by the family of morphisms $(F(a) \xrightarrow{\varphi_b \circ F(a \leq b)} G(b))_{b \in Q}$.
    The result follows.
\end{proof}
 
\begin{definition}\label{definition from Q to R}
    Let $\vec Q \to \vec R$ be a limit-dense inclusion of pogs.
    If $\cC$ is a category enriched over $\Mod(\vec Q^{op})$, we denote by $\overline{\cC}$ the category obtained by applying the adjunction of Lemma \ref{lemma from Q to R} to hom spaces.
    The category $\overline{\cC}$ is enriched over $\Mod(\vec R^{op})$. 
\end{definition}

\subsection{Completion}\label{section completion}

The purpose of this section is to deal with the following sort of issue: 

\begin{example}
    Consider $\vec P = \vec \R$, $P_0 = \Z$. We consider $\module = \bigoplus_{\theta \in [0, 1)\cong S^1} t^{\theta}\Z[[t]]$ as a $\R/\Z(=S^1)$-graded $\Z[\R_+]$-module. Namely, the action is given by
    \begin{eqnarray*}
        \R_+\ni a\colon t^{\theta+n}\mapsto t^{\theta+a+n}.
    \end{eqnarray*}
    Then $\module$ is also a $S^1$-graded $\Z[\Z_+](=\Z[t])$-module, which is complete for the $t$-adic filtration (i.e. also a $S^1$-graded $\Z[[t]]$-module).  
    
    On the other hand, while $\module$ is a module over $\Z[\R_+]$, respecting the $S^1$ grading in the natural way, it is not complete for the induced filtration, i.e., the $\Z[\R_+]$ action does not extend to the action of the Novikov ring. For example, we have $\sum_{n\in \N}t^{n+\frac{1}{n}}\in \Nov$, but $\sum_{n\in \N}t^{n+\frac{1}{n}}:=\sum_{n\in \N}t^{n+\frac{1}{n}}\cdot t^0\not \in \module$. 
    Now consider the ideal $I=\left< t^1\right>\in \Z[\R_+]$. Then $\module/I^n\module\cong \Z[\R]/I^n$.
    Hence the completion is
    \begin{equation*}
        \widehat \module=\lim_{n\rightarrow \infty}\module/I^n\module\cong \lim_{n\rightarrow \infty}\Z[\R]/I^n\cong \Nov.
    \end{equation*}
\end{example}

Let $P$ be a directed pog, namely, for any $c,c'\in P$, there exists $c''\in P$ such that $c,c'\leq c''$ (note that the definition of pog in \cite{Clifford-pog} is directed pog). Since $P$ is directed, $\{c+P_+\}_{c\in P_+}$ forms a directed system of monoid ideals, and it defines a directed system of ideals $\{I_{c}\}_{c\in P_+}$ of $\Z[P_+]$.  Then one can define
\begin{equation*}
   \widehat{\Z[P_+]}:=\lim_{\substack{\longleftarrow\\ c\in P_+}}\Z[P_+]/I_c.
\end{equation*}

Let $P_0 \subset P$ be a normal subgroup.  We write $(P_0)_+ := P_0 \cap P_+$. 

Recall from Remark \ref{rmk filtration} that $\vec P/P_0$ categories are enriched over $\Z[(P_0)_+]$. 

\begin{lemma} \label{P0 complete setup}
    Let $\cC$ be a category enriched over $P/P_0$-graded $\Z[P_+]$-modules which are complete with respect to the filtration induced by the $\Z[(P_0)_+]$-action.  Then the category $\cC\# P/P_0$ is a $\vec P/P_0$-category where the $\Z[(P_0)_+]$ enrichment extends to a $\widehat{\Z[(P_0)_+]}$-enrichment, i.e. the Hom spaces are complete for the filtration induced by $(P_0)_+$.
\end{lemma}
\begin{proof}
    Since $\cC$ is enriched over $P/P_0$-graded $\Z[P_+]$-modules, for any $c, c'\in \cC$, the hom-space has the form
    \begin{equation*}
        \Hom_{\cC}(c, c')=\bigoplus_{g\in P/P_0}\Hom_g(c,c')
    \end{equation*}
    where each $\Hom_g(c,c')$ is a $\Z[(P_0)_{+}]$-module. By the completeness assumption, $\Hom_g(c,c')$ is moreover a module over $\widehat{\Z[(P_0)_+]}$. Since each hom-space of $\cC\# P/P_0$ is given by $\Hom_g(c,c')$, we complete the proof.
\end{proof}

More generally, we have:
\begin{lemma}
    Let $Q$ be a subgroup satisfying $P_0\subset Q\subset P$. Let $\cC$ be a category enriched over $P/P_0$-graded $\Z[P_+]$-modules which are complete with respect to the filtration induced by the $\Z[(P_0)_+]$-action.

    Then the category $\cC\# (P/P_0)|_{\vec Q/P_0}$ is a $\vec Q/P_0$-category where the $\Z[(P_0)_+]$ enrichment extends to a $\widehat{\Z[(P_0)_+]}$-enrichment, i.e. the Hom spaces are complete for the filtration induced by $(P_0)_+$.

    Moreover, the category $\cC\# (P/P_0)|_{\vec Q/P_0}[\vec Q/P_0]$ is enriched over $Q/P_0$-graded $\Z[Q_+]$-modules which are complete with respect to the filtration induced by the $\Z[(P_0)_+]$-action.
\end{lemma}

\begin{lemma-definition}\label{lemma-definition completion of category}
    In the setup of Lemma \ref{P0 complete setup}, 
    we write $\widehat \cC$ for the category obtained by completing Hom spaces with respect to the $P_+$ filtration.  The Hom spaces are no longer $P/P_0$-graded, and the category is enriched over $\widehat {\Z[P_+]}$.
\end{lemma-definition}

\subsection{Almost stuff}
Let us recall some standard notions of almost mathematics. 
Fix a local ring $R$ whose maximal ideal $\mathfrak{m}$ is satisfies $\mathfrak{m}^2=\mathfrak{m}$.  We refer to such data as an almost setup.
\begin{example}
    The Novikov ring $\Nov$ with its maximal ideal is an almost setup. 
\end{example}
\begin{defin} Let $(R, \mathfrak{m})$ be an almost setup. 
    \begin{enumerate}
        \item An $R$-module is \emph{almost zero} if $\mathfrak{m} M=0$.
        \item We denote by $\Mod(R)$, the derived category of $R$-modules. Take the quotient by the almost zero modules, we obtain the category $\Mod(R^a)$, the category of almost modules.
        \item A morphism in $\Mod(R)$ is said to be \emph{almost isomorphic} if it is isomorphic in $\Mod(R^a)$.
    \end{enumerate}
\end{defin}

We now consider categories enriched over $R$. 
\begin{defin} Let $F\colon \cC_1\rightarrow \cC_2$ be an $R$-linear functor between $R$-enriched categories.
    \begin{enumerate}
        \item We say $F$ is \emph{almost fully faithful} if each morphism between hom-space is almost isomorphism. In other words, the induced functor $F\otimes_{\Mod(R)}\Mod(R^a)\colon \cC_1\otimes_{\Mod(R)}\Mod(R^a)\rightarrow \cC_2\otimes_{\Mod(R)}\Mod(R^a)$ is fully faithful.
        \item We say $F$ is \emph{almost essentially surjective} if the following holds: For any $c\in \cC_2$, there exists an object in $c'$ such that the Yoneda modules $\cY(F(c'))$ and $\cY(c)$ are almost isomorphic.
        \item We say $F$ is an \emph{almost equivalence} if it satisfies the above two conditions.
    \end{enumerate}
\end{defin}

\begin{remark}\label{remark:Homotopyfullyfaithful}
    Let $F\colon \cC_1\to \cC_2$ be an almost fully faithful functor. Take the homotopy categories $H^0(F)\colon H^0(\cC_1)\rightarrow H^0(\cC_2)$. Now each hom-space is an $R$-module (not a complex). We put
    \begin{equation*}
        \Hom_{H^0(\cC_i)^a}(c_1, c_2):= \Hom_{H^0(\cC_i)}(c_1, c_2)/\left\{f\in \Hom_{H^0(\cC_i)}(c_1, c_2)| \mathfrak{m}\cdot f=0\right\}.
    \end{equation*}
    Then we get a functor and categories $H^0(F)^a\colon H^0(\cC_1)^a\rightarrow H^0(\cC_2)^a$. This functor is fully faithful in the usual sense. 
\end{remark}

\section{Curved $\Ainf$-categories and quotients}
\label{section cAinf-categories}

We will be working here with curved $\Ainf$ categories.  The purpose of the present section is to recall the basic definitions and in particular describe a quotient construction. 
We leave to the reader to check that the constructions in Section \ref{section pogs} can be generalized to this setting.

(Taking quotients of $\Ainf$-categories plays a role in the construction of wrapped Fukaya categories \cite{GPS20, GPS2}, and, relatedly, of the augmentation category in \cite{Cha19}.  We will want to perform similar constructions in the curved setting.)

\subsection{Basic definitions}

In the following, we either work over a field, or we work over $\Z$ and assume that the hom spaces are cofibrant as in \cite[Section 3.1]{GPS20}. 
Observe that hom spaces that appear in Floer theory are always of this kind.

\begin{defin}\label{definition category}	
    A $\cAinf$-category $\cC$ is the data of
    \begin{enumerate}
        \item a collection of objects $\Ob(\cC)$,
        \item for every objects $X,Y$, a graded  space
        of morphisms $\cC \left( X,Y \right)$ with a decreasing complete $\Z_{\geq 0}$-filtration $\cF^{\geq *} \cC$,
        \item for every object $X$, a degree $0$ element $e_X \in \cC \left( X,X \right)$, and
        \item a family of filtration preserving, degree $(2-d)$, linear maps
        \[\mu^d : \cC \left( X_0, X_1 \right) \otimes \dots \otimes \cC \left( X_{d-1}, X_d \right) \to \cC \left( X_0, X_d \right) \]
        indexed by the sequences of objects $\left( X_0, \dots, X_d \right)$, $d \geq 0$, such that $e_X$ is a strict unit, $\mu^0 \in \cF^{\geq 1} \cC$, and $\sum\limits_{0 \leq i \leq j \leq d} \pm \mu^{d-(j-i)+1} \circ \left( \mathbf{1}^i \otimes \mu^{j-i} \otimes \mathbf{1}^{d-j} \right) = 0$.
    \end{enumerate}
    We say that $\cC$ is a cDG-category if $\mu^d = 0$ for all $d \geq 3$. For a $\cAinf$-category $\cC$, replacing $\cC(X,Y)$ with its graded quotient
    associated to the filtration, we obtain a non-curved $\Ainf$-category $\Gr(\cC)$.	
\end{defin}

\begin{example}
    There is a cDG-category $\cCh$ defined as follows:
    \begin{enumerate}
        \item the objects of $\cCh$ are the pairs $(V,d)$, where $V$ is a graded vector space with a decreasing complete $\Z_{\geq 0}$-filtration $\cF^{\geq *} V$, and $d$ is a filtration preserving endomorphism with $(d\circ d)(\cF^{\geq i}V)\subset \cF^{\geq i+1}V$,
        \item the morphisms $(V_1, d_1) \to (V_2, d_2)$ are filtration preserving maps $f : V_1 \to V_2$, 
        \item $\mu^0(V, d) = d \circ d$, $\mu^1(f) = f \circ d_1 - d_2 \circ f$ for $f : (V_1, d_1) \to (V_2, d_2)$, and $\mu^2$ is the composition of maps.
    \end{enumerate}
\end{example}

\begin{defin}\label{definition functor}
    Let $\cC, \cD$ be two $\cAinf$-categories. 
    A pre-$\cAinf$-functor $\Phi : \cC \to \cD$ is the data of 
    \begin{enumerate}
        \item a map $\Phi : \Ob (\cC) \to \Ob (\cD)$, and
        \item a family of filtration preserving, degree $(1-d)$, linear maps 
        \[\Phi^d : \cC \left( X_0, X_1 \right) \otimes \dots \otimes \cC \left( X_{d-1}, X_d \right) \to \cD \left( \Phi X_0, \Phi X_d \right) \]
        indexed by the sequences of objects $\left( X_0, \dots, X_d \right)$, $d \geq 0$, such that $\Phi^0 \in \cF^{\geq 1} \cD$, $\Phi^1 \left( e_X \right) = e_{\Phi X}$, and $\Phi^d \left( \dots, e_X, \dots \right) = 0$ for $d \geq 2$.
    \end{enumerate}
    We say that $\Phi$ is strict if $\Phi^0 = 0$.
\end{defin}

\begin{remark}\label{rmk curvature of pre-functors}
    Pre-$\cAinf$-functors form a $\cAinf$-category where a morphism from $\Phi_1$ to $\Phi_2$ is a family of maps 
    \[T^d : \cC(X_0, X_1) \otimes \dots \otimes \cC(X_{d-1}, X_d) \to \cD(\Phi_0 X_0, \Phi_1 X_d) \]
    indexed by the sequences of objects $\left( X_0, \dots, X_d \right)$, $d \geq 0$. The curvature is given by
    \[\mu^0(\Phi)^d = \sum\limits_{0 \leq i \leq j \leq d} \pm \Phi^{d-(j-i)+1} \circ \left( \mathbf{1}^i \otimes \mu_{\cC}^{j-i} \otimes \mathbf{1}^{d-j} \right) + \sum\limits_{\substack{r \geq 0 \\ 0 = i_0 \leq \dots \leq i_r = d}} \pm \mu_{\cD}^r \circ \left( \Phi^{i_1-i_0} \otimes \dots \otimes \Phi^{i_r-i_{r-1}} \right). \]
    See for example \cite[Section 3]{DL18} (where what we call a pre-$\cAinf$-functor is called a $\qAinf$-functor).
\end{remark}

\begin{defin}
    A $\cAinf$-functor is a pre-$\cAinf$-functor $\Phi$ which is not curved, i.e. $\mu^0(\Phi) = 0$ (see Remark \ref{rmk curvature of pre-functors}).
    We say that a $\cAinf$-functor $\cC \to \cD$ is an equivalence if so is the induced $\Ainf$-functor $\Gr(\cC) \to \Gr(\cD)$. 
\end{defin}

\begin{defin}\label{definition flat functor}
    We define a functor $\flat$ from the category of $\cAinf$-categories, with morphisms the \emph{strict} $\cAinf$-functors, to the category of $\Ainf$-categories as follows
    \begin{enumerate}
        \item given a $\cAinf$-category $\cC$, $\flat(\cC)$ is the $\Ainf$-category whose objects are the objects $X$ of $\cC$ which are not curved, i.e. $\mu^0(X) = 0$ (and the morphisms are the same as in $\cC$),
        \item given a strict $\cAinf$-functor $\Phi : \cC \to \cD$, $\flat(\Phi)$ is the $\Ainf$-functor induced by $\Phi$ (this is well defined because $\Phi^0=0$).
    \end{enumerate}
\end{defin}

\begin{defin}\label{definition cMod and Mod}
    Let $\cC$ be a $\cAinf$-category. 
    We denote by $\cMod(\cC)$ the cDG-category of strict pre-$\cAinf$-functors from $\cC^{op}$ to $\cCh$.
    Moreover, we let $\Mod(\cC) := \flat(\cMod(\cC))$ be the DG-category of non-curved objects in $\cMod(\cC)$, i.e. the DG-category of strict $\cAinf$-functors from $\cC^{op}$ to $\cCh$.
\end{defin}

\begin{example}
    Given $Y \in \cC$, the Yoneda assignment $X \mapsto \cC(X, Y)$ defines an element of $\cMod(\cC)$, which is in $\Mod(\cC)$ if $\mu_{\cC}^0(Y) = 0$.  
\end{example}

\subsection{Twisted complexes and bounding cochains}

Let $\cC$ be a $\cAinf$-category.

\begin{defin}[Compare with {\cite[Section 3.2]{Aur14}}]\label{definition twisted complexes}

    We define the $\cAinf$-category $\cTw(\cC)$ as follows:
    \begin{enumerate}
        \item an object of $\cTw(\cC)$ is given by
        \begin{enumerate}
	   \item a finite family $\mathfrak{X} = \left( X_i, k_i \right)_{1 \leq i \leq N}$ where $X_i$ is an object of $\cC$ and $k_i$ is an integer,
	   \item a triangular matrix $\delta = \left( \delta_{i,j} \right)_{1 \leq i \leq j \leq N}$, where $\delta_{i,j}$ is a morphism in $\cC \left( X_i, X_j \right)$ of degree $\left( k_j - k_i + 1 \right)$, such that the diagonal entries of $\delta$ and the coefficients of $\sum_{j \geq 0} \mu_{\cC}^j \left( \delta^j \right)$ (in matrix notation) are in $\cF^{\geq 1} \cC$.
        \end{enumerate}
        \item a degree $s$ morphism from $(\mathfrak{X}, \delta)$ to $(\mathfrak{X}', \delta')$ is a matrix $\left( x_{i,j} \right)_{1 \leq i \leq N, 1 \leq j \leq N'}$, where $x_{i,j}$ is a morphism in $\cC \left( X_i, X_j \right)$ of degree $\left( s + k_j' - k_i \right)$,
        \item if $(x^0, \dots, x^{d-1}) \in \hom_{\cTw(\cC)} \left( (\mathfrak{X}_0, \delta_0), (\mathfrak{X}_1, \delta_1) \right) \times \dots \times \hom_{\cTw(\cC)} \left( (\mathfrak{X}_{d-1}, \delta_{d-1}), (\mathfrak{X}_d, \delta_d) \right)$, we set (in matrix notation)
        \[\mu_{\cTw(\cC)}^d \left( x^0, \dots, x^{d-1}  \right) := \sum \limits_{j_0, \dots, j_d \geq 0} \mu_{\cC} \left( \delta_0^{j_0}, x^0, \delta_1^{j_1}, x^1, \dots, \delta_{d-1}^{j_{d-1}}, x^{d-1}, \delta_d^{j_d} \right). \]
    \end{enumerate}
\end{defin}

\begin{defin}\label{definition bounding cochains}
    The $\Ainf$-category $\cC^{bc}$ of bounding cochains in $\cC$ is the full subcategory of $\cTw(\cC)$ with objects the non-curved twisted complexes of the form $((X, 0), b)$.
\end{defin}

\begin{lemma}\label{lemma bounding cochains generate flat twisted complexes}
    The $\Ainf$-category $\flat(\cTw(\cC))$ (see Definition \ref{definition flat functor}) is generated by $\cC^{bc}$.
\end{lemma}
\begin{proof}
    There is an identification $\Tw(\cC^{bc}) = \flat(\cTw(\cC))$.
\end{proof}

\subsection{Quotient of $\cAinf$-categories}
\label{section localization}

Let $\cC$ be a $\cAinf$-category, and let $\cA \subset \cC$ be a full subcategory.

\begin{defin}\label{definition quotient}
    The $\cAinf$-quotient $\cC / \cA$ is the $\cAinf$-category such that
    \begin{enumerate}
        \item the objects of $\cC / \cA$ are those of $\cC$,
        \item the space of morphisms from $X$ to $Y$ is
        \[\cC / \cA \left( X,Y \right) = \cC(X,Y) \oplus \bigoplus \limits_{\substack{p \geq 1 \\ A_1, \dots, A_p \in \cA} } \cC \left( X, A_1 \right)[1] \otimes \dots \otimes \cC \left( A_{p-1}, A_p \right) [1] \otimes \cC \left( A_p, Y \right) \]
        \item $\mu_{\cC/\cA}^0 = \mu_{\cC}^0$, and for $d \geq 1$, $\mu_{\cC/\cA}^d$ sends any sequence $\mathbf{x}_i = (x_i^0 , \dots , x_i^{p_i}) \in \cC / \cA ( X_i , X_{i+1} )$, $0 \leq i \leq d-1$, to 
        \[\mu_{\cC / \cA}^d \left( \mathbf{x}_0, \dots, \mathbf{x}_{d-1} \right) = \sum \limits_{\substack{0 \leq i \leq p_0, 1 \leq j \leq p_{d-1} \\ i \leq j \text{ if } d=1}} x_0^0 \otimes \dots \otimes x_0^i \otimes \mu_{\cC} \left( x_0^{i+1}, \dots, x_{d-1}^j \right) \otimes x_{d-1}^{j+1} \otimes \dots \otimes x_{d-1}^{p_{d-1}}. \]	
    \end{enumerate}
    Observe that there is a strict $\cAinf$-functor $\cC \to \cC / \cA$ given by the identity on objects and the inclusion on morphisms.
\end{defin}

\begin{lemma}[Compare with {\cite[Lemma 3.12]{GPS20}}]\label{lemma quotient module against quotient object}
    For every $\cM \in \cMod(\cC)$ and $A$ in $\cA$, $\Gr(_{\cA \backslash} \cM ( A ))$ is acyclic. 	
\end{lemma}
\begin{proof}
    This follows from $\mu_{\Gr_0(\cC / \cA)}^1 [e_A \otimes e_A] = \left[ \mu_{\cC / \cA}^1 (e_A \otimes e_A) \right]_{\Gr_0(\cC / \cA)} = [e_A]_{\Gr_0(\cC / \cA)}$.
\end{proof}

\begin{lemma}[Compare with {\cite[Lemma 3.13]{GPS20}}]\label{lemma quasi-iso between a module and its localization}
    Let $\cM \in \cMod(\cC)$. If $\Gr(\cM(A))$ is acyclic for every $A$ in $\cA$, then the inclusion $\Gr(\cM(X)) \hookrightarrow \! \Gr(_{\cA \backslash}\cM(X))$ is a quasi-isomorphism for every object $X$.
\end{lemma}
\begin{proof}
    We have to show that the chain complex 
    \[\Gamma := \bigoplus_{\substack{p \geq 1 \\ A_1, \dots, A_p \in \cA}} \Gr \left( \cC(X, A_1)[1] \otimes \dots \otimes \cC(A_{p-1}, A_p) [1] \otimes \cM(A_p) \right) \]
    is acyclic.
    Consider the length filtration
    \[\cG^{\leq \ell} \Gamma := \bigoplus_{\substack{1 \leq p \leq \ell \\ A_1, \dots, A_p \in \cA}} \Gr \left( \cC(X, A_1)[1] \otimes \dots \otimes \cC(A_{p-1}, A_p) [1] \otimes \cM(A_p) \right) \]
    and observe that it is enough to show that the associated graded of this filtration is acyclic.
    Therefore, it is enough to show that given objects $A_1, \dots, A_{\ell}$ in $\cA$ and non-negative integers $k, i_0, \dots, i_{\ell}$ with $i_0+\dots+i_{\ell} = k$, the quotient $C /D$ is acyclic, where 
    \[C = \cF^{\geq i_0} \cC(X, A_1) [1] \otimes \dots \otimes \cF^{\geq i_{\ell-1}} \cC(A_{\ell-1}, A_{\ell})[1] \otimes \cF^{\geq i_{\ell}} \cM(A_{\ell}) \]
    and 
    \[D = \bigoplus_{\substack{j_0+\dots+j_{\ell} = k+1 \\ j_r \geq i_r}} \cF^{\geq j_0} \cC(X, A_1)[1] \otimes \dots \otimes \cF^{\geq j_{\ell-1}} \cC(A_{\ell-1}, A_{\ell}) [1] \otimes \cF^{\geq j_{\ell}} \cM(A_{\ell}). \]
    Since $\Gr_{i_{\ell}} \cM (A_{\ell})$ is acyclic (and projective), there is a map $h : \cF^{\geq i_{\ell}} \cM (A_{\ell}) \to \cF^{\geq i_{\ell}} \cM (A_{\ell})$ of degree $(-1)$ such that $h (\cF^{\geq i_{\ell}+1} \cM (A_{\ell})) \subset \cF^{\geq i_{\ell}+1} \cM (A_{\ell})$ and $x - (h \circ \mu_{\cM}^1 + \mu_{\cM}^1 \circ h)(x) \in \cF^{\geq i_{\ell}+1} \cM (A_{\ell})$ for every $x \in \cF^{\geq i_{\ell}} \cM (A_{\ell})$.
    Then the map 
    \[C \to C, \quad x^0 \otimes x^1 \otimes \dots \otimes x^{\ell-1} \otimes x^{\ell} \mapsto x^0 \otimes x^1 \otimes \dots \otimes x^{\ell-1} \otimes h(x^{\ell}) \]
    induces a homotopy between $\id_{C/D}$ and $0$.
\end{proof}

\begin{lemma}\label{lemma quotient commutes with flat functor}
    If $\cA$ contains only non-curved objects, i.e. if $\cA \subset \flat(\cC)$ (see Definition \ref{definition flat functor}), then there is a tautological $\Ainf$-equivalence $\flat \left( \cC / \cA \right) = \flat(\cC) / \cA$.
\end{lemma}
\begin{proof}
    The two categories are the same, since $\mu_{\cC / \cA}^0 = \mu_{\cC}^0$.
\end{proof}

\paragraph{Localization}

Let $M$ be a set of degree $0$ morphisms in $\cC$ such that $\mu_{\cC}^1(m) \in \cF^{\geq 1} \cC$ for every $m \in M$. 

\begin{defin}\label{definition localization}
    For $m : X_1 \to X_2$ in $M$, we set
    \[\Cone (m) := \left( \left( \left( X_1, 1 \right), \left( X_2, 0 \right) \right), \left( 
    \begin{smallmatrix}
    0 & m \\
    0 & 0
    \end{smallmatrix} 
    \right) \right) \in \cTw(\cC) . \]
    We denote by $\mathrm{Cone} (M)$ the full subcategory of $\cTw(\cC)$ with these objects.
    The localization $\cC \left[ M^{-1} \right]$ of $\cC$ at $M$ is the full subcategory of $\cTw(\cC) / \Cone (M)$ whose objects are those of $\cC$.	
\end{defin}

The definitions we made imply the following.

\begin{lemma}\label{lemma perfect modules commutes with localization}
    There is a $\cAinf$-equivalence $\cTw(\cC) / \Cone(M) \simeq \cTw(\cC[M^{-1}])$.
\end{lemma}

\section{Elementary symplectic geometry of prequantization bundles}\label{section prequantization bundles}

Here we fix some notations and hypotheses for the geometric setup of the remainder of the article.  

\vspace{2mm}

\begin{definition}\label{definition prequantization bundle}
    A \emph{prequantization bundle} is the data of a symplectic manifold $(\base, \omega)$, a contact manifold $(V,\xi)$ with a contact form $\alpha^{\circ}$, and a circle bundle $V \xrightarrow{\pi} \base$ such that
    \begin{enumerate}
        \item $e(V \xrightarrow{\pi} \base) = [\omega]$ in $H^2(\base; \R)$,
        \item $\pi^* \omega = d \alpha^{\circ}$,
        \item the $\alpha^{\circ}$-Reeb flow generates the circle action on $V$ and is $1$-periodic.
    \end{enumerate}	
\end{definition}

We use the superscript on `$\alpha^\circ$' to remind ourselves that this is a circle-invariant contact form; in particular, the Reeb flow is degenerate. We will later use `$\alpha$' for a form with non-degenerate Reeb flow.

\begin{remark}
    Given a symplectic manifold $(\base, \omega)$, there exist a prequantization bundle over it if and only if $\langle \omega, H_2(\base) \rangle \subset \Z$.
\end{remark}

\begin{definition}\label{definition admissible}
    We say a 
    prequantization bundle $(V, \alpha^{\circ}) \to (\base, \omega)$ is {\em admissible} if the following holds:
    \begin{enumerate}
        \item $(\base, \omega)$ is monotone, i.e. there exists $\tau > 0$ such that $c_1(T \base) = \tau [\omega]$,
        \item $c_1(\xi)= 0$, and there exists a non-degenerate contact form $\alpha$ on $(V, \xi)$ such that the contact homology grading $|\gamma|$ of any contractible $\alpha$-Reeb orbit satisfies $|\gamma| \geq 2$.
    \end{enumerate}
    We say a compact Legendrian $\La$ in an admissible prequantization bundle is {\em admissible} if the following holds:
    \begin{enumerate}
        \item $\pi_{|\La}$ 
        has double and transverse multiple points,
        \item the action spectrum of $\La$ (i.e. the set of lengths of $\alpha^{\circ}$-Reeb chords of $\La$) is contained in $\Z \cup (\R \setminus \Q)$.  
    \end{enumerate} 
\end{definition}

In the following, we fix a prequantization bundle $(V, \alpha^{\circ}) \xrightarrow{\pi} (\base, \omega)$ and let $a_{\base} \in \Z_{\geq 0}$ be the generator of $\langle \omega, H_2(\base) \rangle$.

\begin{lemma}\label{lemma fundamental group}
    Let $\gamma$ be a simple fiber of $V \to \base$.
    A $N$-fold cover of $\gamma$ is contractible if and only if $N \in a_{\base} \Z_{\geq 1}$.
\end{lemma}
\begin{proof}
    This follows from the exact sequence $\pi_2(\base)  \xrightarrow{\langle [\omega], - \rangle} \Z = \pi_1(S^1)  \to \pi_1 (V)$. 
\end{proof}

\begin{lemma}\label{lemma chern classes}
    We have $c_1(\xi) = 0$ if and only if $c_1(T \base) = \tau [\omega]$ for some $\tau \in \R$.
\end{lemma}
\begin{proof}
    This follows from the Gysin exact sequence, since $c_1(\xi) = c_1(\pi^* T \base) = \pi^* c_1(T \base)$ and $e(V \xrightarrow{\pi} \base) = [\omega]$.
\end{proof}

\begin{lemma}\label{lemma indices}
    Assume that $c_1(T \base) = \tau [\omega]$. 
    For every $L, \varepsilon \in \R_{>0}$ there exists a positive Morse function $f$ on $\base$ such that
    \begin{enumerate}
	\item $f$ is $\varepsilon$-close to $1$ in $C^2$-norm,
	\item the $(f \alpha^{\circ})$-Reeb orbits of action less than $L$ are multiple covers of the fibers over critical points of $f$ (and are non-degenerate),
	\item the Conley-Zehnder index of a contractible $(f \alpha^{\circ})$-Reeb orbit $\gamma$ of action less than $L$ is a well defined integer given by $\mathrm{CZ} (\gamma) = 2 \tau N_{\gamma} - \frac{\dim(\base)}{2} + \mathrm{ind}_f(\pi(\gamma))$, where $N_{\gamma} \in a_{\base} \Z_{\geq 1}$ is the degree of $\gamma$ along the fiber.
    \end{enumerate}
    In particular, the contact homology grading $|\gamma|$ of a contractible $(f \alpha^{\circ})$-Reeb orbit $\gamma$ of action less than $L$ satisfies 
    \[|\gamma| = \mathrm{CZ} (\gamma) + \frac{\dim(\base)}{2} - 2 = 2 ( \tau N_{\gamma} - 1) + \mathrm{ind}_f(\pi(\gamma)) \geq 2 ( \tau a_\base - 1). \]
\end{lemma}
\begin{proof}
    This computation can be found in \cite[Section 3.1]{DL19}.	
\end{proof}

\begin{prop}\label{prop admissibility}
    Assume that $(\base,\omega)$ is 
    monotone and $1 \notin \langle c_1 (T \base), H_2(\base) \rangle$. 
    Then the prequantization bundle $(V, \alpha^{\circ}) \to (B, \omega)$ is admissible. 
\end{prop}
\begin{proof}
    Admissibility of the prequantization bundle follows from Lemmas \ref{lemma fundamental group}, \ref{lemma chern classes} and \ref{lemma indices}.
\end{proof}

\section{Legendrian invariants}
\label{section augmentation category}

To a Legendrian $\La$, one can associate a ``positive augmentation category'', which computes the category of finite-dimensional modules over the Chekanov-Eliashberg algebra $\cA_{\La}$.  
This category was introduced first for Legendrians in $\mathbf{R}^3$ by somewhat ad-hoc methods \cite{NRS+20, CNS19} ; a general definition was given in \cite{Cha19} using localization of categories. 
Here we denote this category by $\LFuk(S\La)$, and we define a curved version $\cLFuk(S\La)$ of it as the localization of a curved $\Ainf$-category.

We will work in the following setting. Fix a contact manifold $(V, \xi)$ with a non-degenerate contact form $\alpha$.  
We will assume: 
\begin{enumerate}
    \item $c_1(\xi)=0$ in $H^2(V; \R)$,
    \item for every contractible $\alpha$-Reeb orbit $\gamma$, the contact homology grading $|\gamma|$ of $\gamma$ satisfies $|\gamma| \geq 2$.
\end{enumerate}
Moreover, we fix a compact Legendrian $\La \subset V$ which is chord generic with respect to $\alpha$.

\begin{remark}\label{rmk assumptions on V}
    Recall that, in general, the Legendrian invariants should be defined using discs with interior punctures asymptotic to Reeb orbits of $V$, or, in other words, over the DG-algebra $\cA(V)$ generated by Reeb orbits.
    The above assumptions allow us to define the Legendrian invariants without having to consider discs with interior punctures asymptotic to Reeb orbits.
    More precisely, the first assumption ensures the Conley-Zehnder index is well defined as an integer; it is required to state the second assumption. The second assumption ensures that the trivial algebra map $\cA(V) \to \Z$ (which sends any Reeb orbit to $0$) is an augmentation.
    Observe that these assumptions are standard in Legendrian Floer theory, see for example \cite[Appendix B.2]{Ekh08} or \cite[Section 3.1]{DR16bis}. 
\end{remark}

\begin{remark}\label{rmk admissible satisfy assumptions}
    The above assumptions are satisfied when $V$ is the total space of an admissible prequantization bundle $V \xrightarrow{\pi} \base$ (see Definition \ref{definition admissible}).
    According to Proposition \ref{prop admissibility}, $V \xrightarrow{\pi} \base$ is admissible if $B$ is monotone with minimal Chern number different from $1$.
\end{remark}

Finally, we choose a generic almost complex structure on $\xi$. We denote by $J_{\alpha}$ the almost complex structure on $SV$ which restricts to $J$ on $\xi$ and such that $J_{\alpha} \partial_t = R_{\alpha}$ (where we use $\alpha$ to identify $SV$ and $\R_t \times V$).

\subsection{Definition}\label{section definition augmentation category}

Fix $\eta > 0$ smaller than the length of the smallest Reeb chord of $\La$, and set $\La^{-\eta} := \varphi_{R_{\alpha}}^{- \eta} (\La)$.
Pick a positive Morse function $f_1 : \La^{-\eta} \to \mathbf{R}$ with a unique minimum on each connected component $\La_i^{-\eta}$ of $\La^{-\eta}$.
We denote by $\cU$ a small open neighborhood of $f_1$ in $C^{\infty} \left( \La^{-\eta} \right)$ with the following property: for every $f$ in $\cU$, $f$ is positive, Morse, and its critical points and gradient trajectories are in bijection with those of $f_1$. 
Now pick a family $\left( f_n \right)_{n \geq 2}$ of functions in $\cU$.
Then choose a decreasing sequence $\left( \delta_k \right)_{k \geq 1}$ of small enough positive numbers so that the family
\[h_n = \sum_{k=1}^{n} \delta_k f_k, \quad n \in \Z_{\geq 0}, \]
has the following properties (these can be easily achieved by induction on $n$): 
\begin{enumerate}
    \item for every $m<n$, $\left( h_n - h_m \right)$ lies in the cone over $\cU$, 
    \item the maxima of $h_n$ are smaller than $\eta$.
\end{enumerate}
Finally, let $\La_i^{-\eta}(n)$ be the Legendrian corresponding to the one-jet of $(h_n)_{| \La_i^{-\eta}}$ in a Weinstein neighborhood of $\La_i^{-\eta} = \La_i^{-\eta}(0)$.

\begin{defin}\label{definition qAug}
    Let $\cR(\La_i^{-\eta}(m), \La_j^{-\eta}(n))$ denotes the set of $\alpha$-Reeb chords from $\La_i^{-\eta}(m)$ to $\La_j^{-\eta}(n)$.
    We consider the $\cAinf$-category $\cO$ defined as follows:
    \begin{enumerate}
	\item An object of $\cO$ is a Legendrian $\La_j^{-\eta}(n)$ with a local system $M \in \Prop(C_{-*}(\Omega \La_j))$ on $\La_j$, where $\La_j$ is a connected component of $\La$ and $n \in \Z_{\geq 0}$.
    We usually forget the local system $M_j$ in the notation of an object.
	\item Given local systems $M_0 \in \Prop(C_{-*}(\Omega \La_i))$ and $M_1 \in \Prop(C_{-*}(\Omega \La_j))$, we have
    \begin{align*}
        \cO & \left( \La_i^{-\eta}(m), \La_j^{-\eta}(n) \right) \\
        & = \left\{
        \begin{array}{ll}
            \prod\limits_{\gamma \in \cR(\La_i^{-\eta}(m), \La_j^{-\eta}(n))} \hom_{\Z}(M_0, M_1) & \text{if } m < n \\
            \Z \oplus \prod\limits_{\gamma \in \cR(\La_j^{-\eta}(n))} \hom_{\Z}(M_0, M_0) & \text{if } \La_i^{-\eta}(m) = \La_j^{-\eta}(n) \text{ and } M_0=M_1 \\
	    0 & \text{otherwise}.
        \end{array}
        \right. 
    \end{align*}
        \item The operations are defined using $J_{\alpha}$-holomorphic punctured discs in $SV$ with one positive puncture.
    \end{enumerate}
\end{defin}

Since $\La_i$ is chord-generic with respect to $\alpha$, we can choose the perturbations $\La_j^{-\eta}(n)$ so that there are trivial isomorphisms of $\cAinf$-algebras
\[\phi_j^n : \cO \left( \La_j^{-\eta}(n), \La_j^{-\eta}(n) \right) \xrightarrow{\sim} \cO \left( \La_j^{-\eta}(n+1), \La_j^{-\eta}(n+1) \right). \]

\begin{lemma}\label{lemma minima}
    Fix a connected component $\La_j$ of $\La$, and $n \in \Z_{\geq 0}$.
    Then the minimum of $\left( h_{n+1} - h_n \right)$ on $\La_j^{-\eta}$ defines a degree 0 morphism $\gamma_j^n \in \cO \left( \La_j^{-\eta}(n), \La_j^{-\eta}(n+1) \right)$ satisfying the following 
    \begin{enumerate}
        \item $\mu_{\cO}^2(-, \gamma_j^n) : \cO(\La_i^{- \eta} (m), \La_j^{-\eta}(n)) \to \cO(\La_i^{- \eta} (m), \La_j^{-\eta}(n+1))$ is a quasi-isomorphism if $m < n$,
        \item $\mu_{\cO}^2(-, \gamma_j^n) = \mu_{\cO}^2(\gamma_j^n, \phi_j^n(-))$ on $\cO \left( \La_j^{-\eta}(n), \La_j^{-\eta}(n) \right)$,
        \item $\mu_{\cO}^d(\dots, \gamma_j^n, \dots) = 0$ for $d \ne 2$.
    \end{enumerate} 
\end{lemma}
\begin{proof}
    This follows from \cite[Theorem 3.6]{EES09}. Indeed, the latter implies that a rigid pseudo-holomorphic disk negatively asymptotic to $\gamma_j^n$ at one puncture corresponds to a rigid pseudo-holomorphic disk with one puncture less and a gradient flow line of $(h_{n+1} - h_n)$ converging to $\gamma_j^n$ attached on its boundary.
    Since $\gamma_j^n$ is the minimum of $(h_{n+1} - h_n)$, the only possible rigid configurations are trivial cylinders with a marked point on the boundary.
    This gives the desired result.   
\end{proof}

\begin{defin}\label{definition morphisms for localization}
    We denote by $M$ be the set of closed, degree 0, morphisms in $\cO$ given by the minima $\gamma_j^n$ of the functions $\left( h_{n+1} - h_n \right)$ on $\La_j^{-\eta}$.
    According to Lemma \ref{lemma minima}, each $\gamma_j^n$ defines a closed morphism in $\cO^{bc} \left( (\La_j^{-\eta}(n),b), (\La_j^{-\eta}(n+1), \phi_j^n(b)) \right)$ whenever $(\La_j^{-\eta}(n), b) \in \cO^{bc}$.    
    We denote by $N$ the set of these morphisms in $\cO^{bc}$.
\end{defin}

\begin{defin}\label{definition augmentation category}
    We define 
    \begin{enumerate}
        \item the $\cAinf$-category $\cLFuk(S\La)$ as the full subcategory of $\cO \left[ M^{-1} \right]$ whose objects are the Legendrians $\La_i^{-\eta}$ (together with local systems),
        \item the $\Ainf$-category $\LFuk(S\La)$ as the full subcategory of $\cO^{bc} \left[ N^{-1} \right]$ whose objects are the pairs $(\La_i^{-\eta}, b_i)$.
    \end{enumerate}
\end{defin}

\subsection{Some properties}

\begin{prop}\label{prop independence on choices}
    The categories $\cLFuk(S\La)$ and $\LFuk(S\La)$ are independent on the choices made in the construction, i.e. the number $\eta > 0$ and the Morse functions used to perturb $\La^{- \eta}$.
    Moreover, these categories are invariant under Legendrian isotopies of $\La$.
\end{prop}
\begin{proof}
    The strategy to prove independence on choices of categories defined by localization is standard, see \cite[Proposition 3.39]{GPS20}. The proof in our context is exactly the same as the one of \cite[Proposition 3.3.2]{Cha19}.
\end{proof}

\begin{prop}\label{prop stop removal map for Aug} 
    If $\La \subset \La'$ are two Legendrians in $V$, then there are canonical embeddings $\cLFuk(S\La) \hookrightarrow \cLFuk(S\La')$ and $\LFuk(S\La) \hookrightarrow \LFuk(S\La')$.
\end{prop}
\begin{proof}
    Fix any $\eta > 0$ smaller than the length of the smallest Reeb chord of $\La'$.  Then we may use the inclusion at the level of objects and morphisms to define the desired functors. 
\end{proof}

\begin{prop}\label{prop aug computes augmentations of the algebra}
    Let $\cA_{\La}$ be Chekanov-Eliashberg DG-category of $\La$ with loop space coefficients.
    Then there is an $\Ainf$-equivalence $\Psi : \LFuk(S\La) \xrightarrow{\sim} \Prop(\cA_{\La})$.
\end{prop}
\begin{proof}
    This is essentially contained in \cite{EL21} when considering only $1$-dimensional modules (in this case, the left hand side is nothing but the positive augmentation category of $\La$). The proof is a combination of two known facts:
    \begin{enumerate}
        \item Morse complexes with local systems coefficients computes morphisms between $C_{-*}(\Omega \La)$-modules,
        \item The category of bounding cochains for the curved $\Ainf$-algebra generated by Reeb chords of $\La$ is equivalent to the category of augmentations for the Chekanov-Eliashberg DG-algebra of $\La$ with coefficients in $\Z$.
    \end{enumerate}
\end{proof}

\paragraph{Relation between $\cLFuk(S\La)$ and $\LFuk(S\La)$}

\begin{lemma}\label{lemma computation hom in cAug}
    There are natural quasi-isomorphisms
    \[\Gr \hom_{\cO \left[ M^{-1} \right]} (\La_i^{- \eta}, \La_j^{- \eta}) \xrightarrow{\sim} \underset{n}{\colim} \Gr \hom_{\cO \left[ M^{-1} \right]} (\La_i^{- \eta}, \La_j^{-\eta}(n)) \xleftarrow{\sim} \underset{n}{\colim} \Gr \cO(\La_i^{- \eta}, \La_j^{-\eta}(n)) \]
    where the colimits are defined using multiplication by $\gamma_j^n$.
\end{lemma}
\begin{proof}
    The proof is basically the same as in \cite[Lemma 3.37]{GPS20}. We reproduce it here in our context for the reader.
    
    First we show that the leftmost arrow is a quasi-isomorphism.
    This follows from the computation
    \begin{align*}
        \Cone & \left( \Gr \hom_{\cO \left[ M^{-1} \right]} (\La_i^{- \eta}, \La_j^{-\eta}(n)) \xrightarrow{\mu^2(-, \gamma_j^n)} \Gr \hom_{\cO \left[ M^{-1} \right]} (\La_i^{- \eta}, \La_j^{-\eta}(n+1)) \right) \\
        & = \Gr \Cone \left( \hom_{\cO \left[ M^{-1} \right]} (\La_i^{- \eta}, \La_j^{-\eta}(n)) \xrightarrow{\mu^2(-, \gamma_j^n)} \hom_{\cO \left[ M^{-1} \right]} (\La_i^{- \eta}, \La_j^{-\eta}(n+1)) \right) \\
        & = \Gr \left( _{\Cone(M) \backslash} \cM (\Cone(\gamma_j^n)) \right) \simeq 0,
    \end{align*}
    where $\cM = \hom_{\cTw(\cO)}(\La_i^{-\eta}, -)$, and the last equivalence holds according to Lemma \ref{lemma quotient module against quotient object}.

    Now we show that the rightmost arrow is a quasi-isomorphism. Consider the homotopy colimit $\cN$ of $(\cO ( -, \La_j^{- \eta}(n)))_n$ in $\cMod(\cO)$, i.e.
    \[\cN = \Cone \left( \bigoplus_n \cO ( -, \La_j^{- \eta}(n)) \xrightarrow{\bigoplus_n \id_n - \mu_{\cO}^2(-, \gamma_j^n)} \bigoplus_n \cO ( -, \La_j^{- \eta}(n)) \right). \]
    Then $\Gr \cN(\Cone(c))$ is acyclic for every $c \in M$. 
    According to Lemma \ref{lemma quasi-iso between a module and its localization}, the inclusion $\Gr \cN (\La_i^{-\eta}) \to \Gr \left( _{\Cone(M) \backslash} \cN (\La_i^{-\eta}) \right)$ is therefore a quasi-isomorphism.
    The result follows since there are compatible quasi-isomorphisms
    \[\left\{
    \begin{array}{cll}
    \Gr \cN (\La_i^{-\eta}) & \simeq & \underset{n}{\colim} \Gr \cO(\La_i^{- \eta}, \La_j^{-\eta}(n)) \\
    \Gr \left( _{\Cone(M) \backslash} \cN (\La_i^{-\eta}) \right) & \simeq & \underset{n}{\colim} \Gr \hom_{\cO \left[ M^{-1} \right]} (\La_i^{- \eta}, \La_j^{-\eta}(n)).
    \end{array}
    \right. \]
\end{proof}

\begin{lemma}\label{lemma computation hom in Aug}
    There are natural quasi-isomorphisms
    \[\Gr \hom_{\cO^{bc} \left[ N^{-1} \right]} ((\La_i^{- \eta}, b_i), (\La_j^{- \eta}, b_j)) \xrightarrow{\sim} \underset{n}{\colim} \Gr \hom_{\cO^{bc} \left[ N^{-1} \right]} ((\La_i^{- \eta}, b_i), (\La_j^{-\eta}(n), \phi_j^n(b_j)) \]
    and 
    \[\underset{n}{\colim} \Gr \hom_{\cO^{bc} \left[ N^{-1} \right]} ((\La_i^{- \eta}, b_i), (\La_j^{-\eta}(n), \phi_j^n(b_j)) \xleftarrow{\sim} \underset{n}{\colim} \Gr \cO^{bc} ((\La_i^{- \eta}, b_i), (\La_j^{-\eta}(n), \phi_j^n(b_j)) \]
    where the colimits are defined using multiplication by $\gamma_j^n$.
    Moreover, there is a tautological identification
    \[\Gr \cO^{bc} ((\La_i^{- \eta}, b_i), (\La_j^{-\eta}(n), \phi_j^n(b_j)) = \Gr \cO (\La_i^{- \eta}, \La_j^{-\eta}(n)). \]
\end{lemma}
\begin{proof}
    The proof of the two first quasi-isomorphisms is the same as in \cite[Lemma 3.37]{GPS20} (see also the proof of Lemma \ref{lemma computation hom in cAug}).

    The second claim follows since the bounding cochains are in $\cF^{\geq 1} \cO$, so that the chain complexes on both sides have same underlying vector space and same differential.
\end{proof}

\begin{prop}\label{proposition boundig cochains on cAug = Aug}
    The category of bounding cochains on $\cLFuk(S\La)$ is Morita equivalent to $\LFuk(S\La)$.
\end{prop}
\begin{proof}
    We want to prove that $\Tw(\cO[M^{-1}]^{bc}) \simeq \Tw(\cO^{bc}[N^{-1}])$ (see \cite[Section A.4]{GPS3}). 
    Observe that $\Tw(\cO^{bc} [N^{-1}]) = \Tw(\cO^{bc}) / \Cone(N)$.
    Therefore, according to Lemma \ref{lemma bounding cochains generate flat twisted complexes}, we have to show that $\flat(\cTw(\cO[M^{-1}])) \simeq \flat(\cTw(\cO)) / \Cone(N)$ where $\flat$ is the functor introduced in Definition \ref{definition flat functor}.
    Using Lemma \ref{lemma perfect modules commutes with localization} and \ref{lemma quotient commutes with flat functor}, this amounts to prove that
    \[\flat(\cTw(\cO) / \Cone(M)) \simeq \flat(\cTw(\cO) / \Cone(N)).  \]
    Now observe that the categories on both sides have same objects. Moreover, using Lemmas \ref{lemma computation hom in cAug} and \ref{lemma computation hom in Aug}, it is straightforward to define a functor from one side to the other which induces a quasi-equivalence between the associated graded.
    The result follows.    
\end{proof}

\subsection{Proof of Theorem \ref{intro thm: Q mod Z structure}}
\label{section G-structure on augmentation category}

In the following, we assume that $V$ is the total space of a prequantization bundle.
Recall that, in this setting, the $\alpha^{\circ}$-Reeb flow induces an action of $S^1 = \R / \Z$ on $V$.
We also assume that the action spectrum of $\La$ (i.e. the set of lengths of the $\alpha^{\circ}$-Reeb chords of $\La$) is contained in $\Z \cup (\R \setminus \Q)$.
Let $G$ be a discrete subgroup of $\R$ containing $\Z$ (so that $G / \Z$ is a subgroup of $S^1$).
We explain why the prequantization bundle setting induces $(\vec G / \Z)$-actions (see Section \ref{section pogs}) on $\LFuk(S(G \cdot \La))$ (note that a related phenomena was considered by the second author in \cite{Pet22bis}).
To do it in the most natural way, we choose to introduce the Fukaya category of the symplectization $SV$.

According to \cite{Leg21} (see also \cite{DRS21}), there is a well defined $\Ainf$-category $\RFuk(SV)$ associated to the symplectization $SV$. An object of $\RFuk(SV)$ is a connected cylindrical exact Lagrangian cobordism $\Sigma$ in $SV$ together with an augmentation $\cA_{\partial_{-\infty} \Sigma} \to \Z$.
It is straightforward to generalize the latter to define a partially wrapped version $\RFuk(SV, \La)$\footnote{The notation $\RFuk$ stands for ``Rabinowitz-Fukaya"}, where an object is a cylindrical exact Lagrangian cobordism $\Sigma$ in $SV$ such that $(\partial_{+ \infty} \Sigma) \cap \La = \emptyset$, together with a finite-dimensional module over $\cA_{\partial_{-\infty} \Sigma}$.
Then there is an $\Ainf$-embedding
\[\LFuk(S\La) \hookrightarrow \RFuk(SV, \La)^{op}, \, (\La_i^{-\eta}, \varepsilon) \mapsto (S \La_i^{-\eta}, \varepsilon). \]

\begin{proof}[Proof of Theorem \ref{intro thm: Q mod Z structure}]
    By definition, there is a monoidal functor from $\vec G$ to $\Aut(\RFuk(SV, G \cdot \La)^{op})$ such that an element $c \in G$ is sent to (the functor induced by) $\varphi_{\Reeb}^c$, and the morphism in $\hom_{\vec{G}}(c_0, c_1)$ is sent to the continuation map $\varphi_{\Reeb}^{c_0} \Rightarrow \varphi_{\Reeb}^{c_1}$ induced by the positive isotopy $(\varphi_{\Reeb}^t)_{c_0 \leq t \leq c_1}$.
    Since the action of $\Z$ on $\RFuk(SV, G \cdot \La)^{op}$ is trivial (up to shift) on objects, we get (see Remark \ref{rmk structure on opposite category}) a $(\vec G / \Z)$-action on $\Mod(\RFuk(SV, G \cdot \La)^{op})$.
    Then the composition of $\LFuk(S\La) \hookrightarrow \RFuk(SV, \La)^{op}$ and the Yoneda functor gives an embedding $\LFuk(S\La) \hookrightarrow  \Mod(\RFuk(SV, G \cdot \La)^{op})$.
    Since $G / \Z$ preserves this subcategory of $\Mod(\RFuk(SV, G \cdot \La)^{op})$, we get a $(\vec G / \Z)$-action on $\LFuk(S\La)$.
    This completes the proof. 
\end{proof}

\subsection{Relation with partially wrapped Fukaya category}\label{section aug and partially wrapped fukaya}

Let $W$ be a Weinstein manifold whose boundary at infinity is $V$. 
We fix an embedding $S^+V := ( \R_{\geq 0} \times V, e^t \alpha ) \hookrightarrow W$ covering a neighborhood of infinity, and we denote by $\WFuk^+(W, \La)$ the full subcategory of $\WFuk(W, \La)$ whose objects are contained in the image of $S^+V$.
Observe that $\WFuk^+(W, \La)$ is split-generated by the union $D_{\La}$ of the linking discs at given points on $\La$. Indeed, by stop removal the quotient of $\WFuk^+(W, \La)$ by $D_{\La}$ is $\WFuk^+(W)$, which is trivial since $S^+V$ does not intersect the skeleton of $W$.

\begin{definition}\label{definition topologically simple}
    We say that $W$ is \emph{topologically simple} if $c_1(TW)=0$ and the map $\pi_1 (V) \to \pi_1(W)$ is injective.
\end{definition}

\begin{rmk}\label{rmk topologically simple}
    Recall that we made assumptions on $V$ so that our Legendrian invariants are defined using discs with no interior puncture, or, in other words, using the augmentation $\varepsilon_0 : \cA(V) \to \Z$ which is the trivial algebra map (see Remark \ref{rmk assumptions on V}).
    However, the partially wrapped Fukaya category of $W$ is typically related to Legendrian invariants defined using discs with interior punctures asymptotic to Reeb orbits of $V$ which are contractible in $W$, or, in other words, using the augmentation $\varepsilon_W : \cA(V) \to \Z$ induced by $W$ (see \cite{EL21}).
    The topologically simple condition on $W$ ensures that $\varepsilon_W = \varepsilon_0$.
\end{rmk}

\begin{rmk}\label{rmk condition for topologically simple}
    If $\dim(W) \geq 6$, then the map $\pi_1 (V) \to \pi_1(W)$ is injective (see for example \cite[Proposition 2.3]{Zho20}).
\end{rmk}

We have the following key result, which essentially follows from \cite{EL21}.

\begin{prop}\label{prop Koszul dual}
    If $W$ is topologically simple, then there is an $\Ainf$-equivalence
    \[\LFuk(S\La) \xrightarrow{\sim} \Prop(\WFuk^+(W, \La)^{op}). \]
\end{prop}
\begin{proof}
    Let $\cA_\La$ be the Chekanov-Eliashberg DG-category of $\La$ with based loop space coefficients (where the differential only involves discs with no interior puncture).
    According to Proposition \ref{prop aug computes augmentations of the algebra}, there is an equivalence $\LFuk(S\La) \hookrightarrow \Prop \left( \cA_\La \right)$.
    Now denote by $\WFuk(D_{\La})$ the full subcategory of $\WFuk(W, \La)$ whose objects are supported on $D_{\La}$.
    Since $W$ is assumed to be topologically simple, we can apply results of \cite{EL21} to compare $\cA_\La$ and $\WFuk(D_{\La})$ (see Remark \ref{rmk topologically simple}).
    According to \cite[Conjecture 3]{EL21} (proved in \cite[Section B.3]{EL21} and \cite{AE21}), 
    there is an $\Ainf$-equivalence $\WFuk(D_{\La})^{op} \xrightarrow{\sim} \cA_\La$.
    Therefore we get an equivalence $\LFuk(S\La) \xrightarrow{\sim} \Prop \left( \WFuk(D_{\La})^{op} \right)$.
    The result follows since $D_{\La}$ split-generates $\WFuk^+(W, \La)$.
\end{proof}

\paragraph{Equivariance}

In the following, we assume that $V$ is the total space of a prequantization bundle and that the action spectrum of $\La$ is contained in $\Z \cup (\R \setminus \Q)$.
Let $G$ be a discrete subgroup of $\R$ containing $\Z$.
We explain why the prequantization bundle setting induces a $(\vec G / \Z)$-action on $\Mod(\WFuk^+(W, G \cdot \La)^{op})$ so that the embedding of Proposition $\ref{prop Koszul dual}$ is $\vec G$-equivariant.
As in subsection \ref{section G-structure on augmentation category}, we use the Fukaya category of the symplectization $SV$.

By definition, there is an inclusion $\WFuk^+(W, \La) \subset \WFuk(SV, \La)$. 
Moreover, the equivalence of Proposition \ref{prop Koszul dual} is the functor
\[\LFuk(S(G \cdot \La)) \to \Prop \left( \WFuk^+(W, G \cdot \La)^{op} \right), \, (\La_i^{a-\eta}, \varepsilon) \mapsto \hom_{\RFuk(SV, G \cdot \La)^{op}} \left( -, (S \La_i^{a-\eta}, \varepsilon) \right). \]
Recall that we defined (in subsection \ref{section G-structure on augmentation category}) a $(\vec G / \Z)$-action on $\Mod(\RFuk(SV, G \cdot \La)^{op})$, and that the $(\vec G / \Z)$-action on $\LFuk(S(G \cdot \La))$ is induced by the embedding $\LFuk(S(G \cdot \La)) \hookrightarrow  \Mod(\RFuk(SV, G \cdot \La)^{op})$.
On the other hand, the induction map (see \cite[Section A.1]{GPS3}) induced by $\WFuk^+(W, G \cdot \La)^{op} \hookrightarrow \RFuk(SV, G \cdot \La)^{op}$ gives an embedding 
\[\Mod(\WFuk^+(W, G \cdot \La)^{op}) \hookrightarrow \Mod(\RFuk(SV, G \cdot \La)^{op}). \]
Since $\Z$ preserves this subcategory, we get a $(\vec G / \Z)$-action on $\Mod(\WFuk^+(W, G \cdot \La)^{op})$ satisfying the following.

\begin{prop}\label{prop equivariance of Aug in Mod}
    Assume that $W$ is topologically simple.
    Then Proposition \ref{prop Koszul dual} defines $\vec G$-equivariant $\Ainf$-equivalences $\LFuk(S(G \cdot \La)) \xrightarrow{\sim}  \Prop(\WFuk^+(W, G \cdot \La)^{op})$ for every discrete subgroup $G$ of $\R$ containing $\Z$. 
    Moreover, these are compatible with the functors 
    \[\LFuk(S(G \cdot \La)) \hookrightarrow \LFuk(S(G' \cdot \La)), \quad \Prop \left( \WFuk^+(W, G \cdot \La)^{op} \right) \hookrightarrow \Prop \left( \WFuk^+(W, G' \cdot \La)^{op} \right) \]
    when $G \subset G'$ (the leftmost arrow is the one of Proposition \ref{prop stop removal map for Aug}, while the rightmost arrow is the pullback along stop removal). 
\end{prop}

\begin{rmk}\label{rmk relation with partially wrapped fukaya}
    Exactly as for $\RFuk(SV, G \cdot \La)^{op}$, there is a monoidal functor from $\vec G$ to $\Aut(\WFuk(W, G \cdot \La)^{op})$ such that $c \in G$ is sent to (the functor induced by) $\varphi_{X_H}^c$ (where $H = e^t$ on $S^+V$), and the morphism in $\hom_{\vec{G}}(c_0, c_1)$ is sent to the continuation map $\varphi_{X_H}^{c_0} \Rightarrow \varphi_{X_H}^{c_1}$ induced by the positive isotopy $(\varphi_{X_H}^t)_{c_0 \leq t \leq c_1}$.
    Moreover, the transfer functor (see \cite{Leg21}) 
    \[b : \WFuk(W, G \cdot \La)^{op} \to \RFuk(SV, G \cdot \La)^{op}, \quad L \mapsto (S \partial_{\infty} L, \varepsilon_L) \]
    (where $\varepsilon_L$ is the augmentation of $\cA(\partial_{\infty}L)$ induced by $L$) is $\vec G$-equivariant.
    Since the induction map $\Mod(\WFuk^+(W, G \cdot \La)^{op}) \hookrightarrow \Mod(\WFuk(W, G \cdot \La)^{op})$ is the composition 
    \[\Mod(\WFuk^+(W, G \cdot \La)^{op}) \hookrightarrow \Mod(\RFuk(SV, G \cdot \La)^{op}) \xrightarrow{b^*} \Mod(\WFuk(W, G \cdot \La)^{op}), \]
    it is $\vec G$-equivariant for the $\vec G$-structure we defined on $\Mod(\WFuk^+(W, G \cdot \La)^{op})$.
\end{rmk}

\section{Legendrian lifts and $S^1$-gradings}\label{section precompleted Fukaya category}

We fix an admissible Legendrian $\La$ in an admissible prequantization bundle $(V, \alpha^\circ) \xrightarrow{\pi} (\base, \omega)$ (see Definition \ref{definition admissible}). 
We write $\cFuk(\pi_{|\La})$ for the Fukaya $\cAinf$-category over the Novikov ring $\Nov$ (see \cite{AJ10} or \cite[Section 3]{Fuk17}). An object in $\cFuk(\pi_{|\La})$ is a Lagrangian immersion $\pi_{|\La_j}$, where $\La_j$ is a connected component of $\La$, together with a local system $M_j \in \Prop(C_{-*}(\Omega \La_j))$.
We often forget the local system from the notation of an object.

Here we explain that $\La \subset V$ (i.e. the choice of a Legendrian lift of the Lagrangian immersion $\pi_{|\La}$) allows us to define a $\cAinf$-category $\cFuk^{\circ}(\pi_{|\La})$ enriched over $S^1$-graded $\Z[\R_{\geq 0}]$-modules (i.e. over $\Mod(\vec \R^{op} / \Z)$), with the property that $\widehat{\cFuk^{\circ}(\pi_{|\La})} = \cFuk(\pi_{|\La})$ (see Lemma-Definition \ref{lemma-definition completion of category} for the notation).

\vspace{2mm}

Given local systems $M_0 \in \Prop(C_{-*}(\Omega \La_i))$ and $M_1 \in \Prop(C_{-*}(\Omega \La_j))$, we have
\begin{align*}
        &\hom_{\cFuk(\pi_{|\La})} \left( \pi_{|\La_i}, \pi_{|\La_j} \right) \\
        & = \left\{
        \begin{array}{ll}
            \Nov \underset{\Z}{\otimes} \left( \bigoplus\limits_{(x^0, x^1) \in \La_i \underset{\base}{\times} \La_j} \hom_{\Z}(M_0, M_1) \right) & \text{if } i \ne j \\
            \Nov \underset{\Z}{\otimes} \left( \hom_{\Mod(C_{-*}(\Omega \La_j))}(M_0, M_1) \bigoplus \bigoplus\limits_{(x^0, x^1) \in \La_j \underset{\base}{\times} \La_j, \, x^0 \ne x^1} \hom_{\Z}(M_0, M_1) \right) & \text{if } i=j.
        \end{array}
        \right. 
\end{align*}
In particular, $\hom_{\cFuk(\pi_{|\La})}(\pi_{|\La_i}, \pi_{|\La_j})$ is always of the form $\Nov \otimes \, R (\pi_{|\La_i}, \pi_{|\La_j})$.

\begin{rmk}
    We expect the Fukaya category to be defined for more general local systems (i.e. not necessarily finite dimensional), as in \cite{BDHO24}.
\end{rmk}

\begin{definition}\label{definition length map}
    We define a map $\ell_{\La} : \bigsqcup_{\La_i, \La_j} R(\pi_{|\La_i}, \pi_{|\La_j}) \longrightarrow S^1$ which is $0$ on the space $\hom_{\Mod(C_{-*}(\Omega \La_j))}(M_0, M_1)$ and, if $x$ is in a summand indexed by $(x^0, x^1) \in \La_i \underset{\base}{\times} \La_j$, then $\ell_{\La}(x)$ is the length of a $\alpha^{\circ}$-Reeb chord from $x^0$ to $x^1$ (the lengths of two $\alpha^{\circ}$-Reeb chords with same endpoints differ by an integer).  
\end{definition}

\begin{lemma}\label{lemma grading on Fuk}
    If $\langle \mu_{\cFuk(\pi_{|\La})}^d (x_d, \dots, x_1), t^c \otimes x_0 \rangle \ne 0$, then $[c] = \ell_{\La} (x_0) - \sum_{k=1}^d \ell_{\La} (x_k)$.
\end{lemma}
\begin{proof}
    Let $u: \Delta \to \base$ be a $J$-holomorphic map contributing to the non-zero term.
    Let $\nu : \partial \Delta \to V$ be the unique lift of $u_{\partial \Delta}$ with values in $\La$.
    Choose $\ell_0, \ell_1, \dots, \ell_d \in \R_{\geq 0}$ such that
    \[[\ell_0] = \ell_{\La} (x_0), \quad [\ell_k] = \ell_{\La} (x_k) \text{ for } k \in \{1, \dots, d\}. \]
    Lemma \ref{lemma relation on boundary lift} gives
    \[ [c] = \left[ \int_{\Delta} u^* \omega \right] = \left[ \ell_0 - \sum_{k=1}^d \ell_k \right] = \ell_{\La} (x_0) - \sum_{k=1}^d \ell_{\La} (x_k). \]
\end{proof}

Lemma \ref{lemma grading on Fuk} allows us to introduce the following.

\begin{definition}\label{definition graded Fuk}
    We define a $\cAinf$-category $\cFuk^{\circ}(\pi_{|\La})$ enriched over $S^1$-graded $\Z[\R_{\geq 0}]$-modules as follows:
    \begin{enumerate}
        \item the objects are the same as in $\cFuk(\pi_{| \La})$,
        \item given two objects $\pi_{|\La_i}$ and $\pi_{|\La_j}$, $\hom_{\cFuk^{\circ}(\pi_{|\La})}(\pi_{|\La_i}, \pi_{|\La_j})$ is the tensor product (over $\Z$) of $\module$ and $R (\pi_{|\La_i}, \pi_{|\La_j})$ (where $\module = \bigoplus_{\theta \in S^1} t^{\theta}\Z[[t]]$, see Section \ref{section completion}), with $S^1$-grading
        \[\hom_{\cFuk^{\circ}(\pi_{|\La})}(\pi_{|\La_i}, \pi_{|\La_j})_{\theta} = \{ \sum_k \lambda_k t^{c_k} \otimes x_k \mid x_k \in R(\pi_{|\La_i}, \pi_{|\La_j}), \, [c_k] = \theta + \ell_{\La} (x_k) \} \]
        \item the operations are the same as in $\cFuk(\pi_{|\La})$ (these are well defined according to Lemma \ref{lemma grading on Fuk}).
    \end{enumerate}
\end{definition}

\begin{lemma}\label{lemma completion of graded Fuk recovers Fuk}
    We have $\cFuk(\pi_{|\La}) = \widehat{\cFuk^{\circ}(\pi_{|\La})} = \Nov \underset{\Z[\R_{\geq 0}]}{\otimes} \cFuk^{\circ}(\pi_{|\La})$ (see Lemma-Definition \ref{lemma-definition completion of category} for the notation).
\end{lemma}
\begin{proof}
    This follows from the computation $\widehat \module = \Nov$ done in Section \ref{section completion}.
\end{proof}

\begin{lemma}\label{lemma recovering Fuk from approximations}
    We have $\cFuk^{\circ}(\pi_{|\La}) = \overline{\cFuk^{\circ}(\pi_{|\La})_{| \vec \Q / \Z}}$ (see Definition \ref{definition from Q to R} for the notation).
\end{lemma}
\begin{proof}
    Recall that, by definition, $\cFuk^{\circ}(\pi_{|\La})$ is enriched over $S^1$-graded $\Z[\R_{\geq 0}]$-modules or, equivalently (see Lemma \ref{lem: poset module}), over $(\vec \R^{op} / \Z)$-modules.
    According to Lemma $\ref{lemma from Q to R}$, we need to show that $\cFuk^{\circ}(\pi_{|\La})$ is actually enriched over $\Per(\vec \R^{op} / \Z)$.
    For $\theta \in S^1$, let 
    \[L_{\theta} := \max \{ \ell \in [0, 1) \mid \exists x, \, [\ell] = \theta + \ell_{\La}(x) \} \in [0, 1). \]
    Definition of the $S^1$-grading on $\cFuk^{\circ}(\pi_{|\La})$ implies that, for every $\varepsilon \in [0, 1- L_{\theta})$, there is an isomorphism
    \[\hom_{\cFuk^{\circ}(\pi_{|\La})} (\pi_{|\La_i}, \pi_{|\La_j})_{\theta} \xrightarrow{\sim} \hom_{\cFuk^{\circ}(\pi_{|\La})} (\pi_{|\La_i}, \pi_{|\La_j})_{\theta + [\varepsilon]}, \quad t^c \otimes x \mapsto t^{c+\varepsilon} \otimes x. \]
    The result follows.
\end{proof}

We will be interested in the $(\vec{G} / \Z)$-category $\cFuk^{\circ}(\pi_{|\La})_{| \vec G / \Z} \# (\vec{G} / \Z)$, when $\Z \subset G \subset \R$ is a discrete subgroup (see Section \ref{section change of enrichement and reconstruction}).

\begin{lemma}\label{lemma recover from exhaustion}
    We have $\left( \varinjlim \cFuk^{\circ}(\pi_{|\La})_{| \frac{1}{n} \vec \Z / \Z} \# (\frac{1}{n} \vec \Z / \Z) \right) [\vec \Q / \Z] = \cFuk^{\circ}(\pi_{|\La})_{| \vec \Q / \Z}$.
\end{lemma}
\begin{proof}
    This follows directly from Proposition \ref{prop reconstruction}.
\end{proof}

In the case of embedded Lagrangians, the structure discussed above simplifies dramatically. 

\begin{prop}\label{prop embedded case}
    If $\pi_{| \La}$ is an embedding, then there is an isomorphism of categories enriched over $\vec \R^{op} / \Z$-modules
    \[\cFuk^{\circ}(\pi_{| \La}) \xrightarrow{\sim} \module \underset{\Z[[t]]}{\otimes} \cFuk^{\circ}(\pi_{| \La})_{| \vec \Z / \Z}, \quad t^c \underset{\Z}{\otimes} x \mapsto t^c \underset{\Z[[t]]}{\otimes} x. \]
    As a result, in this case, $\cFuk(\pi_{| \La}) = \Nov \underset{\Z[[t]]}{\otimes} \cFuk^{\circ}(\pi_{| \La})_{| \vec \Z / \Z}$.
\end{prop}
\begin{proof}
    Since $\ell_{\La} = 0$ (because $\pi_{| \La}$ is an embedding), the above map is a well defined isomorphism.
\end{proof}

\begin{remark}\label{rmk Fuk not determined by single approximation}
    One cannot expect $\cFuk(\pi_{| \La})$ and $\Nov \underset{\Z[[t]]}{\otimes} \cFuk^{\circ}(\pi_{| \La})_{| \vec \Z / \Z}$ to be isomorphic in general. 
    For example, consider the case where $V = S^1$ and $\La = \La_- \sqcup \La_+$ is the disjoint union of two points.
    Then $\La_-$ and $\La_+$ are isomorphic in $\cFuk(\pi_{| \La})$, but not in $\Nov \underset{\Z[[t]]}{\otimes} \cFuk^{\circ}(\pi_{| \La})_{| \vec \Z / \Z}$.

    Also observe that $\cFuk(\pi_{| \La})$ is in general not determined by $\cFuk^{\circ}(\pi_{| \La})_{| \vec \Z / \Z}$, as the following example shows.
    Consider the prequantization bundle $(\R / \Z) \times \C \xrightarrow{\pi} \C$ and, for $a \in (0,1)$, a connected Legendrian submanifold $\La_a$ such that 
    \begin{enumerate}
        \item $\La_a$ is a standard Legendrian unknot, i.e. $\pi_{|\La_a}$ is an ``$\infty$ sign" exact Lagrangian immersion (see for instance \cite[Example 13.12]{AJ10}),
        \item $\pi_{|\La_a}^* \lambda_\C$ admits a primitive $f_a$ valued in $(0, 1)$, and there are $x_\pm \in \La_a$ such that $\pi(x_-)=\pi(x_+)$ and $f_a(x_+)-f_a(x_-)=a$.
    \end{enumerate}
    Inspecting \cite[Example 13.12]{AJ10} (note that they are using shifted degrees), one can see that $\Fuk(\pi_{|\La_a})$ is generated over $\Nov$ by two cohomological generators $u$, $v$ of degree $0$, $1$ respectively, and two self-intersection generators $(x_-,x_+)$, $(x_+, x_-)$ of degree $2$, $-1$ respectively, and that 
    \[\mu_0=\mu_1(u)=\mu_1(x_-,x_+)=\mu_1(x_+,x_-)=0, \quad \mu_1(v)=t^a \otimes (x_-, x_+).\]
    We deduce that, if $0<a<b<1$, then 
    \[\cFuk(\pi_{|\La_a}) \ne \cFuk(\pi_{|\La_b}), \quad \cFuk^\circ(\pi_{|\La_a})_{| \vec \Z / \Z} = \cFuk^\circ(\pi_{|\La_b})_{| \vec \Z / \Z}. \]
    
\end{remark}

We finish this Section with a description of the category $\cFuk^{\circ}(\pi_{|\La})_{| \vec G / \Z} \# (\vec{G} / \Z)$.

\begin{lemma}\label{lemma Fukaya category of prequantization bundle}
    The $(\vec{G} / \Z)$-category $\cFuk^{\circ}(\pi_{|\La})_{| \vec G / \Z} \# (\vec{G} / \Z)$ can be explicitly described as follows:
    \begin{enumerate}
        \item its objects are the Legendrians $\La_i^a := \varphi_{\Reeb}^a (\La_i)$, where $a \in G / \Z$, and $\La_i$ is a connected component of $\La$,
        \item given two different connected components $\La_i$ and $\La_j$,
        \[\hom(\La_i^a, \La_j^b) = \left\{ \sum_k \lambda_k t^{c_k} \otimes x_k \mid x_k \in R(\pi_{|\La_i}, \pi_{|\La_j}), \, c_k \geq 0, \, \, [c_k] = \ell_{\La} (x_k) + [b-a] \right\} \]
        \item the operations are the restrictions of the operations in $\cFuk^{\circ}(\pi_{|\La})$.
    \end{enumerate}
\end{lemma}
\begin{proof}
    This is just unwrapping Definition \ref{def: unorbit}.
\end{proof}

\begin{rmk}\label{rmk filtration on Fuk}
    An object $c$ in $\vec G$ acts on objects by $\La_i^a \mapsto \La_i^{a+c}$, and the morphism in $\hom_{\vec G} (c_0, c_1)$ sends $t^c \otimes x_0$ to $t^{c+c_1-c_0} \otimes x_0$.
    In particular we have a $G_+$-filtration on $\hom(\La_i^a, \La_j^b)$ given by (see Remark \ref{rmk filtration})
    \begin{align*}
    \cF^{\geq c} \hom(\La_i^a, \La_j^b) & = \{ \sum_k \lambda_k t^{c+c_k} \otimes x_k \mid x_k \in R(\pi_{|\La_i}, \pi_{|\La_j}), \, c_k \geq 0, \, \, [c_k] = [b-c-a] + \ell_{\La} (x_k) \} \\
    & = \{ \sum_k \lambda_k t^{c_k} \otimes x_k \mid x_k \in R(\pi_{|\La_i}, \pi_{|\La_j}), \, c_k \geq c, \, \, [c_k] = [b-a] + \ell_{\La} (x_k) \}. \\
    \end{align*}
\end{rmk}

\section{Comparison}
\label{comparison section} 

In all this Section, we assume that $\La$ is an admissible Legendrian in an admissible prequantization bundle $V \xrightarrow{\pi} \base$ (see Definition \ref{definition admissible}).
Our goal is to compare $\cLFuk(S(G \cdot \La))$ (defined in Section \ref{section augmentation category}) and $\cFuk^{\circ}(\pi_{|\La})_{| \vec G / \Z}$ (defined in Section \ref{section precompleted Fukaya category}) when $G$ is a discrete subgroup of $\R$ containing $\Z$. 

\subsection{Functor from Lagrangian correspondence}
\label{section functor from lagrangian correspondence}

In the following, we use $\alpha^{\circ}$ to identify $(SV, \Omega)$ and $(\R_t \times V, d(e^t \alpha^{\circ}))$.
Observe that $\base$ is a symplectic reduction of $S V$ under the $S^1$-action with moment map $\mu : (t,x) \mapsto e^t$.
Explicitly, the map $(t,x) \mapsto \pi(x)$ induces an isomorphism $\mu^{-1} (1) / S^1 \xrightarrow{\sim} \base$.
Therefore we have a Lagrangian correspondence 
\[\Gamma := \left\{ \left( (0, x), \pi(x) \right) \mid x \in V \right\} \subset SV \times \base . \]

Let $G$ be a discrete subgroup of $\R$ containing $\Z$. 
We set
\[\cD := \cFuk^{\circ}(\pi_{|\La}) \text{ and } \cD_G := \cFuk^{\circ}(\pi_{|\La})_{| \vec G / \Z} \# (\vec{G} / \Z). \]
According to Section \ref{section definition augmentation category} (we use the same notations as there), there is a $\cAinf$-category $\cO_G$ with  
\[\Ob(\cO_G) = \{ \La_j^{b-\eta}(n) \mid b \in G / \Z, \, \La_j \in \pi_0(\La), \, n \in \Z_{\geq 0} \}, \]
and a set of $\cO_G$-morphisms $M_G$, such that $\cLFuk(S(G \cdot \La))$ is the full subcategory of $\cO_G \left[ M_G^{-1} \right]$ whose objects are the Legendrians $\La_j^{b-\eta}(0) = \varphi_{\Reeb}^{b-\eta} (\La_i)$.
The morphisms in $\cO_G$ are $\alpha$-Reeb chords, where $\alpha$ is a contact form on $(V, \xi)$ for which $\La$ is chord generic. 
Without loss of generality, we will assume the following:
\begin{enumerate}
    \item there is an increasing sequence $(0=\varepsilon_0, \varepsilon_1, \dots)$ converging to $\eta$ such that $\La_j^{b-\eta}(n) = \varphi_{\Reeb}^{b-\eta+\varepsilon_n} (\La_j)$,
    \item $\alpha = e^{\delta H \circ \pi} \alpha^{\circ}$ for some $\delta > 0$ and $H : \base \to \R$ such that $\Crit (H \circ \pi)_{| \La} = \Crit(H \circ \pi) \cap \La$ and the latter contains the double points of $\pi_{| \La}$.
\end{enumerate}

\begin{rmk}\label{rmk Reeb chords}
    The assumptions above imply that any $\alpha$-Reeb chord in $\cO_G$ is an $\alpha^{\circ}$-Reeb chord, and any $\alpha^{\circ}$-Reeb chord that is non-degenerate is an $\alpha$-Reeb chord. 
\end{rmk}

We will construct a bimodule (later to be proven to be a functor) by counting pseudoholomorphic quilts, as in \cite{MWW18, Mau, Ganatra-duality, Fuk17}.  
Especially similar considerations appeared in the construction of the K\"unneth functor in \cite{GPS2}. 

\begin{defin}\label{definition bimodule B} 
    We define a $(\cO_G^{op}, \cD^{op})$-bimodule $\cB$ as follows:
    \begin{enumerate}
    \item given local systems $M_0 \in \Prop(C_{-*}(\Omega \La_i))$ and $M_1 \in \Prop(C_{-*}(\Omega \La_j))$,
    \begin{align*}
        & \cB \left( \La_j^{b-\eta}(n), \pi_{|\La_i} \right) \\
        & := \left\{
        \begin{array}{ll}
            \module \underset{\Z}{\otimes} \left( \bigoplus\limits_{(x^0, x^1) \in \Gamma \underset{SV \times \base}{\times} ( S \La_j^{b-\eta}(n) \times \La_i )} \hom_{\Z}(M_0, M_1) \right) & \text{if } i \ne j \\
            \module \underset{\Z}{\otimes} \left( \hom_{\Mod(C_{-*}(\Omega \La_j))}(M_0, M_1) \bigoplus \bigoplus\limits_{\substack{(x^0, x^1) \in \Gamma \underset{SV \times \base}{\times} (S \La_j^{b-\eta}(n) \times \La_i) \\ x^0 \ne x^1}} \hom_{\Z}(M_0, M_1) \right) & \text{if } i=j.
        \end{array}
        \right. 
\end{align*}
    (where $\module = \bigoplus_{\theta \in S^1} t^{\theta}\Z[[t]]$ has been introduced in Section \ref{section completion})
    \item the operations $\mu_{\cB}$ are defined by counting pseudo-holomorphic quilted strips $(w_-, u_+)$ in $(SV, J_{\alpha}) \times (\base, J)$ as in Figure \ref{figure bimodule}, weighted by the area $c = \int w_-^* \Omega + \int u_+^* \omega$.
    \end{enumerate}
\end{defin}

\begin{figure}
    \centering
    \def\svgwidth{0.5\textwidth}
    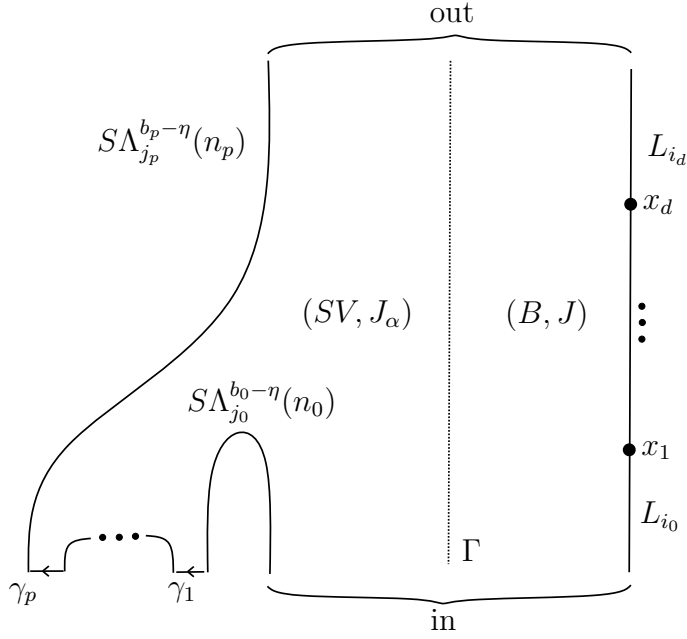		
    \caption{Quilted strip $(w_-, u_+)$ contributing to $\langle \mu_{\cB} (\gamma_p, \dots, \gamma_1, \inp, x_1, \dots, x_d), t^c \otimes \outp \rangle$, where $c = \int w_-^* \Omega + \int u_+^* \omega$.}	
    \label{figure bimodule}
\end{figure}

\begin{rmk}\label{rmk no reeb chord at + infinity}
    It is important (in order to define the bimodule $\cB$) that we always consider curves $w_-$ asymptotic to Reeb chords in the negative end of $SV$, so that $\int w_-^* \Omega$ is finite. 
\end{rmk}

\begin{rmk}\label{rmk compactness}
    Compactness of the moduli space used to define the bimodule $\cB$ (i.e. the moduli space of curves in Figure \ref{figure bimodule}) is guaranteed by Gromov compactness for the component in $\base$, and by compactness in symplectic field theory as in \cite{BEHW03} for the component in $SV$.  
\end{rmk}

\begin{rmk}\label{rmk intersection with Lagrangian correspondence}
    There is a diffeomorphism $\La_i \underset{\base}{\times} \La_j \xrightarrow{\sim} \Gamma \underset{SV \times \base}{\times} ( S \La_j^{b-\eta}(n) \times \La_i )$ given by the maps
    \[\left\{
    \begin{array}{ll}
    \La_i \underset{\base}{\times} \La_j \to \Gamma, & (x^0, x^1) \mapsto ((0, \varphi_{\Reeb}^{b-\eta+\varepsilon_n}(x^1)), \pi (x^1)) \\
    \La_i \underset{\base}{\times} \La_j \to S \La_j^{b-\eta}(n) \times \La_i, & (x^0, x^1) \mapsto ((0, \varphi_{\Reeb}^{b-\eta+\varepsilon_n}(x^1)), x^0).
    \end{array}
    \right.\]
    In particular, $\cB \left( \La_j^{b-\eta}(n), \pi_{|\La_i} \right) = \module \otimes R(\pi_{|\La_i}, \pi_{|\La_j})$ (see Section \ref{section precompleted Fukaya category}).
\end{rmk}

\begin{lemma}\label{lemma grading on B}
    If $\langle \mu_{\cB} (\gamma_p, \dots, \gamma_1, \inp, x_1, \dots, x_d), t^c \otimes \outp \rangle \ne 0$, then
    \[[c] = (\ell_{\La} (\outp) + [b_p+\varepsilon_{n_p}]) - (\ell_{\La} (\inp) + [b_0+\varepsilon_{n_0}]) - \sum_{k=1}^d \ell_{\La} (x_k). \]
    In particular we can consider the $S^1$-grading 
    \[\cB(\La_j^{b-\eta}(n), \pi_{|\La_i})_{\theta} := \{ \sum_k \lambda_k t^{c_k} \otimes x_k \mid x_k \in R(\pi_{|\La_i}, \pi_{|\La_j}), \, [c_k] = \theta + \ell_{\La} (x_k) + [b + \varepsilon_n] \}, \] 
    and $\cB$ gives rise to a $(\cO_G^{op}, \cD_G^{op})$-bimodule $\cB_G$ such that $\cB_G (\La_j^{b-\eta}(n), \La_i^a) = \cB(\La_j^{b-\eta}(n), \pi_{|\La_i})_{[-a]}$.
\end{lemma}
\begin{proof}
    Let $(w_- = (\sigma_-, v_-): \Delta_- \to \R \times V, u_+ : \Delta_+ \to \base)$ be a pseudo-holomorphic quilt as in Figure \ref{figure bimodule}.
    Let 
    \[\partial \Delta_+ \supset \cS_+ \xrightarrow[\rho]{\sim} \cS_- \subset \partial \Delta_- \]
    be the seam of the quilt (observe that $\rho$ reverses the orientation). 
    Denote by $\nu : \partial \Delta_+ \to V$ the unique lift of $(u_+)_{|\partial \Delta_+}$ such that $\nu = v_- \circ \rho$ on $\cS_+$ and $\nu$ takes values in $\La^{-\eta}$ on $(\partial \Delta_+) \setminus \cS_+$.
    Choose $\ell_{\outp}, \ell_{\inp}, \ell_1, \dots, \ell_d \in \R_{\geq 0}$ such that
    \[[\ell_{\outp}] = \ell_{\La} (\outp), \quad [\ell_{\inp}] = \ell_{\La} (\inp), \quad [\ell_k] = \ell_{\La} (x_k) \text{ for } k \in \{1, \dots, d\}. \]
    Then we have
    \begin{align*}
    [c] & = \left[ \int_{\Delta_-} w_-^* \Omega + \int_{\Delta_+} u_+^* \omega \right] \\
        & = \left[ \int_{\Delta_-} (\sigma_-, v_-)^* d(e^t \alpha^{\circ}) + \int_{\Delta_+} u_+^* \omega \right] \\
        & = \left[ \int_{\cS_-} v_-^* \alpha^{\circ} + \int_{\Delta_+} u_+^* \omega \right] \\ 
        & = \left[ (\ell_{\outp} + b_p+\varepsilon_{n_p}) - (\ell_{\inp} + b_0+\varepsilon_{n_0}) - \sum_{k=1}^d \ell_k \right] \\ 
        & = (\ell_{\La} (\outp) + [b_p+\varepsilon_{n_p}]) - (\ell_{\La} (\inp) + [b_0+\varepsilon_{n_0}]) - \sum_{k=1}^d \ell_{\La} (x_k).
    \end{align*}
    where the penultimate equality follows from Lemma \ref{lemma relation on boundary lift}.
\end{proof}

\begin{definition}\label{definition main functor}
    By localizing $\cB_G$ at $M_G$, we get a $(\cO_G[M_G^{-1}]^{op}, \cD_G^{op})$-bimodule $_{M_G^{-1}} \cB_G$, and the latter is equivalent to the data of a $\cAinf$-functor
    \[\Psi : \cO_G[M_G^{-1}] \to \cMod (\cD_G), \quad \La_j^{b-\eta}(n) \mapsto \, _{M_G^{-1}} \cB_G \left( \La_j^{b-\eta}(n), - \right). \]
\end{definition}

\subsection{Representability}

Let $\La_j^{b-\eta}(n)$ be an object in $\cO_G$.
We want to prove that $\Psi(\La_j^{b-\eta}(n))$ is equivalent to $\La_j^b$ in $\cMod(\cD_G)$ (we identify an object in $\cD_G$ with its image in $\cMod(\cD_G)$ under the Yoneda embedding).

\begin{lemma}\label{lemma bijection moduli spaces for representability}
    If $\alpha = e^{\delta H \circ \pi} \alpha^{\circ}$ is close enough to $\alpha^{\circ}$ (i.e. if $\delta$ is sufficiently small), then there is a bijection between the moduli spaces of Figure \ref{figure bijection moduli spaces for representability}.
    Moreover, this bijection respects symplectic area.
\end{lemma}
\begin{proof}
    This follows from transversality of the moduli spaces, since here we do not consider any Reeb chord asymptotic.
\end{proof}

\begin{figure}[h]
    \centering
    \def\svgwidth{1\textwidth}
    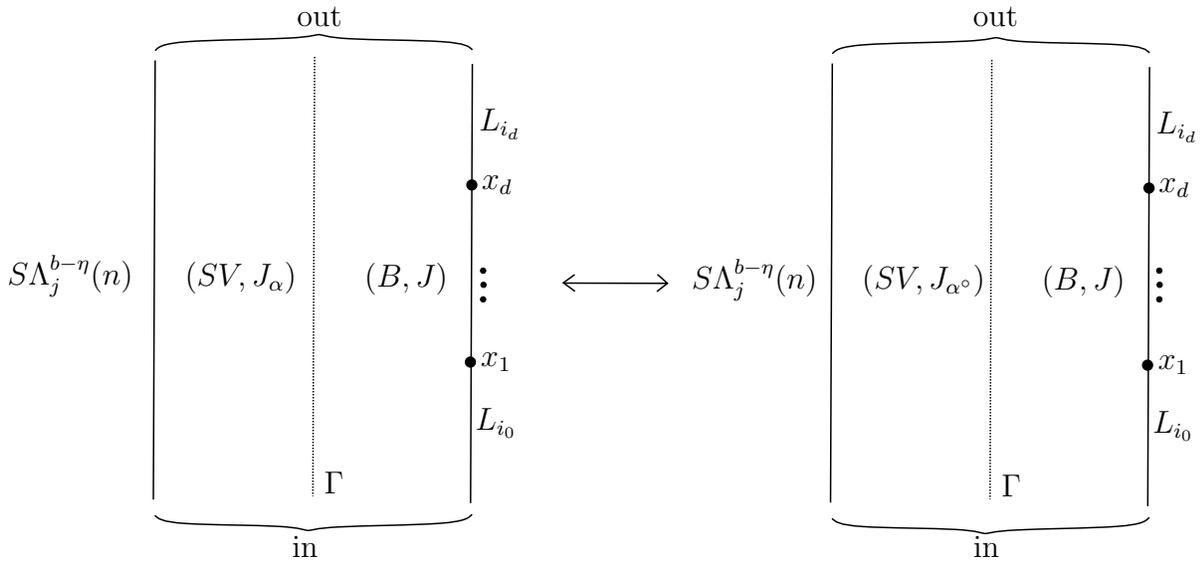		
    \caption{Identification of the curves used to define $\cB_G(S\La_j^{b-\eta}(n), -)$ (on the left) with some quilted strips (on the right)}
    \label{figure bijection moduli spaces for representability}
\end{figure}

\begin{lemma}\label{lemma bijection moduli spaces for modules}
    There is a bijection between the moduli spaces of curves of Figure \ref{figure bijection moduli spaces for modules}. Moreover, this bijection respects symplectic area.
\end{lemma}
\begin{proof}
    If $(w_-=(\sigma_-,v_-), u_+)$ is a quilt as on the left of Figure \ref{figure bijection moduli spaces for modules}, then $u_- := \pi \circ v_-$ and $u_+$ coincide on the seam (because there the pair $(w_-, u_+)$ is in $\Gamma$), so $u_-$ and $u_+$ can be glued together to give a curve $u$ as on the right of Figure \ref{figure bijection moduli spaces for modules}.
    Moreover we have 
    \[\int u^* \omega = \int (\pi \circ v_-)^* \omega + \int u_+^* \omega = \int v_-^* (d  \alpha^{\circ}) + \int u_+^* \omega = \int w_-^* \Omega + \int u_+^* \omega \]
    where the last equality follows because $w_-$ maps one boundary component to $S \La_j^{b-\eta}(n)$ and the other one to $\Pi_{SV}(\Gamma)$. 

    Now let $u : \Delta \to \base$ be a $J$-holomorphic disk as on the right of Figure \ref{figure bijection moduli spaces for modules}. Let $u_- : S_- \to \base$ be the restriction of $u$ to the left part $S_- = [-1,0]_x \times \R_y$ of the domain. 
    Let $\nu$ be the unique section of $u_-^*V$ over $\{ -1 \} \times \R$ with values in $\La_j^{b-\eta}$.
    According to Proposition \ref{prop relation pseudo-holomorphic discs mixed case}, there exists a unique $J_{\alpha^{\circ}}$-holomorphic map $w_- = (\sigma_-, v_-) : S_- \to \R \times V$ such that
    \[\pi \circ v_- = u_-, \quad v_-(-1, y) = \nu(-1, y), \quad \sigma_-(0,y) = 0, \]
    and we have $w_-(x,y) \underset{y \to \pm \infty}{\longrightarrow} (0, \underset{y \to \pm \infty}{\lim} \nu(-1, y))$.
\end{proof}

\begin{figure}[h]
    \centering
    \def\svgwidth{1\textwidth}
    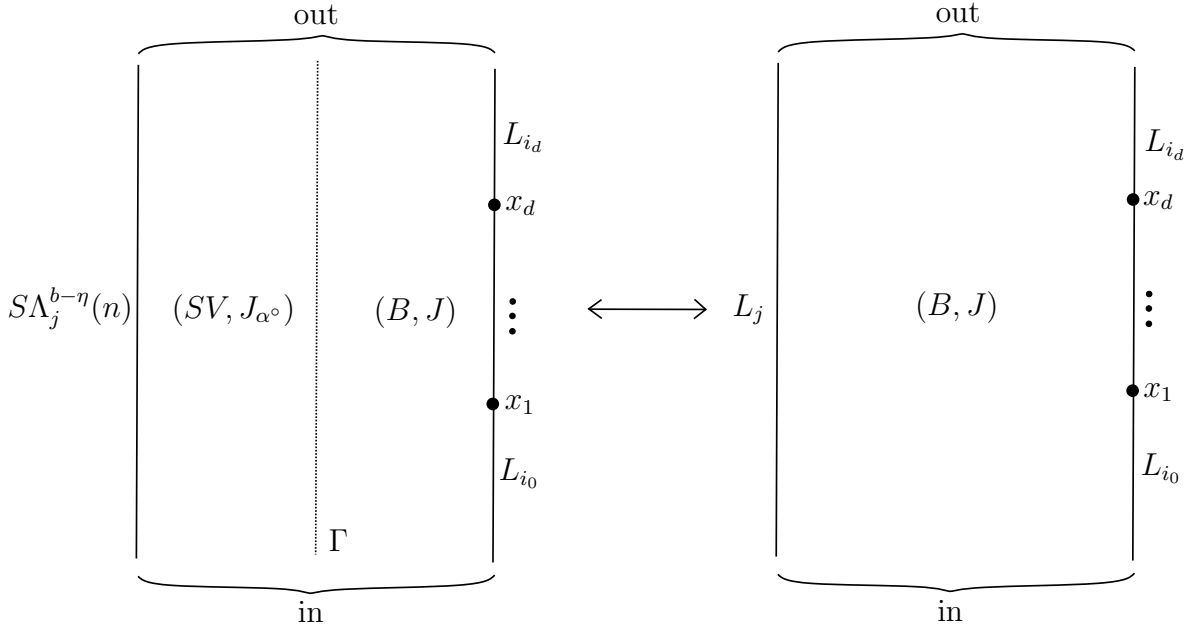		
    \caption{Identification of the curves used to define $\cD_G^{op}(\La_j^b, -) = \cD_G(- ,\La_j^b)$ (on the right) with some quilted strips (on the left)}
    \label{figure bijection moduli spaces for modules}
\end{figure}

\begin{lemma}\label{lemma bijection moduli spaces Morse-Bott}
    If $\alpha = e^{\delta H \circ \pi} \alpha^{\circ}$ is close enough to $\alpha^{\circ}$ (i.e. if $\delta$ is sufficiently small), then there is a bijection between the moduli spaces of Figure \ref{figure bijection moduli spaces Morse-Bott}, where $\gamma_j^n \in M_G \cap \cO_G(\La_j^{b-\eta}(n), \La_j^{b-\eta}(n+1))$ is the Reeb chord above the minimum of $(H \circ \pi)_{| \La_j}$.
    Moreover, this bijection respects symplectic area.
\end{lemma}
\begin{proof}
    Morse-Bott results in symplectic field theory \cite{Bou03, EES09} allow us to describe, for $\delta > 0$ small enough, rigid curves in $(SV, J_{\alpha^{\delta}})$ in term of curves in $(SV, J_{\alpha^{\circ}})$ and gradient flow lines in the manifold of Reeb chords.
    Our claim that, in our situation, the only possible rigid Morse-Bott configuration is the one of Figure \ref{figure bijection moduli spaces Morse-Bott}, relies on two main facts
    \begin{enumerate}
        \item we consider curves asymptotic to the Reeb chord $\gamma_j^n$ above the minimum of $(H \circ \pi)_{| \La_j}$,
        \item curves in $(SV, J_{\alpha^{\circ}})$ correspond, according to Proposition \ref{prop relation pseudo-holomorphic discs}, to curves in $(\base, J)$. 
    \end{enumerate}   
    The latter imply that curves in $(SV, J_{\alpha^{\circ}})$ asymptotic to $\gamma_j^n$ correspond to curves in $(\base, J)$ with one free marked point on the boundary.
    However, the only possibility for such a curve to contribute to a rigid Morse-Bott configuration is that it is constant.
    The result follows.
\end{proof}

\begin{figure}
    \centering
    \def\svgwidth{1\textwidth}
    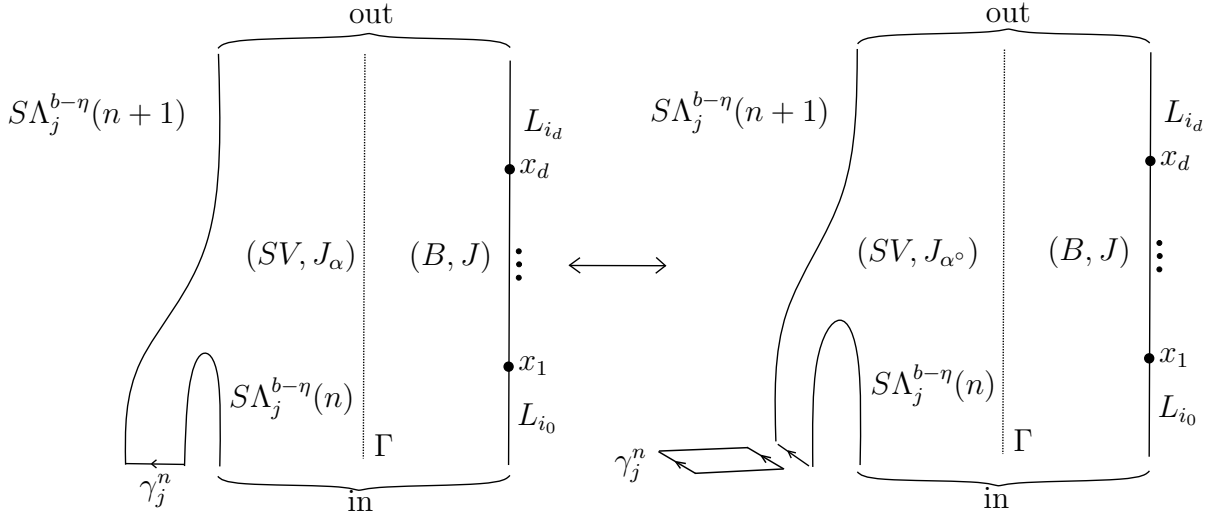		
    \caption{Morse-Bott degeneration (on the right) of the curves used to define the map $\mu_{\cB_G} (\gamma_j^n, -) : \cB_G(S\La_j^{b-\eta}(n), -) \to \cB_G(S\La_j^{b-\eta}(n+1), -)$ (on the left)}
    \label{figure bijection moduli spaces Morse-Bott}
\end{figure}

\begin{prop}\label{prop representability}
    We have the following commutative diagram
    \[\begin{tikzcd}
    \cD_G(-,\La_j^b) \ar[r, "\sim" above, "t^{\varepsilon_n}" below] \ar[d, equal] & \cB_G(\La_j^{b-\eta}(n), -) \arrow{d}{\mu_{\cB_G} (\gamma_j^n, -)} \\
    \cD_G(-,\La_j^b) \ar[r, "\sim" above, "t^{\varepsilon_{n+1}}" below] & \cB_G(\La_j^{b-\eta}(n+1), -).
    \end{tikzcd} \] 
    In particular, $\Psi(\La_j^{b-\eta}(n)) = \! _{ M_G^{-1}} \cB_G \left( \La_j^{b-\eta}(n), - \right)$ is equivalent to $\La_j^b$ in $\cMod(\cD_G)$.
\end{prop}
\begin{proof}
    Using Lemmas \ref{lemma bijection moduli spaces for representability} and \ref{lemma bijection moduli spaces for modules}, we get that the following map is an equivalence
    \[\cD_G(-, \La_j^b) \to \cB_G(\La_j^{b-\eta}(n), -) \quad t^c \otimes x \mapsto t^{c + \varepsilon_n} \otimes x. \]
    It remains to show the commutativity of the diagram.
    In other words, we have to show that $\mu_{\cB_G} (\gamma_j^n, -) : \cB_G(S\La_j^{b-\eta}(n), -) \to \cB_G(S\La_j^{b-\eta}(n+1), -)$ is multiplication by $t^{\varepsilon_{n+1} - \varepsilon_n}$. 
    According to Lemma \ref{lemma bijection moduli spaces Morse-Bott}, the curves contributing to $\mu_{\cB_G} (\gamma_j^n, -)$ correspond to curves as on the right of Figure \ref{figure bijection moduli spaces Morse-Bott}.
    Such a curve consists of a quilted strip $(w_-, u_+)$ of area $d$, together with a gradient flow line from $\gamma$ (the Reeb chord to which $w_-$ converges) to $\gamma_j^n$. Because $\gamma_j^n$ is a Reeb chord above the minimum of the Morse function, the only possibility for this configuration to be rigid is that $u_+$ is constant, and $w_-$ consists of a trivial cylinder on $\gamma$. In particular we have $d = \ell(\gamma) = \ell(\gamma_j^n)= \varepsilon_{n+1} - \varepsilon_n$.
    This concludes the proof.
\end{proof}

\subsection{Fully faithfulness} 

Recall that $\cO_G$ and $\cD_G$ are $\vec G / \Z$-categories and, as such, are enriched over $G_+$-filtered $\Z[t]$-modules (see Remarks \ref{rmk filtration} and \ref{rmk filtration on Fuk}).
In the following, we assume that $G$ is big enough (i.e. $G = \frac{1}{N} \Z$ with $N$ big enough) so that $\inf(G_+) < \ell_0 := \inf (\cS_\La)$ where $\cS_\La$ is the submonoid of $\R_{>0}$ generated by the set of areas of non-constant punctured pseudo-holomorphic discs in $\base$ with boundary on $\pi_\La$.

\begin{remark}\label{rmk assumption on G}
    The latter assumption on $G$ implies that any non-constant pseudo-holomorphic disk in $\base$ with boundary on $\pi_\La$ strictly increases the filtration of $\cD_G$.  
\end{remark}

\begin{lemma}\label{lemma curves contributing to graded map between modules}
    The following holds when $\alpha = e^{\delta H \circ \pi} \alpha^{\circ}$ is close enough to $\alpha^{\circ}$ (i.e. when $\delta$ is sufficiently small).
    If $\gamma$ is an $\alpha$-Reeb chord in $\cO_G(\La_i^{a-\eta}(m), \La_j^{b-\eta}(n))$, then the curves contributing to 
    \[\left[ \mu_{\cB_G} (\gamma, -) \right] \in \Gr \hom_{\cMod(\cD_G)} \left( \cB(\La_i^{a-\eta}(m), -), \cB(\La_j^{b-\eta}(n), -) \right) \]
    are the ones on the left, respectively on the right, of Figure \ref{figure curves contributing to graded map between modules} if $\gamma$ is non-degenerate, respectively degenerate as an $\alpha^{\circ}$-Reeb chord.
\end{lemma}
\begin{proof}
    First assume $\gamma$ is non-degenerate as an $\alpha^{\circ}$-Reeb chord. In this case, it follows from transversality of the moduli spaces that the curves on the left of Figure \ref{figure curves contributing to graded map between modules} are exactly the curves contributing to $\mu_{\cB_G} (\gamma, -)$ (not only its class in the associated graded).

    Assume now that $\gamma$ is degenerate as an $\alpha^{\circ}$-Reeb chord (observe that this implies $\La_i = \La_j$). In this case, Morse-Bott results in symplectic field theory (see \cite{Bou03, EES09}) allow us to describe, for $\delta > 0$ small enough, curves in $(SV, J_{\alpha^{\delta}})$ contributing to $\mu_{\cB_G} (\gamma, -)$ in term of curves in $(SV, J_{\alpha^{\circ}})$ and gradient flow lines in the manifold of Reeb chords.
    In general, we expect that there are more curves contributing to $\mu_{\cB_G} (\gamma, -)$ than the ones depicted on the right of Figure \ref{figure curves contributing to graded map between modules}.
    However, we claim that the latter curves are the only ones contributing to $[\mu_{\cB_G} (\gamma, -)]$.
    Indeed, the top level of any other Morse-Bott configuration contributing to $\mu_{\cB_G} (\gamma, -)$ is a quilt $(w_-, u_+)$, where $w_- = (\sigma_-, v_-)$ is a curve in $(SV, J_{\alpha^{\circ}})$ asymptotic to a Reeb chord $\gamma'$ with $\ell(\gamma') \geq \ell(\gamma)+\ell_0$. 
    By definition, such a quilt increases the filtration by 
    \[\int w_-^* \Omega + \int u_+^* \omega = \ell(\gamma') + \int (\pi \circ v_-)^* \omega + \int u_+^* \omega \geq \ell(\gamma) + \ell_0, \]
    where the first equality follows from Stokes theorem.
    The result follows.
\end{proof}

\begin{figure}[h]
    \centering
    \def\svgwidth{1\textwidth}
    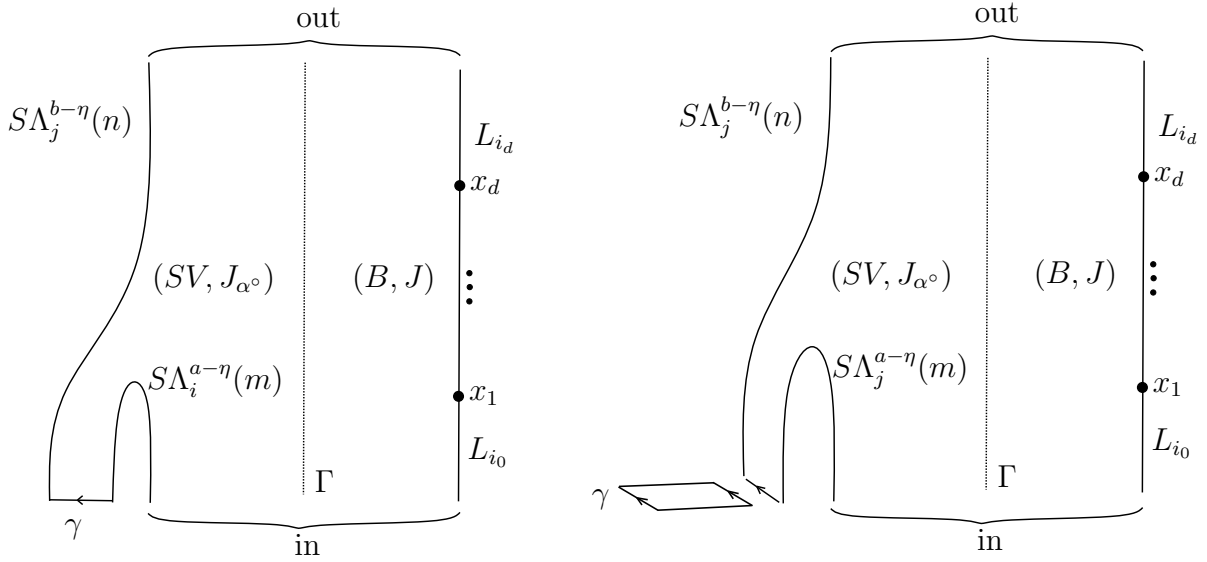		
    \caption{Curves contributing to $[\mu_{\cB_G} (\gamma, -)]$ if $\gamma$ is non-degenerate (on the left) or degenerate (on the right) as an $\alpha^{\circ}$-Reeb chord}
    \label{figure curves contributing to graded map between modules}
\end{figure}

\begin{lemma}\label{lemma bijection moduli spaces for map between modules}
    If $\gamma$ is an $\alpha$-Reeb chord in $\cO_G(\La_i^{a-\eta}(m), \La_j^{b-\eta}(n))$, then there is a bijection between the moduli spaces of curves depicted in Figure \ref{figure curves contributing to graded map between modules in the base} of area $c$ and the moduli space of curves depicted in Figure \ref{figure curves contributing to graded map between modules} of area $c + \ell(\gamma)$.
\end{lemma}
\begin{proof}
    This follows from Proposition \ref{prop relation pseudo-holomorphic discs mixed case} and Stokes theorem.
\end{proof}

\begin{figure}[h]
    \centering
    \def\svgwidth{1\textwidth}
    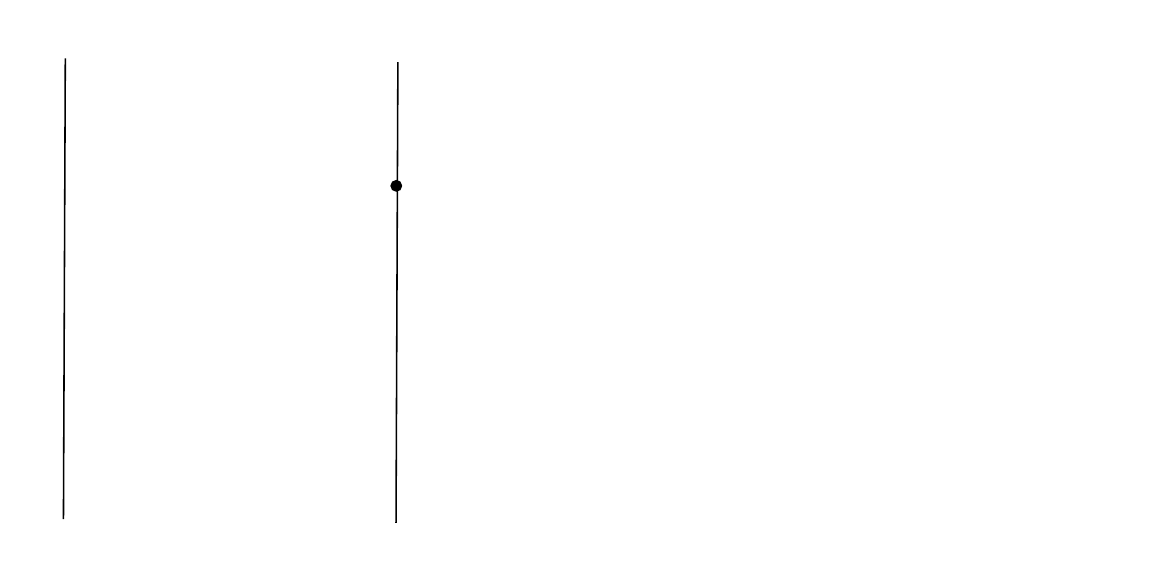		
    \caption{Curves in $(\base, J)$ corresponding to the curves of Figure \ref{figure curves contributing to graded map between modules}}
    \label{figure curves contributing to graded map between modules in the base}
\end{figure}

\begin{lemma}\label{lemma commutative diagram for fully faithfulness}
    If $m<n$ and $\gamma$ is an $\alpha$-Reeb chord in $\cO_G(\La_i^{a-\eta}(m), \La_j^{b-\eta}(n))$ of length $(k + (b-a) + (\varepsilon_n - \varepsilon_m))$, then we have the following commutative diagram in $\Gr (\cMod(\cD_G))$
    \[\begin{tikzcd}
    \cD_G(-,\La_i^a) \ar[r, "\sim" above, "t^{\varepsilon_m}" below] \ar{d}{\mu_{\cD_G}(-, t^{k+(b-a)} \otimes \pi(\gamma))} & \cB_G(\La_i^{a-\eta}(m), -) \arrow{d}{\mu_{\cB_G} (\gamma, -)} \\
    \cD_G(-,\La_j^b) \ar[r, "\sim" above, "t^{\varepsilon_n}" below] & \cB_G(\La_j^{b-\eta}(n), -).
    \end{tikzcd} \] 
\end{lemma}
\begin{proof}
    This follows from Lemmas \ref{lemma curves contributing to graded map between modules} and \ref{lemma bijection moduli spaces for map between modules}.
\end{proof}

\begin{prop}\label{prop fully faithfulness}
    If $m<n$, we have a commutative diagram of the following form (the right vertical arrow is induced by Proposition \ref{prop representability} and the bottom horizontal arrow is the Yoneda map)
    \[\begin{tikzcd}
    \hom_{\Gr(\cO_G)} \left( \La_i^{a-\eta}(m), \La_j^{b-\eta}(n) \right) \arrow{r} \arrow{d}{\sim} & \hom_{\Gr(\cMod(\cD_G))} \left( \cB_G(\La_i^{a-\eta}(m), -), \cB_G(\La_j^{b-\eta}(n), -) \right) \arrow{d}{\sim} \\
    \hom_{\Gr(\cD_G)} \left( \La_i^a, \La_j^b \right) \arrow{r}{\sim} & \hom_{\Gr(\cMod(\cD_G))} \left( \cD_G(-,\La_i^a), \cD_G(-,\La_j^b) \right).
    \end{tikzcd} \] 
    In particular, the functor $\Psi : \cO_G[M_G^{-1}] \to \cMod (\cD_G)$ is an embedding.
\end{prop}
\begin{proof}
    It follows from Lemma \ref{lemma commutative diagram for fully faithfulness} that the diagram in the Proposition is commutative if we consider the map 
    \[\hom_{\Gr(\cO_G)} \left( \La_i^{a-\eta}(m), \La_j^{b-\eta}(n) \right) \to \hom_{\Gr(\cD_G)} \left( \La_i^a, \La_j^b \right)  \]
    which sends a Reeb chord $\gamma$ of length $(k + (b-a) + (\varepsilon_n - \varepsilon_m))$ to $t^{k+(b-a)} \otimes \pi(\gamma)$ (more precisely, when $\gamma$ is degenerate, to $t^{k+(b-a)} \otimes \alpha$ where $\alpha$ is the singular cochain corresponding to the stable manifold of $\pi(\gamma)$).
    It remains to argue that the latter map is a quasi-isomorphism.
    First observe that this map is a bijection on generators.
    Then, if $\La_i \ne \La_j$, the differential on both sides is just $0$, because in this case any curve contributing to the differential strictly increases the filtration.
    Finally, if $\La_i=\La_j$, then the statement is the comparison between Morse and singular cochain complexes. 
    The result follows.
\end{proof}

\subsection{Equivariance}

Let $\vec{G} / \Z \to \Aut(\cO_G), \theta \mapsto \tau_{\theta}$ and $\vec{G} / \Z \to \Aut(\cD_G), \theta \mapsto \rho_{\theta}$ be the $(\vec{G} / \Z)$-structures described in Sections \ref{section G-structure on augmentation category} and \ref{section precompleted Fukaya category} respectively. 
Recall from Remark \ref{rmk structure on opposite category} that the latter induces a $(\vec{G} / \Z)$-structure on $\cMod(\cD_G)$, given by
\[\vec{G}/\Z \to \Aut(\cMod(\cD_G)), \quad \theta \mapsto \rho_{-\theta}^* := (- \circ \rho_{-\theta}). \] 
In the following, we identify $\cAinf$-functors from $\cO_G$ to $\cMod(\cD_G)$ and $(\cO_G^{op}, \cD_G^{op})$-bimodules.
We want to prove that the two functors 
\[\left( \Phi_0 : \vec{G}/\Z \to \Mod(\cO_G^{op}, \cD_G^{op}), \, \theta \mapsto \rho_{-\theta}^* \circ \Psi \right) \text{ and } \left( \Phi_1 : \vec{G}/\Z \to \Mod(\cO_G^{op}, \cD_G^{op}), \, \theta \mapsto \Psi \circ \tau_{\theta} \right). \] 
are equivalent in the DG-category of DG-functors from $\vec{G}/\Z$ to $\Mod(\cO_G^{op}, \cD_G^{op})$.

Given $\theta \in \vec{G} / \Z$, we need to define a morphism in $\Mod(\cO_G^{op}, \cD_G^{op})$ from $\Psi \circ \tau_{\theta}$ to $\rho_{-\theta}^* \circ \Psi$, i.e. a morphism from $\cB_G(-, \rho_{-\theta}(-))$ to $\cB_G(\tau_{\theta}(-),-)$. 
Recall that 
\[\left\{
\begin{array}{lll}
\cB_G(\La_j^{b-\eta}(n), \rho_{-\theta}(\La_i^a)) & = \cB_G(\La_j^{b-\eta}(n), \La_i^{a-\theta}) & \simeq \prod_{\gamma \in \cR(\La_i^{a-\theta}, \La_j^{b-\eta}(n))} \Z \cdot \gamma \\
\cB_G(\tau_{\theta}(\La_j^{b-\eta}(n)), \La_i^a) & = \cB_G(\La_j^{b+\theta-\eta}(n), \La_i^a) & \simeq \prod_{\gamma \in \cR(\La_i^a, \La_j^{b+\theta-\eta}(n))} \Z \cdot \gamma.
\end{array}
\right. \]
Now observe that the map 
\[SV \times \base \to SV \times \base, \quad ((t,x), q) \mapsto ((t,\varphi_{\Reeb}^{\theta}(x)), q) \]
preserves $\Gamma$, and is $(J_{\alpha} \oplus J)$-holomorphic because $\alpha = e^{H \circ \pi} \alpha^{\circ}$.
Therefore, $\varphi_{\Reeb}^{\theta}$ induces a strict morphism $f_{\theta} : \cB_G(-, \rho_{-\theta}(-)) \to \cB_G(\tau_{\theta}(-),-)$.

\begin{prop}\label{prop equivariance}
    The functor $\Psi : \cO_G[M_G^{-1}] \to \cMod (\cD_G)$ is $\vec{G} / \Z$-equivariant.
\end{prop}
\begin{proof}
    First observe that it is enough to show that $\Psi : \cO_G \to \cMod (\cD_G)$ is $\vec{G} / \Z$-equivariant. 
    Let $\theta_0, \theta_1 \in \vec{G} / \Z$ and $t^c \in \hom_{\vec{G} / \Z} (\theta_0, \theta_1)$. We want to prove that $f_{\theta_1} \circ \Phi_0(t^c)$ and $\Phi_1(t^c) \circ f_{\theta_0}$ are homotopic, i.e. that they are equivalent in $H^0 \Mod(\cO_G^{op}, \cD_G^{op})$.
    Using the same proof as \cite[Lemma A.1]{GPS2}, it is enough to prove the statement pointwise on $\cO_G \times \cD_G$. In other words, we want to prove that for every $(\La_j^{b-\eta}(n), \La_i^a) \in \cO_G \times \cD_G$, the two maps 
    \[f_{\theta_1} \circ \Phi_0(t^c), \, \Phi_1(t^c) \circ f_{\theta_0} : \cB_G(\La_j^{b-\eta}(n), \rho_{-\theta_0}(\La_i^a)) \to \cB_G(\tau_{\theta_1}(\La_j^{b-\eta}(n)),\La_i^a) \]
    are homotopic, i.e. equivalent in $H^0 \Ch$.

    Observe that $\Phi_0(t^c)$, respectively $\Phi_1(t^c)$, counts quilted strips as on the left, respectively as 
    on the right, of Figure \ref{figure equivariance} (on the latter, $\cont_c$ denotes the continuation morphism in $\cO_G(\La_j^{b-\eta+\theta_0}, \La_j^{b-\eta+\theta_1})$ induced by the Legendrian isotopy $(\La_j^{b-\eta+\theta_0+s})_{0 \leq s \leq c}$).
    When $\alpha = e^{ \delta H \circ \pi}$ is close enough to $\alpha^{\circ}$ (i.e. when $\delta$ is sufficiently small), then Morse-Bott results in symplectic field theory (see \cite{Bou03, EES09}) imply that the quilted strips of Figure \ref{figure equivariance} are in bijection with the quilted strips of Figure \ref{figure equivariance Morse-Bott} (see the proof of Lemma \ref{lemma bijection moduli spaces Morse-Bott}).
    
    Let us first examine the curves on the left of Figure \ref{figure equivariance Morse-Bott}. Since the projection $(SV, J_{\alpha^{\circ}}) \to (\base,J)$ is holomorphic, such a quilt $(w_-, u_+)$ (of area $d$) gives a pseudo-holomorphic disk in $(\base,J)$ with one free marked point on the boundary. The only possibility for this disk to be rigid is that it is constant. This implies that both $w_-$ and $u_+$ are constant and, in particular, that $d=0$. Therefore we get
    \[\Phi_0(t^c) : \cB_G(\La_j^{b-\eta}(n), \rho_{-\theta_0}(\La_i^a)) \to \cB_G(\La_j^{b-\eta}(n), \rho_{-\theta_1}(\La_i^a)), \quad t^{c'} \otimes x \mapsto t^{c'+c} \otimes x.\]
    Now let us look at the curves on the right of Figure \ref{figure equivariance Morse-Bott}. Such a curve consists of a quilted strip $(w_-, u_+)$ of area $d$, together with a gradient flow line from $\gamma$ (the Reeb chord to which $w_-$ converges) to $\gamma_c$, where $\gamma_c$ is the Reeb chord of length $c$ above the minimum of $(H \circ \pi)_{|\La_i}$. Because $\gamma_c$ is a Reeb chord above the minimum of the Morse function, the same argument as above implies that $u_+$ is constant, but this time $w_-$ consists of a trivial cylinder on $\gamma$. In particular we have $d = \ell(\gamma) = \ell(\gamma_c)= c$.
    Therefore we get 
    \[\Phi_1(t^c) : \cB_G(\tau_{\theta_0}(\La_j^{b-\eta}(n)),\La_i^a) \to \cB_G(\tau_{\theta_1}(\La_j^{b-\eta}(n)),\La_i^a), \quad t^{c'} \otimes x \mapsto t^{c'+c} \otimes x. \]
    This implies that the two maps 
    \[f_{\theta_1} \circ \Phi_0(t^c), \, \Phi_1(t^c) \circ f_{\theta_0} : \cB_G(\La_j^{b-\eta}(n), \rho_{-\theta_0}(\La_i^a)) \to \cB_G(\tau_{\theta_1}(\La_j^{b-\eta}(n)),\La_i^a) \]
    are homotopic.  
\end{proof}

\begin{figure}
    \centering
    \def\svgwidth{1\textwidth}
    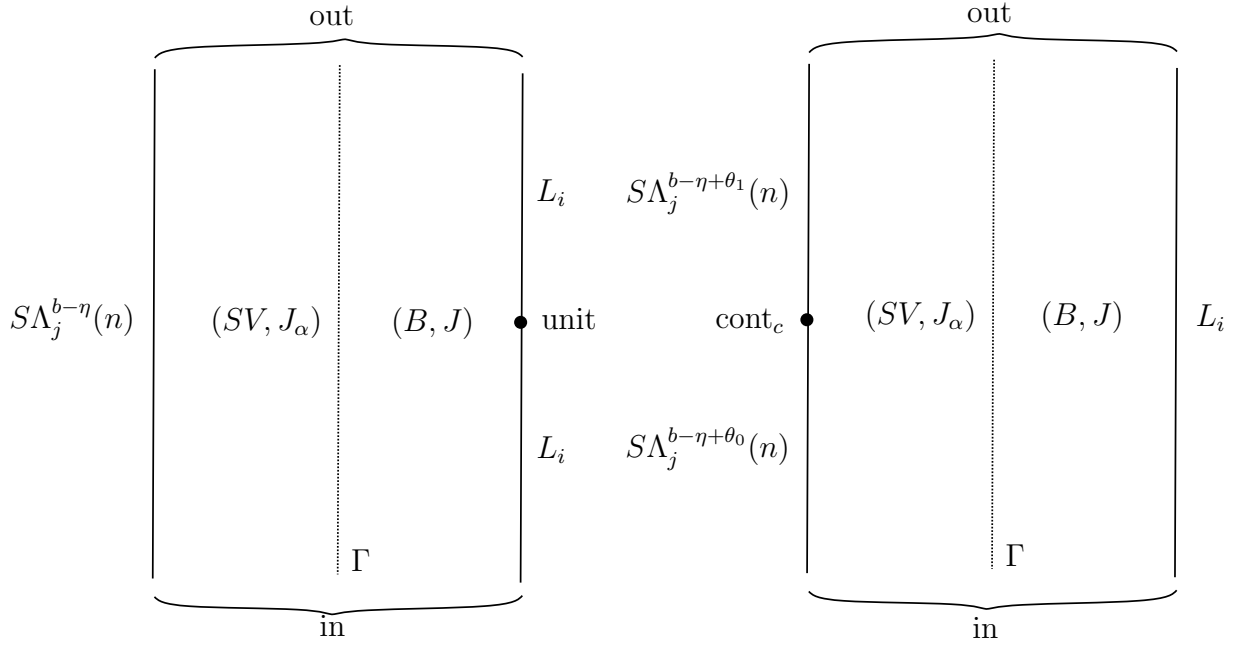		
    \caption{On the left: quilted strip $(w_-, u_+)$ contributing to $\langle \Phi_0(t^c)(\mathrm{in}), t^{c+d} \otimes \mathrm{out} \rangle$, where $d = \int w_-^* \Omega + \int u_+^* \omega$. 
    On the right: quilted strip $(w_-, u_+)$ contributing to the coefficient $\langle \Phi_1(t^c)(\mathrm{in}), t^d \otimes \mathrm{out} \rangle$, where $d = \int w_-^* \Omega + \int u_+^* \omega$.}	
    \label{figure equivariance}
\end{figure}

\begin{figure}
    \centering
    \def\svgwidth{1\textwidth}
    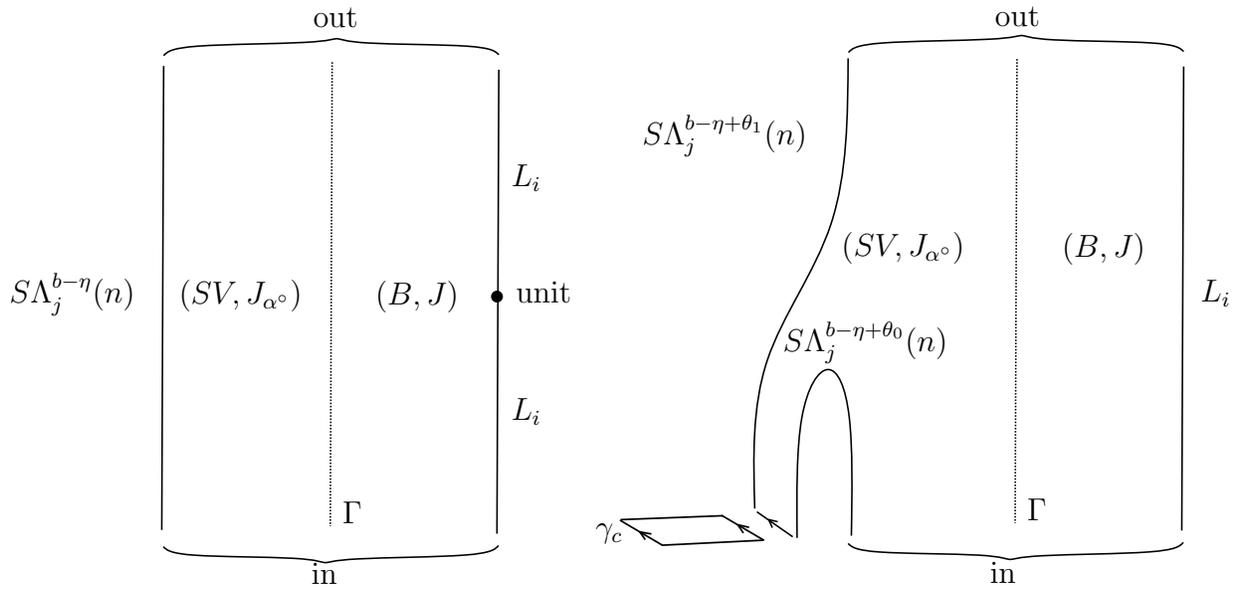		
    \caption{Morse-Bott degeneration of the quilted strips in Figure \ref{figure equivariance}.}	
    \label{figure equivariance Morse-Bott}
\end{figure}

\subsection{Main results}

We are now ready to state and prove the main result of this Section.

\begin{thm}\label{thm cFuk=cAug}
    Assume that $\La$ is an admissible Legendrian in an admissible prequantization bundle $V \xrightarrow{\pi} \base$ (see Definition \ref{definition admissible}). 
    There are $\vec G$-equivariant $\cAinf$-equivalences 
    \[\cLFuk(S(G \cdot \La)) \xrightarrow{\sim} \cFuk^{\circ}(\pi_{|\La})_{| \vec G / \Z} \# (\vec{G} / \Z) \]
    for every discrete subgroup $G$ of $\R$ containing $\Z$.
    Moreover, these are compatible with the functors 
    \[\cLFuk(S(G \cdot \La)) \hookrightarrow \cLFuk(S(G' \cdot \La)), \quad \cFuk^{\circ}(\pi_{|\La})_{| \vec G / \Z} \# (\vec{G} / \Z) \hookrightarrow \cFuk^{\circ}(\pi_{|\La})_{| G' / \Z} \# (\vec{G'} / \Z) \]
    when $G \subset G'$ (the leftmost arrow is the one of Proposition \ref{prop stop removal map for Aug}, while the rightmost arrow is the one of Lemma \ref{lemma change of enrichment}). 
\end{thm}
\begin{proof}
    Recall from subsection \ref{section functor from lagrangian correspondence} that we denoted
    \[\cLFuk(S(G \cdot \La)) = \cO_G[M_G^{-1}] \text{ and } \cFuk^{\circ}(\pi_{|\La})_{| \vec G / \Z} \# (\vec{G} / \Z) = \cD_G, \]
    and we introduced a functor 
    \[\Psi : \cO_G[M_G^{-1}] \to \cMod (\cD_G), \quad \La_j^{b-\eta}(n) \mapsto \, _{M_G^{-1}} \cB_G \left( \La_j^{b-\eta}(n), - \right) \]
    in Definition \ref{definition main functor}.
    The fact that $\Psi$ defines a $\vec G$-equivariant $\cAinf$-equivalence follows from Propositions \ref{prop representability}, \ref{prop fully faithfulness} and \ref{prop equivariance}.
    The commutativity of the diagram follows from the fact that both vertical arrows are just inclusions at the level of objects and morphisms.
    This concludes the proof.
\end{proof}

\begin{corollary}\label{corollary curved version main result}
    There is a $\cAinf$-equivalence
    \[\Nov \otimes_{\Z[\R_{\ge 0}]}\overline{(\varinjlim  \cLFuk(S(G \cdot \La)))[\vec \Q/\Z]} \xrightarrow{\sim} \cFuk(\pi_{|\La}). \]
\end{corollary}
\begin{proof}
    This follows from Theorem \ref{thm cFuk=cAug} and Lemmas \ref{lemma completion of graded Fuk recovers Fuk}, \ref{lemma recovering Fuk from approximations}, \ref{lemma recover from exhaustion}.
\end{proof}

\begin{corollary}\label{corollary conjecture Akaho-Joyce}
    The result \cite[Conjecture 3.19]{AJ10} holds when the prequantization bundle is admissible (see Definition \ref{definition admissible}).
\end{corollary}
\begin{proof}
    In \cite[Section 13.5]{AJ10}, the authors consider a prequantization bundle $V \xrightarrow{\pi} \base$ and associate a $\cAinf$-category $(\cH_{\La} \otimes \Z[[t]], \mathfrak{n}_{\La})$ to any Legendrian $\La \subset V$ such that the multiple points of $\pi_{| \La}$ are double and transverse.
    The result \cite[Conjecture 3.19]{AJ10} then asserts that the corresponding category of bounding cochains is invariant under Legendrian isotopy.
    When $V \xrightarrow{\pi} \base$ and the Legendrians (at the ends of the Legendrian isotopy) are admissible, this result follows from the case $G = \Z$ of Theorem \ref{thm cFuk=cAug}, since 
    \begin{enumerate}
        \item $(\cH_{\La} \otimes \Z[[t]], \mathfrak{n}_{\La}) = \cFuk^{\circ}(\pi_{|\La})_{| \vec \Z / \Z}$ (this is a direct observation),
        \item $\cLFuk(S\La)$, and its category of bounding cochains $\LFuk(S\La)$ (see Proposition \ref{proposition boundig cochains on cAug = Aug}), are invariant under Legendrian isotopy according to Proposition \ref{prop independence on choices}.
    \end{enumerate}
\end{proof}

\begin{corollary}\label{corollary embedded case}
    In the case where $\pi_{| \La}$ is an embedding, there is a $\cAinf$-equivalence
    \[\Nov \otimes_{\Z[[t]]} \cLFuk(S\La) \simeq \cFuk(\pi_{|\La}). \]
\end{corollary}
\begin{proof}
    This follows from Theorem \ref{thm cFuk=cAug} and Proposition \ref{prop embedded case}.
\end{proof}

\subsection{Proofs of Theorem \ref{intro thm: fukaya from augmentations} and Corollary \ref{intro corollary: fukaya from microsheaves}} \label{proofs of intro theorems}

\begin{proof}[Proof of Theorem \ref{intro thm: fukaya from augmentations}]
    According to Proposition \ref{proposition boundig cochains on cAug = Aug}, taking bounding cochains on both sides of the equivalence from Corollary \ref{corollary curved version main result} gives an equivalence 
    \[\Nov \otimes_{\Z[\R_{\ge 0}]}\overline{(\varinjlim  \LFuk(S(G \cdot \La)))[\vec \Q/\Z]} \xrightarrow{\sim} \Fuk(\pi_{|\La}). \]
    The result follows since, according to Proposition \ref{prop aug computes augmentations of the algebra}, there are equivalences $\LFuk(S(G \cdot \La)) \simeq \Prop(\cA_{G \cdot \La})$.
\end{proof}

\begin{proof}[Proof of Corollary \ref{intro corollary: fukaya from microsheaves}]
    This follows directly from Theorem \ref{intro thm: fukaya from augmentations}, Proposition \ref{prop equivariance of Aug in Mod}, and \cite[Theorem 1.4]{GPS3} (and compatibility of \cite{GPS3} with respect to contact isotopy for the equivariance).
\end{proof}

\section{Sheaf preliminaries}\label{section sheaf preliminaries}

\subsection{Positively-microsupported microsheaves}
Let $\R$ be the real line with the Euclidean topology. We denote the discrete additive group of the real numbers by $\R^\delta$. Let $G$ be a subgroup of $\R^\delta$. The group $G$ acts on $\R$ by the translation. We denote the category of $G$-equivariant $\Z$-module sheaves on $\R$ by $\sh(\R/G)$.
\begin{example}
    If $G\cong \Z$, we have $\sh(\R/G)\cong \sh(S^1)$.
\end{example}

Using the canonical orientation of $\R$, we define the positive part $T^*_+\R$ and the non-positive part $T^*_{\leq 0}\R$ of $T^*\R$. We set
\begin{equation}
    \sh^+(\R/G):=\sh(\R/G)/\left\{ \cE\in \sh(\R/G)| ss(\cE)\subset T^*_{\leq 0}S^1\right\}
\end{equation}
where $ss$ for $\sh(R/G)$ is defined via the forgetful functor $\sh(\R/G)\to \sh(\R)$.

Since $G$ is a subgroup of $\R^\delta$, the subset $G_{\geq 0}:=\R_{\geq 0}\cap G$ forms a monoid. We denote the corresponding polynomial ring by $\Z[\G_{\geq 0}]$. We set
\begin{equation*}
    \Nov^G:=\lim_{G\ni c\rightarrow+\infty}\Z[G_{\geq 0}]/T^c\Z[G_{\geq 0}]
\end{equation*}
where $T^c$ is the indeterminate corresponding to $c\in G$.
\begin{example}
    \begin{enumerate}
        \item When $G=\Z$, we have $\Nov\cong \Z[[T]]$ the formal power series ring. Note also that the identification $\sh(\R/\Z)\cong \sh(S^1)$ maps $1_{\mu,\Z}$ to $p_{S^1!}\Z_{[0, \infty)}$ where $p_{S^1}\colon \R\rightarrow S^1=\R/\Z$ is the projection.
        \item When $G=\R$, we set $\Nov:=\Nov^\R$, which is the universal Novikov ring (without Maslov variable $e$) of \cite{FOOO}.
    \end{enumerate}
\end{example}

 \begin{lemma}[\cite{Kuwagaki-WKB, Kuwagaki-Almost, Kuwagaki-Zhang}]
     $\sh^+(\R/G)$ is equivalent to the category of $\R/G$-graded almost $\Nov^G$-modules.
\end{lemma}
For objects $\cE_1, \cE_2$ in $\sh(\R/G)$, we can consider the convolution product defined by
\begin{equation}
    \cE_1\star \cE_2:=m_!(\cE_1\boxtimes \cE_2)
\end{equation}
where $m$ is the group operation of $\R$. The operation induces a monoidal structure on $\sh^+(\R/G)$, which is also denoted by $\star$. This is equivalent to the tensor on the almost modules.

We turn to more general exact symplectic manifolds.  In this case we have the microsheaf category $\mu sh(W)$ defined by \cite{Shende-microlocal,Nadler-Shende}.  We may generalize the above discussion by forming $\mu sh(W\times T^*\R)/\mu sh_{\leq 0}(W\times T^*\R)$ where $\mu sh_{\leq 0}(W\times T^*\R)$ is the full subcategory of $\mu sh(W\times T^*\R)$ spanned by the objects supported in the non-positive part $\left\{\tau\leq 0\right\}$.  However, because we do not know various categorical properties of $\mu sh$, we do not e.g. presently know how to construct a $\sh^+(\R)$ action on this category.\footnote{The current version of \cite{Ike-Kuwagaki-Novikov} contains some errors due to this point.}

The definition for the cotangent bundle case~\cite{Tamarkin} uses the convolution product which is defined through some six operations on sheaves. One can expect that a generalization of six operations on the level of microsheaves is given by microsheaf quantizations of Lagrangian correspondences. However, the composition of such correspondences is presently developed only under some isotropicity hypotheses \cite{Li-Nadler-Shende}.  Thus we must restrict to the subcategory $\mu sh_c(W\times T^*\R)\subset \mu sh(W\times T^*\R)$ of objects with `universally sufficiently Lagrangian supports'. We define $\mu sh_c^+(W\times T^*\R):=\mu sh_c(W\times T^*\R)/\mu sh_{c, \leq 0}(W\times T^*\R)$ where $\mu sh_{c, \leq 0}(W\times T^*\R)$ is defined in the same way as above.
\begin{lemma}[{\cite[Lemma C.1.1]{Li-Nadler-Shende}}]
    There exists an action of $\sh^+(\R)$ on $\mu sh_c^+(W\times \R)$ generalizing that on $\sh^+(M\times \R_t)$.
\end{lemma}
By taking the $G$-invariants, we conclude the following:

\begin{corollary}
    There exists an action of $\sh^+(\R/G)$ by $ \mu sh_c^{+}(W\times T^*\R/G)$.
\end{corollary}    
We denote this action by $\star$.

\subsection{Generation}
Let $\La$ be a Legendrian in $W\times T^\infty_+S^1$. For any object $\cE\in \mu sh^+(W\times T^*S^1)= \mu sh^+(W\times T^*\R/\Z)$, $ss(\cE)$ is defined via the lifting along $\mu sh_{c}(W\times T^*S^1)\to \mu sh^+_{c}(W\times T^*S^1)$ up to non-positive part. We consider the full subcategory $\mu sh^+_{\La}(W\times T^*S^1)$ (resp. $\mu sh_{\La}(W\times T^*S^1)$ ) of $\mu sh^+_{c}(W\times T^*S^1)$ (resp. $\mu sh_{c}(W\times T^*S^1)$ ) spanned by the objects satisfying $ss(\cE)\subset \R_{>0}\cdot \Lambda\cup \mathrm{Core}(W)\times T^*_{S^1}S^1$ where $\mathrm{Core}(W)$ is the core of $W$. Note that the projector functor gives the inclusion $\mu sh^+_\La(W\times T^*S^1)\subset \mu sh_\La(W\times T^*S^1)$.

\begin{lemma}\label{positive sheaf generation}
We have the following: 
    \begin{enumerate}
        \item The microstalks at $\Lambda$ in $\mu sh_{\Lambda}(W\times T^*S^1)$ are contained in $\mu sh^+_{\Lambda}(W\times T^*S^1)$.
        \item The microstalks generate $\mu sh^+_{\Lambda}(W\times T^*S^1)$.
    \end{enumerate}
\end{lemma}
\begin{proof}
Let $\cE_p$ be an object representing microstalk at $p\in \Lambda$. It is enough to show that $\cE_p\star 1_{\mu, \Z}\cong \cE_p$. For any $\cE\in \mu sh(W\times T^*S^1)$, we have $\Hom_{\mu sh(W\times T^*S^1)}(\cE_p\star 1_{\mu, \Z}, \cE)\cong\Hom_{\mu sh(W\times T^*S^1)}(\cE_p,\cH om^\star(1_{\mu, \Z},\cE))$ where $\cH om^\star$ is the adjunction introduced in \cite{Tamarkin, Kuwagaki-WKB}. Note that $\cH om^\star(1_{\mu, \Z}, \cE)$ and $\cE$ are isomorphic in the quotient $\mu sh^+(W\times T^*S^1)$, hence have the same microstalks at $p\in \Lambda$. Then $\Hom_{\mu sh(W\times T^*S^1)}(\cE_p\star 1_{\mu, \Z}, \cE)\cong \Hom_{\mu sh(W\times T^*S^1)}(\cE_p, \cE)$. By Yoneda, we conclude the first statement.

    If an object $\cE\in \mu sh_\Lambda^+$ is orthogonal to the microstalks, the support is contained in $\mathrm{Core}(W)\times T^*_{S^1}S^1$. Since such an object is in $\mu sh_{\mathrm{Core}(W)\times T^*_{S^1}S^1}(W\times T^*S^1)=\mu sh_{\mathrm{Core}(W)}(W)\otimes \Sh_{T^*_{S^1}S_1}(S^1)$, it is killed by $\star 1_{\mu,\Z}$, hence we conclude the generation.
\end{proof}

\subsection{Pog-action in sheaf theory}

\begin{construction}
    Let $G$ be a  subgroup of $\R^\delta$ containing $\Z$.  We describe 
    an action of $\vec{G}/\Z$ on $\mu sh^+_{c}(W \times T^*S^1)$ in two ways. 

    The first is as follows.  For $[g]\in G/\Z$, the convolution $T_{[g]}:=(-)\star p_{S^1!}\Z_{[g,\infty)}$ gives an autoequivalence of $\sh^+(S^1)$, since the inverse is given by $T_{[-g]}$. Also, for a morphism $(g_1\leq g_2)\in \vec G$, we have a natural transformation $T_{[g_1]}\rightarrow T_{[g_2]}$ induced by the canonical morphism $\Z_{[g,\infty)}\rightarrow \Z_{[g_2,\infty)}$. As a result, we obtain $\vec{G}/\Z\rightarrow \Aut(\mu sh^+_{c}(W\times T^*S^1))$. 

    A second description arises from considering the circle Reeb flow on $W\times S^1$. The \cite{GKS} sheaf quantization gives rise to an action on the sheaf category.  The time-$g$ flow gives $T_{[g]}$. The continuation maps give $\vec G/\Z\rightarrow \Aut(\mu sh^+_{c}(W\times T^*S^1))$ obtained in the above.
\end{construction}

Now we can consider $\mu sh^+_{c}(W\times T^*S^1)[\vec{G}/\Z]$.

\begin{lemma} \label{sheaf pog quotient}
    There is a  functor
   \begin{eqnarray*}
       \mu sh^+_{c}(W\times T^*S^1)[\vec{G}/\Z]&\rightarrow& \mu sh^+_{c}(W\times T^*S^1/(G/\Z)) =  \mu sh^+_{c}(W\times T^*\R/G)\\
       \cE &\mapsto &\bigoplus_{[g]\in G/\Z} T_{[g]} \cE
   \end{eqnarray*}
   where $\bigoplus_{[g]\in G/\Z} T_{[g]} \cE$ is equipped with an obvious $(G/\Z)$-equivariant structure. 
   
   The functor is fully faithful if $G/\Z$ is finite.  Otherwise, the functor becomes almost fully faithful after the completion of Lemma-Definition \ref{lemma-definition completion of category}.
\end{lemma}
\begin{proof}
Since the microsheaf category is not known to be cocomplete,
we first need to show the existence of $\bigoplus_{[g]\in G/\Z} T_{[g]} \cE$. It is enough to show for $\cE\in \mu sh_c(W\times T^*S^1)$. By definition, for any $p\in ss(\cE)$, one can take a relatively compact neighborhood $U$ of $p$. The union $\bigcup_{[g]\in G/\Z}T_{[g]}U$ is still relatively compact. By the microlocal Bertini theorem, there exists $r>0$ such that the restriction of $\cE$ to the ball $B_r(p)$ of radius $r$ (with respect to some fixed metric) centered at $p\in \bigcup_{[g]\in G/\Z}T_{[g]}U$ always admit a sheaf representative. Then $(T_{[g]}\cE)|_{B_r(p)}$ has a sheaf representative for any $g$. The direct sum of there representatives give $\bigoplus_{[g]\in G/\Z} T_{[g]} \cE$.

    It is enough to consider the case when $W$ is a point, since we can deduce the general result by tensoring $\mu sh(W)$. We first note that $\cE_1, \cE_2\in \sh^+(S^1)$, we have 
    \begin{equation*}
    \begin{split}
        \Hom_{\sh^+(S^1)[\vec G/\Z]}(\cE_1, \cE_2)&=\bigoplus_{[g]\in G/\Z}\Hom_{\sh^+(S^1)}(\cE_1, T_{[g]}\cE_2).
    \end{split}
    \end{equation*}
    On the other hand,
    \begin{eqnarray*}
        \Hom_{\sh^+(M\times S^1/(G/\Z))}(\bigoplus_{[g]\in G/\Z}T_{[g]}\cE_1, \bigoplus_{[g]\in G/\Z}T_{[g]}\cE_2)\cong \Hom_{\sh^+(M\times S^1)}(\cE_1, \bigoplus_{[g]\in G/\Z}T_g\cE_2).
    \end{eqnarray*}
    Of course these match if $G/\Z$ is finite.  However, in the  general case, they fail to match on account of the fact that $\cE_1$ is not a compact object. 
    
    Let us compute the both sides directly. Note that the category on the left is generated by $\{ p_{S^1!}\Z_{[a, \infty)}\}_{a\in R/\Z}$. Let us examine the hom-spaces for these objects:
      \begin{equation*}
    \begin{split}
        \Hom_{\sh^+(S^1)[\vec G/\Z]}(p_{S^1!}\Z_{[a_1, \infty)}, p_{S^1!}\Z_{[a_2, \infty)})\cong \bigoplus_{[g]\in G/\Z, g=\min\{g'|a_2-a_1+g'\geq 0,[g']=[g] \}}\Z[[T]]T^{a_2-a_1+g}
    \end{split}
    \end{equation*}
    where $T^{a_2-a_1+g+n}$ for $n\in \Z_{\geq 0}$ is induced by the canonical map $\Z_{[a_1,\infty)}\rightarrow \Z_{[a_2+g+n)}$. From the computation of the right hand side given in \cite[Lemma 2.26]{Kuwagaki-WKB}, it is clear that the morphism between hom-spaces is isomorphic on $H^0$ after the completion. On the other hand, the higher cohomology of the right hand side is almost zero as proved in \cite[Lemma 5.2]{Kuwagaki-Almost}. 
\end{proof}

Recall also the completion operation of Lemma \ref{lemma from Q to R}.  

\begin{lemma}\label{lemma from Q to R for sheaves}
    We have: $\overline{\mu sh^+_{c}(W\times T^*S^1)[\vec{\Q}/\Z]} = 
    \mu sh^+_{c}(W\times T^*S^1)[\vec{\R}/\Z]$.
\end{lemma}
\begin{proof}
    Since the completion $\overline{(\cdot)}$ is right adjoint, we have a canonical functor 
    $$\overline{\mu sh^+_{c}(W\times T^*S^1)[\vec{\Q}/\Z]} \to \mu sh^+_{c}(W\times T^*S^1)[\vec{\R}/\Z],$$ 
    which is the identity on the set of objects. In particular, this functor is essentially surjective. 

    For $\cE, \cF\in \mu sh^+_{c}(W\times T^*S^1)[\vec{\Q}/\Z]$ and $[a]\in \R/\Z$, we have
    \begin{eqnarray*}
        \overline{\Hom_{\mu sh^+_{c}(W\times T^*S^1)[\vec{\Q}/\Z]}(\cE, \cF)}_{[a]}&=\underset{\substack{b \in Q \\ a \leq b}}{\varprojlim} \, {\Hom_{\mu sh^+_{c}(W\times T^*S^1)[\vec{\Q}/\Z]}(\cE, \cF)}_{[b]}\\
        &=\underset{\substack{b \in \Q \\ a \leq b}}{\varprojlim} \, {\Hom_{\mu sh^+_{c}(W\times T^*S^1)}(\cE, T_b\cF)}\\
        &= {\Hom_{\mu sh^+_{c}(W\times T^*S^1)}(\cE, \underset{\substack{b \in \Q \\ a \leq b}}{\varprojlim} T_b\cF)}.
    \end{eqnarray*}
    On the cone of the canonical morphism $T_a\cF\to \underset{b \in \Q, \, b \geq a}{\varprojlim} T_b\cF$, $T^a$ is zero for any $a>0$. Such an object is known to be zero (e.g. \cite[Lemma 5.16]{Asano_2024}). Hence we also obtain the fully faithfulness. This completes the proof of the equivalence.
\end{proof}

Let us mention that some ideas related to the previous lemma have appeared also in \cite[Section 1.4]{kashiwara-schapira-persistent} and  \cite[Lemma 4.7]{kuo-shende-zhang}.

\section{Sheaf quantization of rational Lagrangians}

\begin{theorem} \label{intro thm: R-equivariant sheaves comparison}
    Let $W$ be a Weinstein manifold with $c_1(W)=0$ and let $\La$ be a Legendrian in the prequantization bundle $W \times (\R / \hbar \Z) \xrightarrow{\pi} W$ for some $\hbar > 0$, with image $L$.
    Assume that $\pi_{| \La}$ has double and transverse multiple points, and that the action spectrum of $\La$ is contained in $\hbar \Z \cup (\R \setminus \hbar \Q)$.
    Then there is an almost fully faithful embedding
    $$ \Fuk(\pi_{|\La}) \hookrightarrow \mu sh_{\pi^{-1}(L)}^+(W\times T^*\R/\R^\delta).$$
    The image is characterized by the finite rank microstalks.
\end{theorem} 
\begin{proof}
    Consider a Weinstein manifold $(W, \lambda)$ with $c_1(TW)=0$, and the prequantization bundle $(W \times S^1, d \theta + \lambda) \xrightarrow{\pi} (W, d \lambda)$.
    Observe that the latter is admissible, and fix an admissible Legendrian $\La$ in $V$. 

    In the category of $\frac{1}{n} \vec \Z / \Z$-categories, we get 
    \begin{align*}
        \Fuk^{\circ}(\pi_{| \La})_{| \frac{1}{n} \vec \Z / \Z} \# (\tfrac{1}{n} \vec \Z / \Z) & = \LFuk(S(\tfrac{1}{n} \Z \cdot \La)) \\
        & = \Prop \left( \WFuk^+(W\times T^*S^1, \tfrac{1}{n} \Z \cdot \La)^{op} \right) \\
        & = \mu sh_{\frac{1}{n} \Z \cdot \La}^+ (W\times T^*S^1).
    \end{align*}
    The first equality follows from Theorem \ref{thm cFuk=cAug}.
    The second equality follows from Proposition \ref{prop Koszul dual} (and Proposition \ref{prop equivariance of Aug in Mod} for equivariance), where the relevant topological simpleness condition is not satisfied but the injection $\pi_1(V)=\pi_1(W\times S^1)\hookrightarrow \pi_1(W\times T^*S^1)$ is enough because $\partial_{\infty} (W \times T^* S^1)$ is obtained by gluing $V \sqcup V$ and $(\partial_{\infty} W) \times (T^* S^1)$ along convex hypersurfaces, so no new disks appear on the RHS.
    Last equality follows from \cite[Theorem 1.4]{GPS3} and Lemma \ref{positive sheaf generation} (and compatibility of \cite{GPS3} with respect to contact isotopy for the equivariance). 

    Moreover, if $(\cC_n)_n$ is one of the family of categories that appear above and if $m$ divides $n$, then there is an embedding $\cC_m \hookrightarrow \cC_n$ 
    described as follows.
    If $\cC_n = \Fuk^{\circ}(\pi_{| \La})_{| \frac{1}{n} \vec \Z / \Z} \# (\tfrac{1}{n} \vec \Z / \Z)$, this is the natural functor of Lemma \ref{lemma change of enrichment}.
    When $\cC_n = \LFuk(S(\tfrac{1}{n} \Z \cdot \La))$, this is the functor of Proposition \ref{prop stop removal map for Aug}.
    When $\cC_n$ is the category of modules over the partially wrapped Fukaya category, this is the pullback along stop removal.
    Finally, when $\cC_n$ is the category of sheaves, this is the tautological inclusion.
    The constructions of the identifications are compatible with these embeddings.

    Using Lemma \ref{lemma recover from exhaustion}, we get an embedding of categories enriched over $(\vec \Q / \Z)$-modules
    \[\Fuk^{\circ}(\pi_{| \La})_{| \vec \Q / \Z} \hookrightarrow \mu sh_{\Q \cdot \La}^+ (W\times T^*S^1) [\vec \Q / \Z]. \]
    According to Lemmas \ref{lemma recovering Fuk from approximations} 
    and \ref{lemma from Q to R for sheaves}, we obtain
    \[\Fuk^{\circ}(\pi_{| \La}) \hookrightarrow \mu sh_{\R \cdot \La}^+ (W\times T^*S^1) [\vec \R / \Z]. \]
    Now Lemmas \ref{lemma completion of graded Fuk recovers Fuk} and \ref{sheaf pog quotient} give an almost fully faithful embedding
    \[\Fuk(\pi_{| \La}) \hookrightarrow \mu sh_{\R \cdot \La}^+(W\times T^*S^1/ (\R^\delta/\Z))=\mu sh_{\R \cdot \La}^+(W\times T^*\R/\R^\delta). \]
    This completes the proof.
\end{proof}

\begin{remark}
    Since $H^0(\Fuk(\pi|_{\Lambda}))^a=H^0(\Fuk(\pi|_{\Lambda}))$, our functor induces an actual fully faithful functor $H^0(\Fuk(\pi|_{\Lambda}))\hookrightarrow H^0(\sh_{\R \cdot \La}^+(S^1 \times M / (\R/\Z)_{disc}))^a$ by Remark~\ref{remark:Homotopyfullyfaithful}. 
\end{remark}

\begin{remark}
    Our theorem states that the Fukaya category is equivalent to a subcategory of the sheaf category as enriched categories over the almost modules. Since the hom-spaces of Fukaya categories are free, we can recover it from its base-change to the almost modules.
\end{remark}

\begin{corollary}\label{intro corollary: sheaf quantization}
    Let $L$ be a rational Lagrangian submanifold of $W$. If $L$ has a bounding cochain, there exists a sheaf quantization of $L$.
\end{corollary}
\begin{proof}
    By the construction, the image of $\pi_{| \La}$ in the intermediate wrapped Fukaya category intersects once with the linking disk. This is the microlocal rank one condition in the definition of sheaf quantization (see \cite{Ike-Kuwagaki-Novikov} for the definition of microstalk in the equivariant setup).
\end{proof}

\section{Quantum cohomology of $\C \P^1$}
\label{section quantum cohomology}
 
Consider the standard quadric surface $\base := \left\{ [z_0 : \cdots : z_3] \in \C \P^3 \mid z_0^2 + \cdots + z_3^2 = 0 \right\} \subset \C \P^3$.
The prequantization bundle over $\base$ is
\[\pi : U T S^3 \longrightarrow \base, \quad (x, y) \mapsto [x_0 +i y_0 : \cdots : x_3 + i y_3] \]
(here we view $S^3$ as the standard unit sphere in $\R^{d+1}$ with coordinates $(x_0, \dots, x_3)$).
Under the identification $\base = \C \P^1 \times \C \P^1$, the lift in $UTS^3$ of the anti-diagonal in $\C \P^1 \times \C \P^1$ is Legendrian isotopic to the conormal at a point (see \cite[Example 2.6]{DRG19}).

Take a point $a_0\in S^3$. 
Let $\Lambda$ be the cosphere Legendrian at $a_0$.
We want to compute $\Fuk(\pi_{|\La})$ using our results.
Since the minimal Chern number of $B$ is $2$, the algebraic structures we consider are cohomologically $\Z / 4 \Z$-graded. 

We first compute the endomorphism of the linking disk at a point on $\Lambda$. By \cite{GPS3}, this is the same as the endomorphism ring of the microstalk at a point on $\Lambda$ in $\Sh_{\Lambda}(S^3)$. By perturbing a little, we put $\Lambda$ in a position $S^*S^3$ such that
\begin{enumerate}
    \item The restriction $\Lambda\rightarrow S^3$ of the projection $S^*S^3\rightarrow S^3$ is embedding of $S^2$ into $S^3$.
    \item The projected image of $\Lambda$ in $S^3$ is the boundary of a small open ball neighborhood $D$ of $a_0$, and $\Lambda$ is the positive conormal of the boundary $\partial D$.
\end{enumerate}

We first prepare some hom-space computations. We set the image of $\id_{\Z_{S^3\backslash D}}$ under the adjunction isomorphism
$$
\Z \cdot \id_{\Z_{S^3\backslash D}}= \Hom(\Z_{S^3\backslash D}, \Z_{S^3\backslash D})\cong \Hom(\Z_{S^3}, \Z_{S^3\backslash D})
$$
by $r_{S^3\backslash D}$. We also set the image of $\id_{\Z_{D}}$ under the adjunction isomorphism
$$
\Z\cdot\id_{\Z_D}=\Hom(\Z_{D}, \Z_{D})\cong \Hom(\Z_{D}, \Z_{S^3})
$$
by $i_D$.

We fix an orientation of $S^3$ once and for all. Then we have the morphism $f_{S^3}\colon \Z_{S^3}\rightarrow \Z_{S^3}$ corresponding to the fundamental class. 

\begin{lemma}
   We have  $$
\Hom^i(\Z_{S^3}, \Z_D)\cong \begin{cases}
        \Z \text{ if $i=3$}\\
        0 \text{ otherwise.}
    \end{cases}
$$
and $\Hom^3(\Z_{S^3}, \Z_D)=\Z\cdot f_D$ which is the unique morphism satisfying $i_D\circ f_D=f_{S^3}$.
\end{lemma}
\begin{proof}
    This directly follows from the following exact triangle
    $$
    \Hom(\Z_{S^3}, \Z_D)\rightarrow \Hom(\Z_{S^3}, \Z_{S^3})\rightarrow \Hom(\Z_{S^3}, \Z_{S^3\backslash D})\xrightarrow{+1}
    $$.
\end{proof}

Next, we set 

\begin{lemma}\label{lem:sheafcohomologycomputation1}We have
$$
    \Hom^i(\Z_{S^3\backslash D},\Z_D)\cong \begin{cases}
        \Z \text{ if $i=1,3$}\\
        0 \text{ otherwise.}
    \end{cases}
    $$
Moreover, the precomposition of $r_{S^3\backslash D}$ induces $\Hom^3(\Z_{S^3\backslash D}, \Z_D)\xrightarrow{\cong}\Hom^3(\Z_{S^3}, \Z_D)=\Z f_D$.
We denote the preimage of $f_D$ by $g_D$. The degree 1 morphism space is spanned by $h_D$ which is defined by the exact sequence 
$$
0\rightarrow \Z_{D}\rightarrow \Z_{S^3}\rightarrow \Z_{S^3\backslash D}\rightarrow 0.
$$
\end{lemma}
\begin{proof}
    We have an exact triangle
    $$
    \Hom(\Z_{S^3\backslash D}, \Z_{D})\rightarrow\Hom(\Z_{S^3}, \Z_{{D}})_{3}\rightarrow\Hom(\Z_{D}, \Z_{D})_{0}\xrightarrow{[1]}.
    $$
    where the subscripts show the degree of non-zero dimensional spaces (all the nontrivial spaces are 1-dimensional).
    This shows the desired result.
\end{proof}

Consider the cone $E:=\Cone(\Z_{S^3\backslash D}\xrightarrow{g_D} \Z_D[3])$. We denote the defining morphism of the cone by $i_E\colon \Z_{D}[3]\rightarrow E$ and $p_E\colon E\rightarrow \Z_{S^3\backslash D}[1]$.

\begin{lemma}\label{lemma microstalk}
\begin{enumerate}
    \item     $E$ is a microstalk at a point in $\Lambda$. 
    \item $$
    \End^i(E)\cong \begin{cases}
        \Z \text{ if $i=0,-1$}\\
        0 \text{ otherwise.}
    \end{cases}
    $$
\end{enumerate}
\end{lemma}
\begin{proof}
We first give a sketch of the proof of the first claim. We first note that $E$ is fit into the following triangle:
$$
    \Z_{D}[1]\rightarrow \Cone(\Z_{S^3}\rightarrow \Z_{D}[3])\rightarrow E\xrightarrow{[1]}.
$$
Take a small open ball $D'$ centered at a point on $\partial D$. We can easily show that there is an isomorphism of exact triangles
    \[\begin{tikzcd}
     \Hom(E,F)\ar[r] \ar[d, "\cong"]&\Hom(\Cone(\Z_{S^3}\rightarrow \Z_{D}[3]), F)\ar[d,"\cong"]\ar[r] & \Hom(\Z_D[1], F)\ar[d,"\cong"]\xrightarrow{[1]}\\
    \Hom(\Z_{D'\backslash D}[1], F) \ar[r] & \Hom(\Z_{D'}[1], F) \ar[r]&\Hom(\Z_{D'\cap D}[1], F)\xrightarrow{[1]}.
    \end{tikzcd}\]
for any $F\in \Sh_{\Lambda}(S^3)$. Since $\Hom(\Z_{D'\backslash D}[1], F)$ is a microstalk functor, we get the desired result.

We next compute the endomorphisms of $E$. We first note the following triangles:
$$
\Hom(E,\Z_{S^3\backslash D})\rightarrow \Hom(\Z_{D}[3],\Z_{S^3\backslash D})_{\varnothing}\rightarrow \Hom(\Z_{S^3\backslash D},\Z_{S^3\backslash D})_{0}\xrightarrow{[1]}
$$
$$
\Hom(E,\Z_{D}[3])\rightarrow \Hom(\Z_{D}[3],\Z_{D}[3])_{0}\rightarrow \Hom(\Z_{S^3\backslash D},\Z_{D}[3])_{-2,0}\xrightarrow{[1]},
$$
where we used Lemma~\ref{lem:sheafcohomologycomputation1} on the leftmost term in the second line. Here the subscript $\varnothing$ means that the space is zero. Each triangle shows $\Hom(E,\Z_{S^3\backslash D})_{1}$ and $\Hom(E,\Z_{D}[3])_{-1}$ respectively. Hence we have
$$
\Hom(E,\Z_{S^3\backslash D})_{1}\rightarrow \Hom(E,\Z_{D}[3])_{-1}\rightarrow \Hom(E,E)\xrightarrow{[1]}.
$$
This shows $\Hom(E,E)_{-1,0}$.
\end{proof}
The space $\Hom(E,E)$ has the following basis: For degree $0$, it is spanned by $\id_E$. For degree $-1$, it is spanned by $s=i_E\circ h_D\circ p_{E}$. Hence the endoalgebra $\End(E)$ is isomorphic to $\Z[s]/s^2$ with $\deg s=-1$. 
The Koszul dual is $\bigoplus_{i=0}^1\Z[[x^2]]x^i=\Z[[x]]$ with $\deg x=2$.
\begin{lemma}\label{lemma t action on Koszul dual}
    As a $\Z[[t]]$-algebra, the Koszul dual of $\End(E)$ is $\Z[[t]][x]/(x^2-t)$ with $|x|=2$.
\end{lemma}
\begin{proof}
    We first describe the $t$-action. The wrapping-once functor is given by the integral transform $K\circ (-)$ with the Guillermou--Kashiwara--Schapira kernel $K$. For our case (i.e., the sphere case), an explicit description of the kernel $K$ is given by
    \begin{equation*}
        K=\mathrm{Cocone}(\mathrm{Cocone} (\Z_{S^3\times S^3}[3]\rightarrow \Z_{S^3\times S^3}[6])\rightarrow \Z_\Delta[5]).
    \end{equation*}
    All the morphisms are induced by the diagonal class. This is proved in \cite{Arai} based on \cite{GKS}.
    From this expression, we readily obtain the canonical isomorphism
    \begin{equation}
        K\circ E\cong \Z_{\Delta}[4]\circ E\cong E[4].
    \end{equation}
    Then the continuation map $c\colon \Z_D[2]\rightarrow K\circ \Z_D[2]$ gives a morphism 
    \begin{equation*}
        \Hom(E,\Z_D[2])\rightarrow \Hom(E, K\circ \Z_D[2])\cong \Hom(K^{-1}\circ E, \Z_D[2])\cong \Hom(E[-4], \Z_D[2]),
    \end{equation*}
    which is our $t$.
    We can compute the image of $\Z_D[2]$ as
    \begin{equation*}
        K\circ \Z_D[2]\cong \mathrm{Cocone}(\mathrm{Cocone} (\Z_{S^3}[2]\rightarrow \Z_{S^3}[5])\rightarrow \Z_D[7]).
    \end{equation*}
    Under this expression, the continuation map $c\colon \Z_D[2]\rightarrow K\circ \Z_D[2]$ is the unique lift of $i_D\colon \Z_D[2]\rightarrow \Z_{S^3}[2]$. Hence $c$ can be decomposed into 
    \begin{equation*}\label{eqn:continuationdecomposed}
        \Z_D[2]\xrightarrow{i_1} \mathrm{Cocone} (\Z_{S^3}[2]\xrightarrow{} \Z_{D}[5])\xrightarrow{i_2} K\circ \Z_D[2]
    \end{equation*}
    where the $i_1$ is again the unique lift of $i_D$, and $i_2$ is the unique lift of the cocone of the morphism
    \[\begin{tikzcd}
    \Z_{S^3}[2]\ar[r,"\id"]\ar[d]& \Z_{S^3}[2]\ar[d]\\
   \Z_{D}[5] \ar[r, "i_D"]& \Z_{S^3}[5].
    \end{tikzcd} \] 
    Now we apply $\Hom(E,-)$ to $i_1$, then we obtain a canonical morphism
    \begin{equation}\label{eqn:ki1}
        \Hom(E, \Z_D[2])\rightarrow \Hom(E, \mathrm{Cocone} (\Z_{S^3}[2]\xrightarrow{} \Z_{D}[5])).
    \end{equation}
    We also apply $\Hom(E,-)$ to $i_2$, then we obtain the following isomorphism between morphisms
    \[\begin{tikzcd}
    \Hom(E, \mathrm{Cocone} (\Z_{S^3}[2]\xrightarrow{} \Z_{D}[5]))\ar[r,"K\circ (i_2)"]& \Hom(E, K\circ \Z_D[2])\\
   \Hom(E, \Z_D[4])\ar[r]\ar[u, "\cong"]& \Hom(E, \mathrm{Cocone}(\Z_{S^3}[4]\rightarrow \Z_D[7]))\ar[u, "\cong"],
    \end{tikzcd} \] 
    since the components $\Z_{S^3}[2]$ are canceled out. The vertical isomorphisms are implied by the vanishing of $\Hom(E, \Z_{S^3})$. Now the lower arrow is precisely $K\circ (i_1)[2]$ from (\ref{eqn:ki1}). Moreover lower right object is canonically isomorphic to $\Hom(E, \Z_{D}[6])$ by the same vanishing as above.
    Hence what we obtain by applying $\Hom(E,-)$ to (\ref{eqn:continuationdecomposed}) is an expression of $t$ as
    \begin{equation*}
        \Hom(E, \Z_D[2])\xrightarrow{K\circ(i_1)} \Hom(E,\Z_D[4])\xrightarrow{K\circ(i_1)[2]} \Hom(E, \Z_{D}[6]).
    \end{equation*}
    Name $x=K\circ(i_1)$, we complete the proof.
\end{proof}

\begin{corollary}
    The quantum cohomology over $\Nov$ of $\C \P^1$ is given by 
    \[QH^*(\C \P^1) = \Nov \underset{\Z[[t]]}{\otimes} \Z[[t]][x]/(x^2-t). \]
\end{corollary}
\begin{proof}
    The quantum cohomology over $\Nov$ of $\C \P^1$ is the Fukaya algebra over $\Nov$ of the anti-diagonal in $\C \P^1 \times \C \P^1$, i.e. $\Fuk(\pi_{| \La})$.
    We have $\Fuk(\pi_{| \La}) = \Nov \underset{\Z[[t]]}{\otimes} \Fuk^{\circ}(\pi_{| \La})_{\vec \Z / \Z}$ according to Proposition \ref{prop embedded case}.
    Using Theorem \ref{thm cFuk=cAug}, Proposition \ref{prop Koszul dual} and \cite{GPS3}, we know that $\Fuk^{\circ}(\pi_{| \La})_{\vec \Z / \Z}$ is the Koszul dual of $\End(E)$, where $E$ is a representative of the microstalk functor at a point in $\La$. 
    Now Lemma \ref{lemma t action on Koszul dual} says that the Koszul dual of $\End(E)$ is $\Z[[t]][x]/(x^2-t)$.
    This gives the desired result.
\end{proof}

\appendix

\section{Pseudo-holomorphic discs in the symplectization of a prequantization}
\label{appendix pseudo-holomorphic discs}

Let $V \xrightarrow{\pi} \base$ be a prequantization bundle.
The goal of this Appendix is to relate pseudo-holomorphic discs in $\base$ and in $SV$. 
The necessary analytic results appeared already in \cite{ENS02} and \cite{DR16}; however, the authors of those articles were considering the case where $V = \R \times P$, with $P$ Liouville.  Here we adapt their results to the case of a circle bundle. 
Observe that this problem has also been addressed in \cite[Section 2.2]{BCSW24II}.

\vspace{2mm}

Let $J$ be an almost complex structure on $\base$ compatible with $\omega$.
If $\alpha$ is a contact form on $(V, \xi)$, we denote by $J_{\alpha}$ the almost complex structure on $SV$ such that $J_{\alpha} = \pi^* J$ on $\xi$ (observe that $\pi : V \to \base$ induces an isomorphism $\xi \simeq \pi^*T \base$) and $J_{\alpha} \partial_t = R_{\alpha}$, where we use $\alpha^{\circ}$ to identify $SV$ and $\R_t \times V$.
With these choices, the projection $(\R \times V,J_{\alpha^{\circ}}) \to (\base, J)$ is pseudo-holomorphic.

In the following, we fix a Riemann disk $(D, i)$ with $(d+1)$ cyclically ordered marked points $(\zeta_0, \zeta_1, \dots, \zeta_d)$ on $\partial D$.
We set $\Delta := D \setminus \{ \zeta_0, \dots, \zeta_d \}$, and we choose strip-like ends 
\[\epsilon_0 : \R_{\geq 0} \times [0,1] \to \Delta, \quad \epsilon_k : \R_{\leq 0} \times [0,1] \to \Delta \text{ for } k \in \{1, \dots, d \}. \]

We consider a smooth map $u : \Delta \to \base$ which extends to a continuous map from $D$ to $\base$, together with a lift $\nu : \partial \Delta \to V$ of $u_{| \partial \Delta}$ such that the following limits exist for $m \in \{0, 1\}$:
\[\nu_{0,m} := \underset{s \to +\infty}{\lim} (\nu \circ \epsilon_0) (s, m), \quad \nu_{k,m} := \underset{s \to -\infty}{\lim} (\nu \circ \epsilon_k) (s, m) \text{ for } k \in \{1, \dots, d \}. \]
Finally, we fix $\ell_0, \ell_1, \dots, \ell_d \in \R_{\geq 0}$ such that 
\[\varphi_{\Reeb}^{\ell_0} (\nu_{0,0}) = \nu_{0,1}, \quad \varphi_{\Reeb}^{\ell_k} (\nu_{k,0}) = \nu_{k,1} \text{ for } k \in \{1, \dots, d \}. \]

\begin{lemma}\label{lemma relation on boundary lift}
    There exists an integer $K$ such that $\ell_0 -\sum_{k=1}^d \ell_k = \int_{\Delta} u^* \omega - \int_{\partial \Delta} \nu^* \alpha^{\circ} + K$.
\end{lemma}
\begin{proof}
    Consider a trivialization $(u^*V, u^* \nabla_{\! \alpha^{\circ}}) \simeq (\Delta \times S^1, \nabla_{\! A})$, where $A = p_{S^1}^* d \theta + p_{\Delta}^* a$ for some primitive $a$ of $u^* \omega$.
    The lift $\nu : \partial \Delta \to V$ of $u_{| \partial \Delta}$ induces, via the latter trivialization, a section $f_0$ of $p_{\Delta}$ over $\partial \Delta$.
    Choose a lift $h_0 : \partial \Delta \to \R$ of $p_{S_1} \circ f_0$, and set 
    \[l_{0,m} = \underset{s \to +\infty}{\lim} (h_0 \circ \epsilon_0) (s, m), \quad l_{k,m} = \underset{s \to -\infty}{\lim} (h_0 \circ \epsilon_k) (s, m) \text{ for } k \in \{1, \dots, d \}. \]
    Since $\varphi_{\Reeb}^{\ell_0} (\nu_{0,0}) = \nu_{0,1}$ and $\varphi_{\Reeb}^{\ell_k} (\nu_{k,0}) = \nu_{k,1}$ for $k \in \{1, \dots, d \}$, we know that $\ell_k - (l_{k,1} - l_{k,0})$ is an integer for every $k \in \{0, \dots, d \}$.
    Therefore, there is an integer $K$ such that 
    \[\deg (p_{S^1} \circ f_0) = \sum_{k=1}^d  (l_{k,1} - l_{k,0}) - (l_{0,1} - l_{0,0}) = \sum_{k=1}^d \ell_k - \ell_0 + K. \]
    The result follows since moreover
    \[\deg (p_{S^1} \circ f_0) = \int_{\partial \Delta} f_0^* p_{S^1}^* d \theta = \int_{\partial \Delta} f_0^* A - \int_{ \partial \Delta} a = \int_{\partial \Delta} \nu^* \alpha^{\circ} - \int_{\Delta} u^* \omega. \]
\end{proof}

\begin{lemma}\label{lemma holomorphic condition}
    Let $v : \Delta \to V$ be a lift of $u$.
    Then a map $w = (\sigma, v) : \Delta \to \R \times V$ is $J_{\alpha^{\circ}}$-holomorphic if and only if $u$ is $J$-holomorphic and $d \sigma = (v^* \alpha^{\circ}) \circ i$.
\end{lemma}
\begin{proof}
    This is a straightforward verification using that the projection $(\R \times V,J_{\alpha^{\circ}}) \to (\base, J)$ is pseudo-holomorphic.
\end{proof}

In the following, we assume that the derivatives of $u \circ \epsilon_0$, $\nu \circ \epsilon_0$, $u \circ \epsilon_k$ and $\nu \circ \epsilon_k$ for $k \in \{1, \dots, d\}$, decay exponentially (for some choice of metrics on $\base$ and $V$).

\begin{prop}\label{prop relation pseudo-holomorphic discs}
    Assume that $\nu$ takes values in a Legendrian $\La \subset V$.
    There exists a pseudo-holomorphic map $w = (\sigma, v) : (\Delta, i) \to (\R \times V, J_{\alpha^{\circ}})$ such that $\pi \circ v = u$, $v_{| \partial \Delta} = \nu$ and satisfying the asymptotic conditions
    \[\left\{
    \begin{array}{ll}
    (v \circ \epsilon_0) (s, t) \underset{s \to +\infty}{\longrightarrow} \varphi_{\Reeb}^{\ell_0 t} (\nu_{0,0}), & (\sigma \circ \epsilon_0) (s, t) \underset{s \to +\infty}{\sim} \ell_0 s \\
    (v \circ \epsilon_k) (s, t) \underset{s \to -\infty}{\longrightarrow} \varphi_{\Reeb}^{\ell_k t} (\nu_{k,0}), & (\sigma \circ \epsilon_k) (s, t) \underset{s \to -\infty}{\sim} \ell_k s \text{ for } k \in \{1, \dots, d \}
    \end{array}
    \right.\]
    if and only if $\ell_0 -\sum_{k=1}^d \ell_k = \int_{\Delta} u^* \omega$.
    In this case, such a map $w$ is unique up to translation in the symplectization coordinate.
\end{prop}
\begin{proof}
    If such a map $w$ exists, then the relation follows from a standard application of Stokes theorem.
    Now assume that the relation holds.
    According to Lemma \ref{lemma holomorphic condition}, we have to find a section $v$ of $u^*V \to \Delta$ such that 
    \[\left\{
    \begin{array}{rll}
    d ( ((u^* \nabla_{\! \alpha^{\circ}}) v) \circ i ) & = 0 & \text{on } \Delta \\
    v & = \nu & \text{on } \partial \Delta
    \end{array}
    \right. \]
    where $\nabla_{\! \alpha^{\circ}} v = v^* \alpha^{\circ}$.
    Consider a trivialization $(u^*V, u^* \nabla_{\! \alpha^{\circ}}) \simeq (\Delta \times S^1, \nabla_{\! A})$ with $A = p_{S^1}^* d \theta + p_{\Delta}^* a$ for some primitive $a$ of $u^* \omega$.
    The lift $\nu : \partial \Delta \to V$ of $u_{| \partial \Delta}$ induces, via the latter trivialization, a section $f_0$ of $p_{\Delta}$ over $\partial \Delta$.
    Since $\ell_0 -\sum_{k=1}^d \ell_k = \int_{\Delta} u^* \omega = \deg(p_{S_1} \circ f_0)$, we can choose a lift $h_0 : \partial \Delta \to \R$ of $p_{S_1} \circ f_0$ so that, if we set
    \[l_{0,m} = \underset{s \to +\infty}{\lim} (h_0 \circ \epsilon_0) (s, m), \quad l_{k,m} = \underset{s \to -\infty}{\lim} (h_0 \circ \epsilon_k) (s, m) \text{ for } k \in \{1, \dots, d \}, \] 
    then $l_{0,1} -  l_{0,0} = \ell_0$ and $l_{k,1} -  l_{k,0} = \ell_k$ for $k \in \{1, \dots, d \}$.
    Let $h : \Delta \to \R$ be the unique solution to the following Poisson equation with Dirichlet boundary conditions
    \[\left\{
    \begin{array}{rll}
    d ( (dh + a) \circ i ) & = 0 & \text{on } \Delta \\
    h & = h_0 & \text{on } \partial \Delta
    \end{array}
    \right. \]
    and let $g$ be a primitive of $(dh + a) \circ i$.
    The analysis done in the proof of \cite[Theorem 7.7]{ENS02} or \cite[Lemma 7.1]{DR16} shows that 
    \[\left\{
    \begin{array}{ll}
    (h \circ \epsilon_0) (s, t) \underset{s \to +\infty}{\longrightarrow} l_{0,0} + \ell_0 t, & (g \circ \epsilon_0) (s, t) \underset{s \to +\infty}{\sim} \ell_0 s \\
    (h \circ \epsilon_k) (s, t) \underset{s \to -\infty}{\longrightarrow} l_{k,0} + \ell_k t, & (g \circ \epsilon_k) (s, t) \underset{s \to -\infty}{\sim} \ell_k s \text{ for } k \in \{1, \dots, d \}
    \end{array}
    \right.\]
    The map $v :\Delta \to V$ is then obtained, via the trivialization $(u^*V, u^* \nabla_{\! \alpha^{\circ}}) \simeq (\Delta \times S^1, \nabla_{\! A})$, from the section $f$ of $p_{\Delta}$ satisfying $(p_{S^1} \circ f)(z) = [h(z)]$, and the map $\sigma : \Delta \to \R$ is a primitive of $(\nabla_{\! \alpha^{\circ}} v) \circ i$.
\end{proof}

In order to state the next result, we denote by $\cS$ the interval $(\zeta_d, \zeta_0)$ in $\partial \Delta$.

\begin{prop}\label{prop relation pseudo-holomorphic discs mixed case}
    There exists a unique pseudo-holomorphic map $w = (\sigma, v) : (\Delta, i) \to (\R \times V, J_{\alpha^{\circ}})$ such that $\pi \circ v = u$, $v_{| \partial \Delta \setminus \cS} = \nu_{| \partial \Delta \setminus \cS}$, $\sigma_{| \cS} = 0$, and satisfying the asymptotic conditions 
    \[\left\{
    \begin{array}{ll}
    (v \circ \epsilon_0) (s, t) \underset{s \to +\infty}{\longrightarrow} \nu_{0,1}, & (\sigma \circ \epsilon_0) (s, t) \underset{s \to +\infty}{\rightarrow} 0 \\
    (v \circ \epsilon_d) (s, t) \underset{s \to -\infty}{\longrightarrow} \nu_{d,1}, & (\sigma \circ \epsilon_0) (s, t) \underset{s \to -\infty}{\rightarrow} 0 \\
    (v \circ \epsilon_k) (s, t) \underset{s \to -\infty}{\longrightarrow} \varphi_{\Reeb}^{\ell_k t} (\nu_{k,0}), & (\sigma \circ \epsilon_k) (s, t) \underset{s \to -\infty}{\sim} \ell_k s \text{ for } k \in \{1, \dots, d-1 \}
    \end{array}
    \right.\]
\end{prop}
\begin{proof}
    According to Lemma \ref{lemma holomorphic condition}, we need to find a section $v$ of $u^*V \to \Delta$ such that 
    \[\left\{
    \begin{array}{rll}
    d ( (\nabla_{\! \alpha^{\circ}} v) \circ i ) & = 0 & \text{on } \Delta \\
    v & = \nu & \text{on } \partial \Delta \setminus \cS \\
    (\nabla_{\! \alpha^{\circ}} v) (n_{\cS}) & = 0 & \text{on } \cS
    \end{array}
    \right. \]
    where $\nabla_{\! \alpha^{\circ}} v = v^* \alpha^{\circ}$ and $n_{\cS}$ is the outward normal vector to $\cS$.
    As above, consider a trivialization $(u^*V, u^* \nabla_{\! \alpha^{\circ}}) \simeq (\Delta \times S^1, \nabla_{\! A})$ with $A = p_{S^1}^* d \theta + p_{\Delta}^* a$ for some primitive $a$ of $u^* \omega$.
    The lift $\nu_{|\partial \Delta \setminus \cS}$ of $u_{| \partial \Delta \setminus \cS}$ induces, via the latter trivialization, a section $f_0$ of $p_{\Delta}$ over $\partial \Delta \setminus \cS$.
    Let $h_0 : \partial \Delta \setminus \cS \to \R$ be a lift of $p_{S^1} \circ f_0$ so that, if we set $l_{k,m} := \underset{s \to -\infty}{\lim} (h_0 \circ \epsilon_k) (s, m)$ for $k \in \{1, \dots, d-1 \}$, then $l_{k,1} -  l_{k,0} = \ell_k$.
    Let $h : \Delta \to \R$ be the unique solution to the following Poisson equation with mixed boundary conditions (of Dirichlet type on $\partial \Delta \setminus \cS$ and of Neumann type on $\cS$)
    \[\left\{
    \begin{array}{rll}
    d ( (dh + a) \circ i ) & = 0 & \text{on } \Delta \\
    h & = h_0 & \text{on } \partial \Delta \setminus \cS \\
    (dh + a) (n_{\cS}) & = 0 & \text{on } \cS
    \end{array}
    \right. \]
    and let $g$ be a primitive of $(dh + a) \circ i$.
    Analysis of this equation shows that
    \[\left\{
    \begin{array}{ll}
    \underset{s \to + \infty}{\lim} (h \circ \epsilon_0) (s, t) =  \underset{s \to + \infty}{\lim} (h_0 \circ \epsilon_0) (s, 1), & \underset{s \to +\infty}{\lim} (g \circ \epsilon_0) (s, t) = 0 \\
    \underset{s \to - \infty}{\lim} (h \circ \epsilon_d) (s, t) =  \underset{s \to - \infty}{\lim} (h_0 \circ \epsilon_d) (s, 1), & \underset{s \to - \infty}{\lim} (g \circ \epsilon_d) (s, t) = 0 \\
    (h \circ \epsilon_k) (s, t) \underset{s \to -\infty}{\longrightarrow} l_{k,0} + \ell_k t, & (g \circ \epsilon_k) (s, t) \underset{s \to -\infty}{\sim} \ell_k s \text{ for } k \in \{1, \dots, d-1 \}
    \end{array}
    \right.\]
    The map $v :\Delta \to V$ is then obtained, via the trivialization $(u^*V, u^* \nabla_{\! \alpha^{\circ}}) \simeq (\Delta \times S^1, \nabla_{\! A})$, from the section $f$ of $p_{\Delta}$ satisfying $(p_{S^1} \circ f)(z) = [h(z)]$, and the map $\sigma : \Delta \to \R$ is a primitive of $(\nabla_{\! \alpha^{\circ}} v) \circ i$.
\end{proof}

\bibliographystyle{plain}
\bibliography{bibli.bib}

\end{document}